\documentclass[a4paper,reqno,10pt]{amsart}

\usepackage{epsf,graphicx,epsfig}
\usepackage{amscd}
\usepackage{bbm}
\usepackage{amsmath,latexsym,amssymb,amsthm}
\usepackage[nospace,noadjust]{cite}
\usepackage{textcomp}
\usepackage{dsfont}
\usepackage{fdsymbol}
\usepackage{setspace,cite}
\usepackage{lscape,fancyhdr,fancybox}
\usepackage{stmaryrd}
\usepackage[all,cmtip]{xy}
\usepackage{tikz}
\usepackage{cancel}
\usetikzlibrary{shapes,arrows,decorations.markings}
\usepackage{graphicx}

\usepackage{color}
\usepackage{url}
\usepackage{enumerate}
\usepackage[mathscr]{euscript}
  \theoremstyle{plain}

\swapnumbers
    \newtheorem{thm}{Theorem}[section]
    \newtheorem{prop}[thm]{Proposition}
     
   \newtheorem{lemma}[thm]{Lemma}

    \newtheorem{subsec}[thm]{}
\theoremstyle{definition}
    \newtheorem{defn}[thm]{Definition}
        \newtheorem{remark}[thm]{Remark}
    \newtheorem{exam}[thm]{Example}

\theoremstyle{remark}

\title{}
\author{}
\date{}
\usepackage{amssymb}

\usepackage{hyperref}
\hypersetup{
	colorlinks,
	citecolor=blue,
	filecolor=black,
	linkcolor=blue,
	urlcolor=black
}

\begin{document}

\title[Averaging operators on groups, racks and Leibniz algebras]{Averaging operators on groups, racks and Leibniz algebras}

\author{Apurba Das}
\address{Department of Mathematics,
Indian Institute of Technology, Kharagpur 721302, West Bengal, India.}
\email{apurbadas348@gmail.com, apurbadas348@maths.iitkgp.ac.in}

%\author{Ramkrishna Mandal}
%\address{Department of Mathematics, Indian Institute of Technology, Kharagpur 721302, West Bengal, India.}
%\email{ramkrishnamandal430@gmail.com}

\begin{abstract}
This paper considers averaging operators on various algebraic structures and studies the induced structures.
We first introduce the notion of an averaging operator on a group $G$ and show that it induces a rack structure. Moreover, the given group structure and the induced rack structure form a group-rack. We observe that any pointed group-rack can be embedded into an averaging group. We show that the differentiation of a smooth pointed averaging operator on a Lie group gives rise to an averaging operator on the corresponding Lie algebra. Next, we consider averaging operators on a rack that induces a hierarchy of new rack structures. Moreover, any two racks with increasing hierarchy levels form a rack-pairing, a structure that is related to two-sided skew braces by conjugation. We also consider averaging operators on cocommutative Hopf algebras and braided vector spaces in relations to averaging operators on groups, Lie algebras and racks. In the end, we define averaging operators on a Leibniz algebra, find the induced structure and show that the differentiation of a smooth pointed averaging operator on a pointed Lie rack yields an averaging operator on the corresponding Leibniz algebra.
\end{abstract}

\maketitle

%\curraddr{}
%\email{}

%\subjclass[2010]{}
%\keywords{}

\medskip

\medskip

\medskip



%Averaging algebras of nonzero weight, Triassociative algebras, $L_\infty$-algebras, Cohomology, Homotopy algebras.

%\medskip

%\noindent {\sf Date of resubmission:} July 26, 2021.

\thispagestyle{empty}

\tableofcontents

%\vspace{0.2cm}

\medskip

\section{Introduction}

\subsection{Averaging algebras} The concept of averaging operators has a long history as it appears in many branches of mathematics and fluid dynamics. Averaging operators first implicitly appeared in the study of turbulence theory in a paper by Reynolds that was published in 1895 \cite{rey}. In the mathematical study of turbulence theory, such an operator appears as the time average operator of real-valued functions defined in a time-space. Subsequently, averaging operators were further investigated in functional analysis and probability theory \cite{kampe,kelley,miller,moy}. On the algebraic side, Brainerd \cite{brain} considered the conditions under which an averaging operator can be realized as a generalization of the integral operator on the ring of real-valued measurable functions. However, averaging operators became popular among algebraists when Cao \cite{cao} wrote his Ph.D. thesis in 2000. Recall that an averaging operator on an associative algebra $A$ is a linear map $P: A \rightarrow A$ that satisfies
\begin{align*}
    P(a) P(b) = P (P(a) b) = P (a P(b)), \text{ for } a, b \in A.
\end{align*}
An algebra $A$ equipped with a distinguished averaging operator is called an averaging algebra. Cao mainly constructed the free commutative averaging algebra and described the induced
Lie and Leibniz algebras. His study was generalized by Pei and Guo \cite{pei-guo} who obtained free associative averaging algebras using a class of bracketed words, called averaging words, and discovered their relations with Schr\"{o}der numbers.
Averaging operators were also defined on Lie algebras \cite{aguiar}. They are called embedding tensors in the physics literature and related to Leibniz algebras, tensor hierarchies and higher gauge theories \cite{kotov-strobl,bon}. See also \cite{sheng-tang-zhu,das,das-sen} for some cohomological study of averaging algebras. The general study of averaging operators on any binary quadratic operads and their relations with bisuccessors, duplicators and Rota-Baxter operators are systematically studied in \cite{bai-b,pei-b}

\subsection{Rota-Baxter operators and averaging operators on groups} While Rota-Baxter operators and averaging operators are well understood on associative, Lie and other algebras, can one define these operators on a Lie group so that their differentiation yields same type of operators on the corresponding Lie algebra? In the seminal work \cite{guo-adv}, Guo, Lang and Sheng introduced and studied Rota-Baxter operators of weight $1$ on groups. Among others, they showed that the differentiation of a smooth Rota-Baxter operator of weight $1$ on a Lie group yields a Rota-Baxter operator of weight $1$ on the corresponding Lie algebra. Their study got much attention when Bardakov and Gubarev \cite{bardakov} showed that a Rota-Baxter group gives rise to a skew brace, a structure introduced by Guarnieri and Vendramin \cite{guar} in connections with the set-theoretical solutions of the Yang-Baxter equation. Subsequently, Goncharov \cite{gon} defined Rota-Baxter operators on a cocommutative Hopf algebra $H$ and showed that there is a one-to-one correspondence between Rota-Baxter operators of weight $1$ on a group $G$ and Rota-Baxter operators on $H= {\bf k}[G]$. Recently, Rota-Baxter operators of weight $0$ on a group $G$ with respect to a linear representation were considered and their relations with braces, Yang-Baxter equation were found in \cite{li-sheng-tang}.

Our primary aim in this paper is to introduce and study averaging operators on (both abstract and Lie) groups. Given a group $G$, a set-map $\mathcal{A} : G \rightarrow G$ is said to be an {averaging operator} on $G$ if it satisfies
\begin{align}\label{avg-ident}
    \mathcal{A} (g) \mathcal{A} (h) \mathcal{A}(g)^{-1} = \mathcal{A} ( \mathcal{A} (g) h \mathcal{A}(g)^{-1}  ), \text{ for all } g, h \in G.
\end{align}
A group $G$ equipped with an averaging operator is called an averaging group. We give some characterizations of an averaging operator. We show that an averaging operator $\mathcal{A} : G \rightarrow G$ on a group $G$ induces a rack structure (called the descendent rack) $g \diamond_\mathcal{A} g :=  \mathcal{A} (g) h \mathcal{A}(g)^{-1}$ on the underlying set $G$. When $\mathcal{A}$ is the identity map on $G$, the induced rack is simply the conjugation rack structure. Given an averaging operator, the given group structure on $G$ and the induced rack structure form a new structure which we call a group-rack. We show that any pointed group-rack (a group-rack whose underlying rack is pointed with distinguished element $e$, the identity element of the group) can be embedded into an averaging group. Next, we consider a smooth pointed (i.e. $\mathcal{A}(e) = e$) averaging operator on a Lie group $G$ and show that its differentiation $\mathcal{A}_{*e}: \mathfrak{g} \rightarrow \mathfrak{g}$ gives rise to an averaging operator on the corresponding Lie algebra $\mathfrak{g}$. Finally, we also define relative averaging operators on a group $G$ (with respect to both an action and a linear representation). Finally, when we differentiate smooth relative averaging operators on a Lie group with respect to the adjoint representation $\mathfrak{g}$, again we obtain averaging operators on the Lie algebra $\mathfrak{g}$.

\subsection{Averaging operators on racks and rack-pairings} An averaging operator on a group is described by the identity (\ref{avg-ident}) which can be understood as $\mathcal{A} (g) \diamond \mathcal{A} (h) = \mathcal{A} (\mathcal{A} (g) \diamond h)$ in terms of the conjugation rack operation on $G$. This guided us to consider averaging operators on an arbitrary rack $(Q, \diamond)$. We observe that an averaging operator on a rack $(Q, \diamond)$ induces a new rack structure $x \diamond_\mathcal{A} y := \mathcal{A} (x) \diamond y$, for $x, y \in Q$, on the underlying set $Q$. Moreover, the induced operation $\diamond_\mathcal{A}$ is left distributive over $\diamond$. Hence $(Q, \diamond, \diamond_\mathcal{A})$ forms a new structure which we call a rack-pairing. A rack-pairing is weaker than the earlier notion of a multi-rack labelled by a set of two elements. We show that an averaging operator $\mathcal{A}$ on a rack $(Q, \diamond)$ yields a hierarchy of new rack structures $\diamond_{\mathcal{A}^k}$, for any $k \geq 0$. Further, $(Q, \diamond_{\mathcal{A}^k}, \diamond_{\mathcal{A}^{k+l}})$ is a rack pairing, for any $k, l \geq 0$. Like racks can be obtained from a group by conjugation, we observe that a rack-pairing is obtained from a two-sided skew brace by conjugation. This says that rack-pairing is the analogue of skew braces in the context of rack theory. By keeping this in mind, we generalize some characterizations of skew braces in our context. Among others, we show that a rack-pairing structure on a given rack is described by the regular subrack of the holomorph rack. In addition, we also consider  $\lambda$-homomorphic and symmetric rack-pairings.

In \cite{bai} the authors have considered the notion of a post-group in connection with Rota-Baxter groups and skew braces. They showed that the category of skew braces and the category of post-groups are isomorphic. Moreover, post-Lie groups are the global objects corresponding to post-Lie algebras. Motivated by their study, we consider the notion of a di-rack. We observe that a rack-pairing gives rise to a di-rack and vice-versa. Moreover, the correspondences are invertible and functorial which shows that the category of rack-pairings and the category of di-racks are isomorphic.

While finishing our paper, we came to know that Bardakov and Bovdi have considered averaging operators on racks in their recent preprint \cite{bardakov-bovdi}. They mainly focused on a particular result which states that an averaging operator on a rack induces an averaging operator on the corresponding rack algebra. However, our theme is broad and the whole study is completely different from theirs.

\subsection{Averaging operators on cocommutative Hopf algebras and braided vector spaces} Any group $G$ gives rise to the group algebra ${\bf k}[G]$ which is infact a cocommutative Hopf algebra. Motivated by the work of Goncharov \cite{gon}, we introduce averaging operators on a cocommutative Hopf algebra $H$ and show that there is a one-to-one correspondence between averaging operators on a group $G$ and averaging operators on the cocommutative Hopf algebra $H= {\bf k}[G]$.

On the other hand, it is well-known that Lie algebras, groups and racks are intimately related to braided vector spaces and the set-theoretical solutions of the Yang-Baxter equation \cite{baez-crans,eiser}. By keeping this in mind, we set the definition of an averaging operator on braided vector spaces and braided sets. We show that an averaging operator on a Lie algebra yields an averaging operator on the corresponding braided vector space. Similarly, an averaging operator on a rack can be realized as an averaging operator on the corresponding braided set.

\subsection{Averaging operators on Leibniz algebras} Leibniz algebras are noncommutative analogues of Lie algebras and can be obtained from pointed Lie racks by differentiation \cite{kinyon}. We introduce averaging operators on a Leibniz algebra and show that an averaging operator induces a new structure which we call a di-Leibniz algebra. We provide some examples of di-Leibniz algebras and show that any di-Leibniz algebra can be embedded into an averaging Leibniz algebra.

\medskip

Next, we focus on some Lie-theoretic results regarding averaging operators and induced structures. Let $G$ be a Lie group and $\mathcal{A}: G \rightarrow G$ be a smooth pointed averaging operator on $G$. Then the descendent rack is a pointed Lie rack and hence gives a Leibniz algebra structure on $\mathfrak{g} = T_e G$ by differentiation. On the other hand, the averaging operator $\mathcal{A} : G \rightarrow G$ induces an averaging operator $\mathcal{A}_{* e} : \mathfrak{g} \rightarrow \mathfrak{g}$ on the corresponding Lie algebra. Hence it yields a Leibniz algebra structure on $\mathfrak{g}$. Then the above two Leibniz algebras coincide. We also introduce the notion of a Lie-Leibniz algebra and show that a pointed Lie group-rack gives rise to a Lie-Leibniz algebra by differentiation. Finally, we consider smooth pointed averaging operators on a pointed Lie rack and show that their differentiation yields left averaging operators on the corresponding Leibniz algebra.

\subsection{Organization of the paper} The paper is organized as follows. In Section \ref{sec2}, we introduce and study averaging operators on a group. In particular, we show that the differentiation of a smooth pointed averaging operator on a Lie group gives an averaging operator on the corresponding Lie algebra. In Sections \ref{sec3} and \ref{sec4}, we study averaging operators on racks and induced rack-pairings. We consider averaging operators on braided vector spaces and braided sets in Section \ref{sec5}. The notion of averaging operators on Leibniz algebras and the concept of di-Leibniz algebras are introduced in Section \ref{sec6}. Various Lie-theoretic results regarding averaging operators are considered in \ref{sec7}. Finally, we discuss some generalizations and future works in Section \ref{sec8}.

\section{Averaging operators on groups}\label{sec2}

In this section, we introduce averaging operators on a group. We show that an averaging operator on a group $G$ induces a rack. Moreover, each left translation of the rack operation is a group automorphism. This compatibility yields the structure of a group-rack. We show that any pointed group-rack can be embedded into an averaging group. Finally, we show that the differentiation of a smooth pointed averaging operator on a Lie group gives rise to an averaging operator (embedding tensor) on the corresponding Lie algebra.

\medskip

Let $G$ be a group with the identity element $e$. The adjoint action of $G$ onto itself is the group homomorphism $\mathrm{Ad} : G \rightarrow \mathrm{Aut}(G), g \mapsto \mathrm{Ad}_g$ defined by
\begin{align*}
    \mathrm{Ad}_g (h) = g h g^{-1}, \text{ for } h \in G.
\end{align*}

\begin{defn}
(i) Let $G$ be a group. An {\bf averaging operator} on $G$ is a map $\mathcal{A} : G \rightarrow G$ that satisfies
\begin{align}
    \mathcal{A} (g) \mathcal{A}(h) \mathcal{A}(g)^{-1} = \mathcal{A} ( \mathcal{A} (g) h \mathcal{A}(g)^{-1}), \text{ for all } g, h \in G.
\end{align}
An averaging operator $\mathcal{A}: G \rightarrow G$ is said to be {\bf pointed} if $\mathcal{A} (e) = e$.

(ii) An {\bf averaging group} is a group $G$ equipped with an averaging operator $\mathcal{A} : G \rightarrow G$. We often denote an averaging group as above by the pair $(G, \mathcal{A})$.
\end{defn}

%\begin{exam}
 %   Let $G$ be a group. Then any constant map $\mathcal{A} : G \rightarrow G$ is an averaging operator on $G$.
%\end{exam}

\begin{exam}
Let $G$ be a group. Then any constant map, the identity map $\mathrm{id}: G \rightarrow G$ and the $n$-th power map $(-)^n : G \rightarrow G, g \mapsto g^n$ (for $n \in \mathbb{Z}$) are all averaging operators on $G$. In particular, the inverse map $( - )^{-1}: G \rightarrow G$ is an averaging operator.
\end{exam}

\begin{exam}
Let $G$ be a group and let $Z (G)$ be the centre of $G$. Then any map $\mathcal{A} : G \rightarrow Z(G) \subset G$ is an averaging operator on $G$. In particular, if $G$ is abelian then any map $\mathcal{A} : G \rightarrow G$ is an averaging operator on $G$. 
%It is further pointed if $\mathcal{A}(e) = e.$
\end{exam}

\begin{exam}
    Let $\mathcal{A} : G \rightarrow G$ be an averaging operator on $G$. Then the map $\widetilde{ \mathcal{A}} : G \rightarrow G$, $\widetilde{\mathcal{A}} (g) := \mathcal{A} (g^{-1})$ is an averaging operator on $G$.
    %As a consequence, the inverse map $(- )^{-1} : G \rightarrow G, x \mapsto x^{-1}$ is an averaging operator on $G$.
\end{exam}

\begin{exam}
    Let $G$ be a group and $\mathcal{A}: G \rightarrow G$ be an averaging operator on $G$. If $x \in G$ is an element that commutes with the images of $\mathcal{A}$, then the map $\widetilde{\mathcal{A}} : G \rightarrow G$ defined by $\widetilde{\mathcal{A}} (g) := x \mathcal{A} (g) = \mathcal{A} (g) x$ is an averaging operator on $G$.
\end{exam}

\begin{exam}
    Let $\mathcal{A}, \mathcal{B} : G \rightarrow G$ be two averaging operators on $G$ that commute with each other. Then the composition $\mathcal{A} \circ \mathcal{B}$ is an averaging operator on $G$.
\end{exam}

\begin{exam}
    Let $G$ be a group and $\mathcal{A}: G \rightarrow G$ be a pointed averaging operator on $G$ that is also a group homomorphism. Let $N \subset G$ be a normal subgroup such that $\mathcal{A}(N) \subset N$. Then $\mathcal{A}$ induces an averaging operator $\widetilde{\mathcal{A}} : G / N \rightarrow G/ N$ given by
    \begin{align*}
        \widetilde{\mathcal{A}} ( g N) =   \mathcal{A}(g) N, \text{ for } g N \in G/ N.
    \end{align*}
\end{exam}

\begin{exam}
Let $G, H$ be two groups and $\Phi: G \rightarrow \mathrm{Aut}(H)$ be a group homomorphism. Consider the semidirect product group $G \ltimes H$ with the operation
\begin{align*}
    (g, h) \cdot (g', h') = (gg' , h \Phi_g h'), \text{for }(g, h), (g', h') \in G \times H.
\end{align*}
If $\mathcal{A} : G \rightarrow G$ is an averaging operator on $G$, then the map $\widetilde{\mathcal{A}} : G \ltimes H \rightarrow G \ltimes H$, $\widetilde{\mathcal{A}} (g, h) = ( \mathcal{A} (g), e_H)$ is an averaging operator on the semidirect product $G \ltimes H$.
\end{exam}

\begin{exam}
Let $H, K$ be two normal subgroups of $G$ such that $G = H K$ and $H \cap K = \{ e \}$. For any averaging operator $\mathcal{A} : H \rightarrow H$, the map $\widetilde{\mathcal{A}} : G \rightarrow G$ defined by $\widetilde{\mathcal{A}} (hk) := \mathcal{A} (h) k$ is an averaging operator on $G$.
\end{exam}

%\begin{exam}
 %   Let $G$ be any group. Then for any $n \geq 0$, the map $\mathcal{A} : G \rightarrow G, x \mapsto x^n$ is an averaging operator on $G$.
%\end{exam}

\begin{prop}
    Let $\mathcal{A}: G \rightarrow G$ be an averaging operator on $G$. Then for each $n \in \mathbb{N}$, the map $\mathcal{A}^n: G \rightarrow G$ is an averaging operator.
\end{prop}
\begin{proof}
    Suppose $\mathcal{A}^r$ is an averaging operator, for some $r \in \mathbb{N}$. Then for any $g, h \in G$, we have
    \begin{align*}
        \mathcal{A}^{r+1} (g)  \mathcal{A}^{r+1} (h)  \mathcal{A}^{r+1} (g)^{-1} 
        &=  \mathcal{A}^{r} ( \mathcal{A} (g)) \mathcal{A}^{r} ( \mathcal{A} (h)) \mathcal{A}^{r} ( \mathcal{A} (g))^{-1}  \\
        &= \mathcal{A}^r \big(   \mathcal{A}^{r+1} (g)  \mathcal{A} (h)  \mathcal{A}^{r+1} (g)^{-1}  \big) \\
        &= \mathcal{A}^r \big(   \mathcal{A} (   \mathcal{A}^r (g)) \mathcal{A} (h)  \mathcal{A} (   \mathcal{A}^r (g))^{-1}   \big) \\
        &= \mathcal{A}^{r+1} \big(  \mathcal{A}^{r+1} (g) h \mathcal{A}^{r+1} (g)^{-1} \big).
    \end{align*}
    This shows that $\mathcal{A}^{r+1}$ is an averaging operator on $G$. Hence the result follows by the mathematical induction.
\end{proof}

%**********************************************************************

In the following, we give some characterizations of an averaging operator on a group. These characterizations are based on rack structures. 

\begin{defn}
A {\bf (left) rack} $(Q, \diamond)$ is a nonempty set $Q$ equipped with a binary operation $\diamond : Q \times Q \rightarrow Q$ that satisfy the following axioms:

- (left-distributivity)  for all $x, y, z \in Q$, $x \diamond (y \diamond z) = (x \diamond y) \diamond (x \diamond z)$,

- for any $x, y \in Q$, there exists a unique $c \in Q$ such that $x \diamond c = y$.

%- for any $x \in Q$, we have $\mathds{1} \diamond x = x$ and $x \diamond \mathds{1} = \mathds{1}$.
\end{defn}

A rack $(Q, \diamond)$ is said to be {\bf pointed} if there exists a distinguished element $e \in Q$ such that $x \diamond e  = e$ and $e \diamond x = x$, for all $x \in Q$. A pointed rack is denoted by the triple $(Q, \diamond, e)$.

\begin{remark}
\begin{itemize}
\item[(i)] Let $Q$ be a nonempty set. Define a binary operation $\diamond: Q \times Q \rightarrow Q$ by $x \diamond y:= y$, for all $x, y \in Q$. Then $(Q, \diamond)$ is a rack, called the trivial rack.

\item[(ii)] Let $(Q, \diamond)$ be a rack. Then for each $x \in Q$, the left translation $L_x := x \diamond -$ is a rack automorphism (i.e. $L_x : Q \rightarrow Q$ is bijective and $L_x (y \diamond z) = L_x (y) \diamond L_x (z)$, for all $y, z \in Q$). Moreover, for any $x, y \in Q$, we have $L_x L_y L_x^{-1} = L_{x \diamond y}$.

\item[(iii)] Racks can be realized as a generalization of the group conjugation. Explicitly, if $G$ is a group with the identity element $e \in G$, then the operation $g \diamond h := ghg^{-1}$ makes $(G, \diamond, e)$ into a pointed rack. This (pointed) rack is said to be induced from the group structure.
\end{itemize}
\end{remark}

Let $(Q, \diamond)$ be a rack. A {\bf subrack} of $(Q, \diamond)$ is a subset $S \subset Q$ such that $x \diamond y \in S$ whenever $x, y \in S$. Let $(Q, \diamond)$ and $(Q', \diamond')$ be two racks. A map $\varphi : Q \rightarrow Q'$ is said to be a {\bf rack homomorphism} if $\varphi (x \diamond y) = \varphi (x) \diamond' \varphi(y)$, for all $x, y \in Q.$

\begin{prop}\label{cong-rack}
Let $G$ be a group. Then $(G \times G, \diamond)$ is a rack, where the rack operation $\diamond$ is given by
\begin{align}\label{ad-rack}
    (g_1, g) \diamond (h_1, h) = (g_1 h_1 g_1^{-1}, g_1 h g_1^{-1}), \text{ for } (g_1, g), (h_1, h) \in G \times G.
\end{align}
We denote this rack simply by $G \times_\mathrm{Ad} G.$
\end{prop}

\begin{proof}
    For $(g_1, g), (h_1, h)$ and $(k_1, k) \in G \times G,$ we have
    \begin{align*}
        (g_1, g) \diamond \big(  (h_1, h) \diamond (k_1, k)  \big) =~& (g_1, g) \diamond (h_1 k_1 h_1^{-1}, h_1 k h_1^{-1}) \\
        =~& (g_1 h_1 k_1 h_1^{-1} g_1^{-1}, g_1 h_1 k h_1^{-1} g_1^{-1}) \\
        =~& \big(  g_1 h_1 g_1^{-1} g_1 k_1 g_1^{-1} (g_1 h_1 g_1^{-1})^{-1} ,  g_1 h_1 g_1^{-1} g_1 k g_1^{-1} (g_1 h_1 g_1^{-1})^{-1} \big) \\
        =~& (g_1 h_1 g_1^{-1}, g_1 h g_1^{-1}) \diamond (g_1 k_1 g_1^{-1}, g_1 k g_1^{-1}) \\
        =~& \big(  (g_1, g) \diamond (h_1, h)  \big) \diamond \big( (g_1, g) \diamond (k_1, k)   \big).
    \end{align*}
    This shows that the left distributivity holds. Moreover, for any $(g_1, g) , (h_1, h) \in G \times G$, we have
    \begin{align*}
        (g_1, g) \diamond (g_1^{-1} h_1 g_1, g_1^{-1} h  g_1) = (h_1, h).
    \end{align*}
    Therefore, the equation $(g_1, g) \diamond c = (h_1, h)$ has a unique solution. %Finally, $(e,e) \diamond (g_1, g) = (g_1, g)$ and $(g_1, g) \diamond (e,e) = (e,e)$.
    This shows that $(G \times G, \diamond)$ is a rack.
\end{proof}

\begin{prop}\label{prop-graph-group}
    Let $G$ be a group. A set map $\mathcal{A} : G \rightarrow G$ is an averaging operator on the group $G$ if and only if the graph $Gr(\mathcal{A}) = \{  ( \mathcal{A}(g), g ) |~ g \in G \}$ is a subrack of $G \times_\mathrm{Ad} G.$
\end{prop}

\begin{proof}
    Let $g, h \in G$. Then we have
    \begin{align*}
        ( \mathcal{A} (g) , g) \diamond ( \mathcal{A} (h) , h) = (  \mathcal{A} (g) \mathcal{A} (h) \mathcal{A}(g)^{-1} , \mathcal{A}(g) h \mathcal{A}(g)^{-1}).
    \end{align*}
    This element is in the graph $Gr (\mathcal{A})$ if and only if $ \mathcal{A} (x) \mathcal{A} (y) \mathcal{A}(x)^{-1} = \mathcal{A} (  \mathcal{A}(x) y \mathcal{A}(x)^{-1})$, for all $x, y \in G$. In other words, $\mathcal{A}$ is an averaging operator on the group $G$.
\end{proof}

Let $G$ be a group and $\mathcal{A} : G \rightarrow G$ be any map. First, consider the rack $G \times_\mathrm{Ad} G$ given in Proposition \ref{cong-rack}. We define a map 
\begin{align*}
\xi_\mathcal{A} : G \times G \rightarrow G \times_\mathrm{Ad} G  ~~~ \text{ by } ~~~ \xi_\mathcal{A} (g, h) = (\mathcal{A} (h) g, h), \text{ for } (g, h) \in G \times G.
\end{align*}
The map $\xi_\mathcal{A}$ is obviously invertible with the inverse
    $\xi_\mathcal{A}^{-1} (g, h) = (  \mathcal{A} (h)^{-1} g, h ),$ for $(g, h) \in G \times_\mathrm{Ad} G.$
Hence by using the map $\xi_\mathcal{A}$ and its inverse $\xi_\mathcal{A}^{-1}$, one can transfer the rack structure of $G \times_\mathrm{Ad} G$ to the set $G \times G$. More precisely, the rack operation is given by
\begin{align*}
    (g_1, g) * (h_1, h) =~& \xi_\mathcal{A}^{-1} \big(  \xi_\mathcal{A} (g_1, g) \diamond \xi_\mathcal{A} (h_1, h)  \big) \\
    =~& \xi_\mathcal{A}^{-1} \big(  (  \mathcal{A}(g) g_1, g)  \diamond (  \mathcal{A}(h) h_1, h)    \big) \\
    =~& \xi_\mathcal{A}^{-1} \big(   \mathcal{A}(g) g_1  \mathcal{A}(h) h_1 g_1^{-1} \mathcal{A} (g)^{-1},  \mathcal{A}(g) g_1 h g_1^{-1} \mathcal{A} (g)^{-1}          \big) \\
    =~& \big(  \mathcal{A} \big( \mathcal{A}(g) g_1 h g_1^{-1} \mathcal{A} (g)^{-1} \big)^{-1}  \mathcal{A}(g) g_1  \mathcal{A}(h) h_1 g_1^{-1} \mathcal{A} (g)^{-1},  \mathcal{A}(g) g_1 h g_1^{-1} \mathcal{A} (g)^{-1} \big).
\end{align*}
We denote this rack simply by the notation $G \times_\mathrm{Ad}^\mathcal{A} G$. Then we have the following.

\begin{prop}
    Let $G$ be a group with the identity element $e$. Let $\mathcal{A}: G \rightarrow G$ be a map. Then $\mathcal{A}$ is an averaging operator on $G$ if and only if $\{ e \} \times G$ is a subrack of $G \times_\mathrm{Ad}^\mathcal{A} G$.
\end{prop}

\begin{proof}
Let $(e, g), (e, h) \in \{ e \} \times G$ be arbitrary. Then we have
\begin{align*}
    (e, g) * (e, g) = \bigg(   \big(  \mathcal{A} (   \mathcal{A} (g) h \mathcal{A}(g)^{-1}) \big)^{-1} \mathcal{A}(g) \mathcal{A}(h) \mathcal{A}(g)^{-1}, \mathcal{A}(g) h \mathcal{A}(g)^{-1}   \bigg).
\end{align*}
This is in $\{ e \} \times G$ if and only if $ \big(  \mathcal{A} (   \mathcal{A} (g) h \mathcal{A}(g)^{-1}) \big)^{-1} \mathcal{A}(g) \mathcal{A}(h) \mathcal{A}(g)^{-1} = e$. This is equivalent to the fact that $\mathcal{A}$ is an averaging operator on the group $G$.
\end{proof}

Next, we show that averaging operators on groups (i.e. averaging groups) are more intimately related to racks. 

\begin{prop}\label{prop-des}
    Let $\mathcal{A}: G \rightarrow G$ be an averaging operator on a group $G$.

    (i) Then $(G, \diamond_\mathcal{A})$ is a rack (called the descendent rack), where
    \begin{align}\label{indu-rack}
        g \diamond_\mathcal{A} h := \mathcal{A} (g) h \mathcal{A}(g)^{-1}, \text{ for } g,h \in G.
    \end{align}

If $\mathcal{A}$ is a pointed averaging operator then the descendent rack $(G, \diamond_\mathcal{A})$ is also pointed.

    (ii) The map $\mathcal{A}: G \rightarrow G$ is a homomorphism of racks from the descendent rack to the rack structure on $G$ induced from the group structure.
\end{prop}

\begin{proof}
    (i) It follows from Proposition \ref{prop-graph-group} that $Gr (\mathcal{A})$ is a subrack of $G \times_\mathrm{Ad} G$. Since $G$ is isomorphic to $Gr (\mathcal{A})$ via the correspondence $g \longleftrightarrow (\mathcal{A}(g), g)$, the above rack structure on $Gr (\mathcal{A})$ induces a rack structure on $G$. This rack structure on $G$ is precisely the one given in (\ref{indu-rack}).

    (ii) For any $g, h \in G$, we have
    \begin{align*}
        \mathcal{A} (g \diamond_\mathcal{A} h) = \mathcal{A} (   \mathcal{A}(g) h \mathcal{A}(g)^{-1}) = \mathcal{A}(g) \mathcal{A}(h) \mathcal{A}(g)^{-1} = \mathcal{A}(g) \diamond \mathcal{A} (h). 
    \end{align*}
    Hence the result follows.
\end{proof}

Let $G$ be a group. Then $G$ inherits its conjugation rack structure. The group structure and rack structure on $G$ satisfy a nice compatibility. More generally, if we have an averaging operator on a group, then the group structure and the induced rack structure satisfy the same compatibility. Motivated by this, we define the following.

\begin{defn}
    A {\bf group-rack} is a triple $(G, \cdot, \smallblackdiamond)$ in which $(G, \cdot)$ is a group, $(G, \smallblackdiamond)$ is a rack such that for each $x \in G$, the map $L_x^\smallblackdiamond = x \smallblackdiamond - : G \rightarrow G$ is a group homomorphism, i.e.
    \begin{align*}
         x \smallblackdiamond yz = (x \smallblackdiamond y) (x \smallblackdiamond z), \text{ for } x, y, z  \in G.
    \end{align*}
    Here we use the notation $x y$ for the product $x \cdot y$.
\end{defn}

\begin{remark}
    Let $(G, \cdot, \smallblackdiamond)$ be a group-rack in which the rack $(G, \smallblackdiamond)$ is pointed with distinguished element $e$ (the identity element of the group). Such a group-rack is called {\bf pointed}. Then for any $x , y \in G$, we observe that $(x 
 \smallblackdiamond y) (x \smallblackdiamond y^{-1}) = x \smallblackdiamond e = e$. This shows that $(x \smallblackdiamond y)^{-1} = x \smallblackdiamond y^{-1}$.
\end{remark}

\begin{prop}
    Let $(G, \cdot)$ be a group and $\mathcal{A} : G \rightarrow G$ be an averaging operator on $G$. Then $(G, \cdot, \diamond_\mathcal{A})$ is a group-rack. Further, if $\mathcal{A}$ is a pointed averaging operator then $(G, \cdot, \diamond_\mathcal{A})$ is a pointed group-rack.
\end{prop}

\begin{proof}
    We have already seen that $\diamond_\mathcal{A}$ defines a rack structure on $G$. For any $x, y, z \in G$, we also have 
    \begin{align*}
        x \diamond_\mathcal{A} yz = \mathcal{A} (x) yz \mathcal{A}(x)^{-1} = \mathcal{A} (x) y \mathcal{A} (x)^{-1} \mathcal{A} (x) z \mathcal{A} (x)^{-1} = (x \diamond_\mathcal{A} y) (x \diamond_\mathcal{A} z).
    \end{align*}
    Hence the first part follows. The second part is easy.
\end{proof}

Thus it follows that an averaging group gives rise to a group-rack. The following result gives a partial converse.

\begin{thm}
    Every pointed group-rack can be embedded into an averaging group.
\end{thm}

\begin{proof}
    Let $(G, \cdot, \smallblackdiamond)$ be a pointed group-rack. Consider the subgroup $\mathrm{Inn} (G, \smallblackdiamond)$ of the automorphism group $\mathrm{Aut}(G, \smallblackdiamond)$ generated by the set $\{ L_x^\smallblackdiamond = x \smallblackdiamond - | ~ x \in Q \}$. Since each $L_x^\smallblackdiamond$ is a group homomorphism of $(G, \cdot)$, we have the inclusion map
    \begin{align*}
        \Psi : \mathrm{Inn}(G, \smallblackdiamond) \hookrightarrow \mathrm{Aut} (G, \cdot), ~ \Psi (L_x^\smallblackdiamond) = L_x^\smallblackdiamond.
    \end{align*}
    The map $\Psi$ is a group homomorphism. That is the group $\mathrm{Inn}(G, \smallblackdiamond)$ acts on the group $(G, \cdot)$ by automorphisms. Therefore, the cartesian product $\mathrm{Inn} (G, \smallblackdiamond) \times G$ carries a group structure with the semidirect product composition
    \begin{align*}
        (L_x^\smallblackdiamond , y) (L_{x'}^\smallblackdiamond, y') = (L_x^\smallblackdiamond L^\smallblackdiamond_{x'}, y L_x^\smallblackdiamond (y') ), \text{ for } (L_x^\smallblackdiamond, y), (L_{x'}^\smallblackdiamond, y') \in \mathrm{Inn} (G, \smallblackdiamond) \times G.
    \end{align*}
    We define a map $\mathcal{A} : \mathrm{Inn} (G, \smallblackdiamond) \times G \rightarrow \mathrm{Inn} (G, \smallblackdiamond) \times G$ by $\mathcal{A} (L_x^\smallblackdiamond, y) = (L_y^\smallblackdiamond, e)$. Then it is easy to verify that $\mathcal{A}$ is an averaging operator on the group $\mathrm{Inn} (G, \smallblackdiamond) \times G$. In other words, $(\mathrm{Inn} (G, \smallblackdiamond) \times G, \mathcal{A})$ is an averaging group. Finally, it is straightforward to check that the inclusion map $i : G \hookrightarrow \mathrm{Inn} (G, \smallblackdiamond) \times G$, $x \mapsto (\mathrm{Id}, x)$ is a morphism of group-racks, where the group-rack structure on $\mathrm{Inn} (G, \smallblackdiamond) \times G$ is induced by the averaging operator $\mathcal{A}$. This completes the proof.
\end{proof}

Next, we recall averaging operators on a Lie algebra \cite{aguiar} and showed that the differentiation of a smooth pointed averaging operator on a Lie group gives an averaging operator of the corresponding Lie algebra.

\begin{defn}
   (i) Let $\mathfrak{g} = (\mathfrak{g}, [~,~])$ be a Lie algebra. An {\bf averaging operator} on the Lie algebra $\mathfrak{g}$ is a linear map $P : \mathfrak{g} \rightarrow \mathfrak{g}$ that satisfies
    \begin{align}\label{embed-ten}
        [P(x) , P(y)] = P[P(x), y], \text{ for all } x, y \in \mathfrak{g}.
    \end{align}

    (ii)  An {\bf averaging Lie algebra} is a Lie algebra $\mathfrak{g}$ equipped with a distinguished averaging operator $P: \mathfrak{g} \rightarrow \mathfrak{g}$. We denote an averaging Lie algebra as above by the pair $(\mathfrak{g}, P).$
\end{defn}
Note that the condition (\ref{embed-ten}) is equivalent to $[P(x) , P(y)] = P[x, P(y)]$, for all $x, y \in \mathfrak{g}$. Indeed we have
\begin{align*}
    [P(x), P(y)] = - [P(y), P(x)] \stackrel{(\ref{embed-ten})}{=} - P [P(y), x] = P [x, P(y)].
\end{align*}

In literature, averaging operators on a Lie algebra are also known as embedding tensors \cite{kotov-strobl,bon,sheng-tang-zhu}. They are related to averaging (associative) algebras by skew-symmetrization. Namely, an {\em averaging algebra} is a pair $(A, P)$ consisting of an associative algebra $A$ with a linear map $P : A \rightarrow A$ that satisfies
\begin{align}\label{avg-ass}
    P(a) P(b) = P (P(a) b) = P (a P(b)), \text{ for } a, b \in A.
\end{align}
The identity (\ref{avg-ass}) is called the averaging identity and the map $P$ is called an averaging operator on the algebra $A$. If $(A, P)$ is an averaging algebra then $(A_c, P)$ is an averaging Lie algebra, where $A_c$ is the vector space $A$ endowed with the commutator Lie bracket.

%The following result shows that an averaging Lie algebra always induces a Leibniz algebra structure.

%\begin{prop}
%    Let $(\mathfrak{g}, A)$ be an averaging Lie algebra. Then the vector space $\mathfrak{g}$ inherits a Leibniz algebra structure with the bracket
%    \begin{align*}
 %       \{ x, y \} := [A(x), y], \text{ for } x, y \in \mathfrak{g}.
%    \end{align*}
%    This is called the descendent Leibniz algebra of the averaging Lie algebra $(\mathfrak{g}, A).$
%\end{prop}

\begin{thm}\label{firr-thm}
    Let $G$ be a Lie group and $\mathcal{A}: G \rightarrow G$ be a smooth pointed averaging operator on $G$. If $\mathfrak{g}$ is the Lie algebra of the Lie group $G$, then $P = \mathcal{A}_{* e}: \mathfrak{g} \rightarrow \mathfrak{g}$ is an averaging operator on the Lie algebra $\mathfrak{g}$. In other words, $(\mathfrak{g}, P)$ is an averaging Lie algebra.
\end{thm}

\begin{proof}
Let $\mathrm{exp}: \mathfrak{g} \rightarrow G$ be the exponential map. Then the Lie bracket of $\mathfrak{g}$ is given by
\begin{align*}
    [x, y] = \frac{d^2}{dt ds}\big|_{t, s = 0} \mathrm{exp}^{tx} \mathrm{exp}^{sy} \mathrm{exp}^{-tx}, \text{ for } x,y \in \mathfrak{g}.
\end{align*}
Hence for any $x, y \in \mathfrak{g}$, we have
    \begin{align*}
        [P(x), P(y)] =~& \frac{d^2}{dt ds} \bigg|_{t, s = 0} \mathrm{exp}^{t P(x)} ~ \mathrm{exp}^{s P(y)} ~ \mathrm{exp}^{-t P(x)} \\
        =~& \frac{d^2}{dt ds} \bigg|_{t, s = 0} \mathcal{A} ( \mathrm{exp}^{t x} ) \mathcal{A} (  \mathrm{exp}^{s y}  ) \mathcal{A} ( \mathrm{exp}^{tx} )^{-1} ~~~ (\text{as } P = \mathcal{A}_{* e}) \\
        =~& \frac{d^2}{dt ds} \bigg|_{t, s = 0} \mathcal{A} \big(  \mathcal{A} ( \mathrm{exp}^{t x} ) ~ \mathrm{exp}^{s y} ~ \mathcal{A} ( \mathrm{exp}^{tx} )^{-1}  \big) \\
        =~& \mathcal{A}_{* e} \big(  \frac{d^2}{dt ds} \bigg|_{t, s = 0} \mathrm{Ad}_{ \mathcal{A} (\mathrm{exp}^{tx})} \mathrm{exp}^{sy}  \big) = P ([P(x) , y]).
    \end{align*}
    This shows that $P$ is an averaging operator on the Lie algebra $\mathfrak{g}$.
\end{proof}

In the following, we generalize averaging operators in the presence of a group action. We observe that some of the above results still hold in the general context. Let $G$ be a group with the identity element $e$. Suppose that $G$ acts on a set $X$ from the left. That is, there is a map (called the action map) $\Phi: G \times X \rightarrow X$, $(g, x) \mapsto \Phi_g x$ satisfying
\begin{align*}
    \Phi_e (x) = x ~~~~ \text{ and } ~~~~ \Phi_{gh} (x) = \Phi_g (\Phi_h (x)), \text{ for all } g, h \in G, x \in X.
\end{align*}
Note that the group $G$ acts on itself by the adjoint action $\mathrm{Ad}: G \times G \rightarrow G$,  $(g, h) \mapsto \mathrm{Ad}_g h = ghg^{-1}$, for all $g, h \in G$.

\begin{prop}
    Let $G$ be a group and $\Phi: G \times X \rightarrow X$ be an action of $G$ on a set $X$. Then the cartesian product $G \times X$ carries a rack structure with the binary operation
    \begin{align*}
        (g, x) \diamond (h, y) := (ghg^{-1}, \Phi_g y), \text{ for } (g,x), (h, y) \in G \times X.
    \end{align*}
    We denote this rack simply by $G \times_\Phi X.$
\end{prop}

\begin{proof}
    For any $(g, x), (h, y)$ and $(k, z) \in G \times X$, we have
    \begin{align*}
        (g, x) \diamond (( h, y) \diamond (k, z)) =~& (g, x) \diamond (hkh^{-1}, \Phi_h z) \\
        =~& ( ghkh^{-1} g^{-1}, \Phi_g \Phi_h z ) \\
        =~& \big( ghg^{-1} (gkg^{-1}) (ghg^{-1})^{-1}, \Phi_{ghg^{-1}} \Phi_g z   \big) \\
        =~& ( ghg^{-1}, \Phi_g y) \diamond (gkg^{-1}, \Phi_g z) \\
        =~& \big(  (g, x)  \diamond (h, y)  \big) \diamond \big(  (g, x) \diamond (k, z)  \big).
    \end{align*}
    This proves the left-distributivity of $\diamond$. Next, let $(g, x), (h, y) \in G \times X$ be two arbitrary elements. Then it is easy to see that the equation $(g, x) \diamond c = (h, y)$ has unique solution $c = (g^{-1} h g, \Phi_{g^{-1}} y)$. This proves that $(G \times X, \diamond)$ is a rack.
\end{proof}

\begin{defn}\label{rel-defn}
Let $G$ be a group and $\Phi: G \times X \rightarrow X$ be an action of $G$ on a set $X$.
    A {\bf relative averaging operator} (with respect to the action $\Phi$) is a set-map $\mathcal{B} : X \rightarrow G$ that satisfies
    \begin{align}\label{rel-avg}
        \mathcal{B}(x) \mathcal{B}(y) \mathcal{B}(x)^{-1} = \mathcal{B} (\Phi_{\mathcal{B} (x)} y), \text{ for } x, y \in X.
    \end{align}
\end{defn}

   Note that the condition (\ref{rel-avg}) can be equivalently written as $\mathrm{Ad}_{\mathcal{B} (x)} \mathcal{B} (y) = \mathcal{B} (\Phi_{\mathcal{B} (x)} y),$ for $x, y \in X.$ Further, any averaging operator on a group $G$ is a relative averaging operator with respect to the adjoint action $\mathrm{Ad}$.

   \begin{exam}
Let $G$ be a group and $N \subset G$ be a normal subgroup. Then the inclusion map $i: N \hookrightarrow G$ is a relative averaging operator, where $G$ acts on $N$ by the adjoint action.
\end{exam}

\begin{exam}
    Let $G$ be a group and $N \subset G$ be a normal subgroup. Suppose $\Phi: G \times X \rightarrow X$ is a group action that satisfies $\Phi_n x = x$, for all $n \in N$ and $x \in X$. Then $\Phi$ induces a group action $\widetilde{\Phi} : G/ N \times X \rightarrow X$ given by $\widetilde{\Phi}_{gN} x = \Phi_g x$, for $gN \in G/N$. If $\mathcal{A}: X \rightarrow G$ is a relative averaging operator (with respect to the action $\Phi$) then the map $\widetilde{\mathcal{A}}: X \rightarrow G/N$, $\widetilde{\mathcal{A}} (x) = \mathcal{A}(x) N$ is a relative averaging operator with respect to the action $\widetilde{\Phi}.$ 
\end{exam}

   The next two results generalizes Proposition \ref{prop-graph-group} and Proposition \ref{prop-des} in the relative context. Since the proofs are similar to the earlier results, we do not repeat them here.

\begin{prop}\label{graph-prop}
   Let $G$ be a group and $\Phi: G \times X \rightarrow X$ be an action of $G$ on a set $X$. A map $\mathcal{B}: X \rightarrow G$ is a relative averaging operator (with respect to the action $\Phi$) if and only if  the graph $Gr (\mathcal{B}) = \{ ( \mathcal{B}(x), x) |~ x \in X \}$ is subrack of $G \times_\Phi X.$
\end{prop}

%\begin{proof}
 %   For any $x, y \in Q$, consider the elements $(\mathcal{B} (x), x), (\mathcal{B} (y), y) \in Gr (\mathcal{B})$. Then we have
  %  \begin{align*}
   %     (\mathcal{B} (x), x) \diamond (\mathcal{B} (y), y) = \big( \mathcal{B} (x)  \mathcal{B} (y) \mathcal{B} (x)^{-1}, \Phi_{\mathcal{B} (x) } y  \big).
    %\end{align*}
    %This is in $Gr (\mathcal{B})$ if and only if $Gr (\mathcal{B}) \subset G \times_\Phi X$ is a subrack.
%\end{proof}

%As a consequence, we obtain the following result.

\begin{prop}
   Let $\mathcal{B}: X \rightarrow G$ be a relative averaging operator (with respect to the action $\Phi$).

   (i) Then $(X, \diamond_\mathcal{B})$ is a rack (called the descendent rack), where
   \begin{align}\label{des-rack}
       x \diamond_\mathcal{B} y := \Phi_{\mathcal{B} (x)} y, \text{ for } x, y \in X.
   \end{align}

   (ii) The map $\mathcal{B}: X \rightarrow G$ is a morphism of racks from the descendent rack to the rack structure on $G$ induced from the group structure.
\end{prop}

One may also consider relative averaging operators with respect to a representation (i.e, with respect to an action $\Phi: G \times V \rightarrow V$ on a vector space $V$ such that the map $\Phi_x: V \rightarrow V$ is linear, for each $x \in G$) similar to Definition \ref{rel-defn}. Note that, any Lie group has an adjoint representation on its Lie algebra. Therefore, in this case, we may define relative averaging operators with respect to the adjoint representation. We show that the differentiation yields an averaging operator on the corresponding Lie algebra.

Let $G$ be a Lie group with the Lie algebra $(\mathfrak{g}, [~,~])$. 
Note that the group $G$ acts on $G$ itself by the adjoint action $\mathrm{Ad} : G \times G \rightarrow G$, $\mathrm{Ad}_g h = ghg^{-1}$, for $g, h \in G$. Since the map $\mathrm{Ad}_g: G \rightarrow G$ takes the identity element to itself, by differentiating we obtain a linear map $(\mathrm{Ad}_g)_*: \mathfrak{g} \rightarrow \mathfrak{g}$. Finally, the adjoint representation of $G$ on its Lie algebra $\mathfrak{g}$ is given by the map (also denoted by the same notation) $\mathrm{Ad}: G \times \mathfrak{g} \rightarrow \mathfrak{g}$, $\mathrm{Ad}_g x:= (\mathrm{Ad}_g)_* x$, for $g \in G$ and $x \in \mathfrak{g}$. The differentiation of this action gives the adjoint action $\mathrm{ad}: \mathfrak{g} \times \mathfrak{g} \rightarrow \mathfrak{g}$ of the Lie algebra $\mathfrak{g}$ on itself. In other words, $\mathrm{ad}_x y = [x, y]= \frac{d^2}{dtds} \big|_{t, s = 0} \mathrm{Ad}_{\mathrm{exp} (tx)} sy,$ for $x, y \in \mathfrak{g}$.

%on its Lie algebra $\mathfrak{g}$ by adjoint action $\mathrm{Ad} : G \times \mathfrak{g} \rightarrow \mathfrak{g}$ given by $\mathrm{Ad}_g X = (\mathrm{Adj}_g)_* X$, where $\mathrm{Adj}_g : G \rightarrow G$ is the map $\mathrm{Adj}_g h = ghg^{-1}$, for $h \in G$. 

\begin{thm}
    Let $G$ be a Lie group with the Lie algebra $\mathfrak{g}$. If $\mathfrak{B} : \mathfrak{g} \rightarrow G$ is a smooth pointed (i.e. $\mathfrak{B} (0) = e$) relative averaging operator on $G$ with respect to the adjoint representation on $\mathfrak{g}$, then $P = \mathfrak{B}_{* 0 } : \mathfrak{g} \rightarrow \mathfrak{g}$ is an averaging operator on the Lie algebra $\mathfrak{g}$.
\end{thm}

\begin{proof}
    For any $x,y \in \mathfrak{g}$, we have
    \begin{align*}
        [P(x), P(y)] =~& \frac{d^2}{dt ds} \big|_{t, s = 0} \mathrm{exp}^{t P (x)} \mathrm{exp}^{s P(y)} \mathrm{exp}^{- t P (x)} \\
        =~&  \frac{d^2}{dt ds} \big|_{t, s = 0}  \mathfrak{B} (t x) \mathfrak{B}(sy) \mathfrak{B}(tx)^{-1} \\
        =~&  \frac{d^2}{dt ds} \big|_{t, s = 0} \mathfrak{B} \big(  \mathrm{Ad}_{ \mathfrak{B} (t x)  } sy \big) \\
        =~& \mathfrak{B}_{* 0 } \big(   \frac{d^2}{dt ds} \big|_{t, s = 0}   \mathrm{Ad}_{ \mathfrak{B} (t x)  } sy   \big) \\
        =~& P ( [P(x), y]).
    \end{align*}
    This shows that $P$ is an averaging operator on the Lie algebra $\mathfrak{g}$.
\end{proof}

Note that a Lie group $G$ has an adjoint representation on its Lie algebra. Hence one may easily define relative averaging operators with respect to the adjoint representation. For arbitrary abstract groups, this approach doesn't make sense. However, we may define relative averaging operators by using the Malcev completion of a group due to Quillen \cite{qui} (see also \cite{li-sheng-tang}).

Let $G$ be a group with the identity element $e$. First, consider the group ring ${\bf k}[G]$. This has a standard (cocommutative) Hopf algebra structure with the coproduct, counit and antipode respectively given by
\begin{align*}
    \Delta (x) = x \otimes x, \quad \varepsilon (x) = 1 ~~~~ \text{ and } ~~~~ S(x) = x^{-1}, \text{ for } x \in G.
\end{align*}
Let $I = \mathrm{ker} (\varepsilon) = \langle x - e | x \in G \rangle$ be the augmentation ideal of ${\bf k}[G]$. We define
\begin{align*}
    \widehat{ {\bf k}[G]   } := \lim_{n \rightarrow \infty} {\bf k}[G]/ I^n.
\end{align*}
Explicitly, we have $ \widehat{ {\bf k}[G]   } = \big\{  \sum_{i=0}^\infty x_i |~ \sum_{i=0}^n x_i \in {\bf k}[G]/ I^n \text{ for each } n \geq 0   \big\}.$ Note that the coproduct $\Delta$ induces a linear map $\widetilde{\Delta} : {\bf k}[G]/ I^n \rightarrow \bigoplus_{i+j = n} {\bf k}[G]/ I^i \otimes {\bf k}[G]/ I^j$. By taking the inverse limit, we obtain a coassociative map $\widehat{\Delta}:  \widehat{ {\bf k}[G]   } \hat{\otimes}  \widehat{ {\bf k}[G]   }$. The completion  $\widehat{ {\bf k}[G]   }$ is a complete Hopf algebra (say, with unit $\mathbbm{1}$) consisting of formal power series of group ring elements. Let $\widehat{G}$ be the set of all group-like elements of $ \widehat{ {\bf k}[G]   }$, and $\widehat{\mathfrak{g}}$ the primitive Lie algebra of $ \widehat{ {\bf k}[G]   }$. (When $G$ is finitely generated, the group $\widehat{G}$ is called the Malcev completion of $G$.) If $\widehat{I}$ is the completion of $I$ in $\widehat{ {\bf k}[G]}$, then we have 
\begin{align*}
    \widehat{G} = \{ f \in \mathbbm{1} + \widehat{I} |~ \widehat{\Delta} (f) = f \hat{\otimes} f   \} ~~~ \text{ and } ~~~~ \widehat{\mathfrak{g}} = \{ f \in \widehat{I} |~ \widehat{\Delta} (f) = \mathbbm{1} \hat{\otimes} f + f \hat{\otimes} \mathbbm{1} \}.
\end{align*}
For each $n$, we set $\widehat{G}_n = \widehat{G} \cap (\mathbbm{1} + \widehat{I}^n)$. Then $\widehat{G}^n := \widehat{G}/ \widehat{G}_{n+1}$ is a unipotent algebraic group over ${\bf k}$ (lying in $\mathbbm{1} + \widehat{I}/ \widehat{I}^{n+1}$) whose Lie algebra $\widehat{\mathfrak{g}}^n$ satisfies $\widehat{g}^n \subset \widehat{I}/ \widehat{I}^{n+1}$. Moreover, we have $\widehat{G} \cong \lim_{n \rightarrow \infty} \widehat{G}^n$ and $\widehat{\mathfrak{g}} \cong \lim_{n \rightarrow \infty} \widehat{\mathfrak{g}}^n$. Each group $\widehat{G}^n$ acts on its Lie algebra $\widehat{g}^n$ by adjoint action, which induces an action of $\widehat{G}$ on $\widehat{\mathfrak{g}}$. This gives rise to an action of $G$ on $\widehat{\mathfrak{g}}$. We call this the adjoint action of any abstract group $G$. Hence one may define a relative averaging operator (with respect to the adjoint action) as a map $\mathfrak{B}: \widehat{\mathfrak{g}} \rightarrow G$ that satisfies 
\begin{align*}
    \mathfrak{B} (f) \mathfrak{B}(g) \mathfrak{B} (f)^{-1} = \mathfrak{B} ( \mathfrak{B} (f) g \mathfrak{B} (f)^{-1}   ), \text{ for } f, g \in \widehat{\mathfrak{g}}. 
\end{align*}

%\begin{proof}
 %   (i) Note that the set $Q$ has a bijection with the set $Gr (\mathfrak{B})$ via $x \leftrightarrow (\mathfrak{B} (x) , x).$ Since $\mathfrak{B}$ is a relative averaging operator, it follows that $Gr (\mathfrak{B})$ has a rack structure (cf. Proposition \ref{graph-prop}). This structure induces a rack structure on $Q$ with the binary operation $\diamond_\mathfrak{B}$ given by (\ref{des-rack}).

  %  (ii) For any $x, y \in Q$, we have
   % \begin{align*}
    %    \mathfrak{B} (x \diamond_\mathfrak{B} y) = \mathfrak{B} (\Phi_{\mathfrak{B} (x)}) =  \mathfrak{B} (x)  \mathfrak{B} (y) \mathfrak{B} (x)^{-1} = \mathfrak{B} (x) \diamond_G \mathfrak{B} (y).
    %\end{align*}
    %Here $\diamond_G$ denotes the rack operation on $G$ induced from the group multiplication. Hence the proof.
%\end{proof}

%\section{Averaging operators on cocommutative Hopf algebras}

%\section{Averaging operators on Lie groupoids and Lie algebroids}

\section{Averaging operators on racks}\label{sec3}

In this section, we consider averaging operators on a rack. We observe that an averaging operator on a rack induces a hierarchy of new rack structures. 

\begin{defn}
(i) Let $(Q, \diamond)$ be a rack. A map $\mathcal{A} : Q \rightarrow Q$ is said to be an {\bf averaging operator} on $Q$ if it satisfies
\begin{align}
   \mathcal{A} (x) \diamond \mathcal{A}(y) = \mathcal{A} (  \mathcal{A} (x) \diamond y), \text{ for } x, y \in Q.
\end{align}
If $(Q, \diamond)$ is a pointed rack with distinguished element $e$, then an averaging operator $\mathcal{A}$ is called pointed if $\mathcal{A}(e) =e.$

(ii) An {\bf averaging rack} is a rack $Q$ equipped with an averaging operator $\mathcal{A} : Q \rightarrow Q$. An averaging rack is denoted by the pair $(Q, \mathcal{A})$.
\end{defn}

\begin{remark}
    Let $G$ be a group and consider the conjugation rack $(G, \diamond)$. Then a map $\mathcal{A}: G \rightarrow G$ is an averaging operator on the group $G$ if and only if it is an averaging operator on the rack $(G, \diamond)$.
\end{remark}

\begin{exam}
    Let $(Q, \diamond)$ be a rack. Then the identity map $\mathrm{Id}: Q \rightarrow Q$ is an averaging operator.
\end{exam}

\begin{exam}
    Let $(Q, \diamond)$ be a rack. Define $Z (Q) = \{ x \in Q | x \diamond y = y, \text{ for all } y \in Q \}$. Then any map $\mathcal{A} : Q \rightarrow Z (Q) \subset Q$ is an averaging operator on $Q$. A rack $(Q, \diamond)$ is said to be abelian if $Z (Q) = Q$. In this case, any map $\mathcal{A} : Q \rightarrow Q$ is an averaging operator on $Q$.
\end{exam}

\begin{exam}
    Let $A$ be an abelian group. Then $A$ carries a rack structure with the operation $a \diamond b := 2a - b$, for $a, b \in A$. (This is a quandle structure, called the {\em Takasaki quandle}). Let $\varphi: A \rightarrow A$ be an idempotent group homomorphism. Then for any $a, b \in A$,
    \begin{align*}
        \varphi ( \varphi (a) \diamond b) = \varphi (2 \varphi (a) - b) = 2 \varphi^2 (a) - \varphi (b) = 2 \varphi (a) - \varphi (b) = \varphi (a) \diamond \varphi (b).
    \end{align*}
    This shows that $\varphi$ is an averaging operator on the rack $(A, \diamond)$.
\end{exam}

The above example has the following generalization.

\begin{exam}
    Let $A$ be an abelian group and $t \in \mathrm{Aut} (A)$ be an automorphism. Then $A$ carries a rack structure with the operation $a \diamond b := (\mathrm{Id} - t) (a) + t (b)$, for $a, b \in A$. This is also a quandle structure, called the {\em Alexander quandle}. Let $\varphi: A \rightarrow A$ be an idempotent group homomorphism that commutes with $t$, then $\varphi$ is an averaging operator on the rack $(A, \diamond)$.
\end{exam}

\begin{exam}
    Let $X = \{ a_1, \ldots, a_n \}$ be a set with $n \geq 2$ many elements. Then one can define a rack structure on $X$ with the operation $a_i \diamond a_j := a_{n-j+1}$, for all $1 \leq i, j \leq n$. Let $\varphi : \{ 1, \ldots, n \} \rightarrow \{ 1, \ldots, n \}$ be a map that satisfies $\varphi (i) + \varphi (n-i+1) = n+1$, for all $i$, then the map $\mathcal{A} : X \rightarrow X$ defined by $\mathcal{A} (a_i)= a_{\varphi(i)}$ is an averaging operator on the rack $(X, \diamond)$. In fact, any averaging operator on $X$ corresponds to such a map $\varphi.$ In particular, if $n$ is odd, then the constant map $\mathcal{A} (a_i)= a_{\frac{n+1}{2}}$, for all $1\leq i \leq n$, is an averaging operator on $(X, \diamond)$.
\end{exam}

The following result is a generalization of Proposition \ref{prop-des} in the context of racks.

\begin{prop}
    Let $(Q, \diamond)$ be a rack and $\mathcal{A}: Q \rightarrow Q$ be an averaging operator on $Q$. Then $(Q, \diamond_\mathcal{A})$ is a rack (called the descendent rack), where
    \begin{align*}
        x \diamond_\mathcal{A} y := \mathcal{A} (x) \diamond y, \text{ for } x, y \in Q.
    \end{align*}
    Then $\mathcal{A} : Q \rightarrow Q$ is a morphism of racks from the descendent rack $(Q, \diamond_\mathcal{A})$ to the rack $(Q, \diamond)$. Moreover, $\mathcal{A}: Q \rightarrow Q$ is an averaging operator on the descendent rack $(Q, \diamond_\mathcal{A})$.
\end{prop}

\begin{proof}
First observe that for any $x \in Q$, the left translation map $x \diamond_\mathcal{A} -$ is simply $\mathcal{A}(x) \diamond -$ which is bijective on $Q$. Next, for any $x, y, z \in Q$, we observe that
\begin{align*}
    x \diamond_\mathcal{A} (y \diamond_\mathcal{A} z) =~& \mathcal{A} (x) \diamond (\mathcal{A}(y) \diamond z) \\
    =~& (\mathcal{A}(x) \diamond \mathcal{A}(y)) \diamond (\mathcal{A}(x) \diamond z) \\
    =~& \mathcal{A} ( \mathcal{A}(x) \diamond y) \diamond (x \diamond_\mathcal{A} z) \\
    =~& (x \diamond_\mathcal{A} y) \diamond_\mathcal{A} ( x \diamond_\mathcal{A} z ).
\end{align*}
This proves that the operation $\diamond_\mathcal{A}$ is left-distributive. Hence $(Q, \diamond_\mathcal{A})$ is a rack.

Next, for any $x, y \in Q$, we have $\mathcal{A} (x \diamond_\mathcal{A} y ) = \mathcal{A} ( \mathcal{A}(x) \diamond y) = \mathcal{A}(x) \diamond \mathcal{A}(y)$ which imples that $\mathcal{A}: Q \rightarrow Q$ is a morphism of racks from the descendent rack $(Q, \diamond_\mathcal{A})$ to the rack $(Q, \diamond)$. Finally, we also have
\begin{align*}
    \mathcal{A} ( \mathcal{A} (x) \diamond_\mathcal{A} y) = \mathcal{A} ( \mathcal{A}^2 (x) \diamond y) = \mathcal{A}^2 (x) \diamond \mathcal{A} (y) = \mathcal{A} (x) \diamond_\mathcal{A} \mathcal{A} (y),
\end{align*}
for any $x, y \in Q$. This shows that $\mathcal{A}: Q \rightarrow Q$ is an averaging operator on the descendent rack $(Q, \diamond_\mathcal{A})$.
\end{proof}

In the next result, we show that an averaging operator on a rack induces a hierarchy of new rack structures.

\begin{prop}\label{avg-power-rack}
    Let $(Q, \diamond)$ be a rack and $\mathcal{A}: Q \rightarrow Q$ be an averaging operator on $Q$.

    (i) Then for each $k \geq 0$, the map $\mathcal{A}^k : Q \rightarrow Q$ is an averaging operator on $Q$.

    (ii) For any $k, l \geq 0$, the map $\mathcal{A}^k : Q \rightarrow Q$ is an averaging operator on the descendent rack $(Q, \diamond_{\mathcal{A}^l})$.

    (iii) For any $k, l \geq 0$, the descendent racks $(Q, \diamond_{\mathcal{A}^{k+l}})$ and $(Q, (\diamond_{\mathcal{A}^l})_{\mathcal{A}^k})$ are the same.
\end{prop}

\begin{proof}
    (i) Suppose $\mathcal{A}^r : Q \rightarrow Q$ is an averaging operator, for some $r \in \mathbb{N}$. Then for any $x, y \in Q$, we have
    \begin{align*}
        \mathcal{A}^{r+1} (x) \diamond \mathcal{A}^{r+1} (y) =~& \mathcal{A}^r \big(  \mathcal{A}^{r+1} (x) \diamond \mathcal{A}(y)  \big) \quad (\because ~ \mathcal{A}^r \text{ is an averaging operator}) \\
        =~& \mathcal{A}^{r+1} \big( \mathcal{A}^{r+1} (x) \diamond y  \big) \quad (\because ~ \mathcal{A} \text{ is an averaging operator}).
    \end{align*}
    This shows that $\mathcal{A}^{r+1}$ is an averaging operator on $Q$. Hence by the mathematical induction, $\mathcal{A}^k$ is an averaging operator on $Q$, for any $k \geq 0$. 

    (ii) First observe that
    \begin{align*}
        \mathcal{A} (x) \diamond_{\mathcal{A}^l} \mathcal{A}(y) = \mathcal{A}^{l+1} (x) \diamond \mathcal{A} (y) = \mathcal{A} ( \mathcal{A}^{l+1} (x) \diamond y ) = \mathcal{A} ( \mathcal{A}(x) \diamond_{\mathcal{A}^l} y),
    \end{align*}
    for any $x, y \in Q$. This proves that $\mathcal{A} : Q \rightarrow Q$ is an averaging operator on the descendent rack $(Q, \diamond_{\mathcal{A}^l})$. Assume that $\mathcal{A}^r$ (for some $r \geq 0$) is an averaging operator on the rack $(Q, \diamond_{\mathcal{A}^l})$. Then for any $x, y \in Q$, we have
    \begin{align*}
        \mathcal{A}^{r+1} (x) \diamond_{\mathcal{A}^l} \mathcal{A}^{r+1} (y) = \mathcal{A}^r (\mathcal{A} (x) ) \diamond_{\mathcal{A}^l} \mathcal{A}^r  (\mathcal{A} (y)) = \mathcal{A}^r ( \mathcal{A}^{r+1} \diamond_{\mathcal{A}^l} \mathcal{A} (y)) = \mathcal{A}^{r+1} \big(  \mathcal{A}^{r+1} (x) \diamond_{\mathcal{A}^l} y \big).
    \end{align*}
    Hence $\mathcal{A}^{r+1}$ is an averaging operator on $(Q, \diamond_{\mathcal{A}^l})$. Therefore, the result is followed by the mathematical induction.

    (iii) For any $x, y \in Q$, we have $x (\diamond_{\mathcal{A}^l})_{\mathcal{A}^k} y = \mathcal{A}^k (x) \diamond_{\mathcal{A}^l} y = \mathcal{A}^{k+l} (x) \diamond y = x \diamond_{\mathcal{A}^{k+l}} y$ which shows that the racks $(Q, \diamond_{\mathcal{A}^{k+l}})$ and $(Q, (\diamond_{\mathcal{A}^l})_{\mathcal{A}^k})$ are the same.
\end{proof}

%\begin{prop}
%    Let $\mathcal{A}: G \rightarrow G$ be an averaging operator on a group $G$. Then $\mathcal{A}$ is an averaging operator on the descendent rack $(G, \ast).$
%\end{prop}

%\begin{proof}
%    For any $g, h \in G$, 
%    \begin{align*}
%        \mathcal{A} (g) \ast \mathcal{A}(h) = \mathcal{A}^2 (g) \mathcal{A} (h) ( \mathcal{A}^2(g))^{-1} = \mathcal{A} \big(  \mathcal{A}^2(g) h (\mathcal{A}^2 (g) )^{-1} \big) = \mathcal{A} ( \mathcal{A} (g) \ast h).
%    \end{align*}
%    Hence the result follows.
%\end{proof}

\begin{remark}
    Let $\mathcal{A}: G \rightarrow G$ be an averaging operator on a group $G$. Then for any $k \geq 0$, the pair $(G, \diamond_{\mathcal{A}^k})$ is a rack, where $g \diamond_{\mathcal{A}^k} h := \mathcal{A}^k (g) h (\mathcal{A}^k (g))^{-1}$, for $g, h \in G.$
\end{remark}

%Till now, we mainly considered averaging operators on groups and racks, and also studied averaging operators on Lie algebras. These algebraic structures (namely, groups, racks and Lie algebras) are closely related to the solutions of the Yang-Baxter equation

\section{Rack-pairings and di-racks}\label{sec4}
In this section, we first introduce the notion of a rack-pairing. We show that an averaging operator on a rack yields a rack-pairing. Any two-sided skew brace also gives rise to a rack-pairing by conjugation. We give some characterizations of rack-pairings. In the end, we introduce and study a closely related structure, called di-racks.

\begin{defn}
    A {\bf (left) rack-pairing} is a triple $(Q, \diamond, \smallblackdiamond)$ that consists of a set $Q$ with two binary operations $\diamond$ and $\smallblackdiamond$ that make both $(Q, \diamond)$ and $(Q, \smallblackdiamond)$ into (left) racks satisfying additionally
    \begin{align}\label{rack-pair-comp}
        x \smallblackdiamond (y \diamond z) = (x \smallblackdiamond y) \diamond (x \smallblackdiamond z), \text{ for } x, y, z \in Q.
        \end{align}
\end{defn}

The condition (\ref{rack-pair-comp}) simply means that for all $x \in Q$, the left translation $L_x^\smallblackdiamond := x \smallblackdiamond -$ is a rack automorphism on $(Q, \diamond)$. 

Let $(Q, \diamond, \smallblackdiamond)$ and $(Q', \diamond', \smallblackdiamond')$ be two rack-pairings. A {\bf morphism} from $(Q, \diamond, \smallblackdiamond)$ to $(Q', \diamond', \smallblackdiamond')$ is a set-map $\varphi : Q \rightarrow Q'$ that satisfies $\varphi (x \diamond y) = \varphi (x) \diamond' \varphi (y)$ and $\varphi (x \smallblackdiamond y) = \varphi (x) \smallblackdiamond' \varphi (y)$, for all $x, y \in Q$.

\begin{exam}\label{tr-rack-p}
    Let $(Q, \diamond)$ be a rack. Then $(Q, \diamond, \diamond)$ is a rack-pairing, called the trivial rack-pairing.
\end{exam}

\begin{exam}
    Let $(Q, \diamond)$ be any rack. Then $(Q, \diamond, \smallblackdiamond)$ and $(Q, \smallblackdiamond, \diamond)$ are both rack-pairings, where $\smallblackdiamond$ is the trivial rack $x \smallblackdiamond y = y$.
\end{exam}

\begin{exam}
    Let $(Q, \diamond, \smallblackdiamond)$ and $(Q', \diamond', \smallblackdiamond')$ be two rack-pairings. Then $(Q \times Q', \diamond \times \diamond', \smallblackdiamond \times \smallblackdiamond')$ is also a rack-pairing.
\end{exam}

%A rack-pairing is said to be {\bf trivial} if $ x \diamond y = x \smallblackdiamond y $, for all $x, y \in Q$.

\begin{exam}
    Let $G$ be a group. Then $(G, \diamond, \smallblackdiamond)$ and $(G, \smallblackdiamond, \diamond)$ are both rack-pairings, where
    \begin{align*}
        g \diamond h = ghg^{-1} ~~~~ \text{ and } ~~~~ g \smallblackdiamond h = g^{-1} h g, \text{ for } g, h \in G.
    \end{align*}
\end{exam}

\begin{exam}
    Let $(Q, \diamond)$ and $(Q', \diamond')$ be two racks. Suppose $\alpha : (Q , \diamond) \rightarrow \mathrm{Aut} (Q', \diamond')$ is a homomorphism of racks, where $\mathrm{Aut} (Q', \diamond')$ is endowed with the conjugation rack structure. Then $(Q \times Q' , \square, \blacksquare)$ is a rack-pairing, where for all $(x,x') , (y, y') \in Q \times Q'$,
    \begin{align*}
        (x,x') \square (y, y') =~& (x \diamond y, x' \diamond' y'),\\
        (x, x') ~\blacksquare~ (y, y') =~& (x \diamond y, x' \diamond' \alpha (x) (y')).
    \end{align*}
\end{exam}

\begin{exam}
Let $Q$ be a nonempty set and $(Q', \diamond')$ be any rack. Let $\alpha : Q \rightarrow \mathrm{Aut}(Q, \diamond')$ be a map that satisfies $\alpha (x) \alpha (y) = \alpha (y) \alpha (x)$, for all $x, y \in Q$. Then the product $Q \times Q'$ can be given a rack-pairing structure with the operations
\begin{align*}
    (x, x') \square (y, y') =~& (y, x' \diamond' \alpha (x) (y') ),\\
    (x, x') ~\blacksquare ~(y, y') =~& (y, x' \diamond' y').
\end{align*}
\end{exam}

\begin{prop}\label{gpr-skr}
    Let $(G, \cdot, \smallblackdiamond)$ be a pointed group-rack. Then the triple $(G, \diamond, \smallblackdiamond)$ is a rack-pairing, where $x \diamond y = x y x^{-1}, \text{ for } x, y \in G.$
\end{prop}

\begin{proof}
    For any $x, y, z \in G$, we have
    \begin{align*}
        x \smallblackdiamond (y \diamond z) =~& x \smallblackdiamond (yz y^{-1}) \\
        =~& (x \smallblackdiamond yz) (x \smallblackdiamond y^{-1}) \\
        =~& (x \smallblackdiamond y) (x \smallblackdiamond z) (x \smallblackdiamond y)^{-1} \\
        =~& (x \smallblackdiamond y) \diamond (x \smallblackdiamond z).
    \end{align*}
    Hence the compatibility condition of a rack-pairing follows.
\end{proof}

\begin{prop}\label{avg-rack-pair}
    Let $(Q, \diamond)$ be a rack and $\mathcal{A}: Q \rightarrow Q$ be an averaging operator on $Q$. Then $(Q, \diamond , \diamond_\mathcal{A})$ is a rack-pairing, where $\diamond_\mathcal{A}$ is the descendent rack structure induced by $\mathcal{A}$. More generally, for any $k, l \geq 0$, the triple $(Q, \diamond_{\mathcal{A}^k}, \diamond_{\mathcal{A}^{k+l}})$ is a rack-pairing.
\end{prop}

\begin{proof}
    We only need to verify the compatibility condition (\ref{rack-pair-comp}) of rack-pairing. For any $x, y, z \in Q$, we have
    \begin{align*}
        x \diamond_\mathcal{A} (y \diamond z) = \mathcal{A}(x) \diamond (y \diamond z) = (\mathcal{A}(x) \diamond y) \diamond (\mathcal{A} (x) \diamond z) = (x \diamond_\mathcal{A} y ) \diamond (x \diamond_\mathcal{A} z).
    \end{align*}
    This proves the first part. In Proposition \ref{avg-power-rack}, we have seen that $\mathcal{A}^k$ and $\mathcal{A}^{k+l}$ are both averaging operators on the rack $Q$. Hence $(Q, \diamond_{\mathcal{A}^{k}})$ and $(Q, \diamond_{\mathcal{A}^{k+l}})$ are the descendent racks. We also have
    \begin{align*}
        x \diamond_{\mathcal{A}^{k+l}} (y \diamond_{\mathcal{A}^k} z) =~& \mathcal{A}^{k+l} (x) \diamond (\mathcal{A}^{k} (y) \diamond z) \\
        =~& (\mathcal{A}^{k+l}(x) \diamond \mathcal{A}^{k} (y)) \diamond (\mathcal{A}^{k+l} (x) \diamond z) \\
        =~& \mathcal{A}^{k} ( \mathcal{A}^{k+l} (x) \diamond y) \diamond (x \diamond_{\mathcal{A}^{k+l}} z) \\
        =~& (x \diamond_{\mathcal{A}^{k+l}} y) \diamond_{\mathcal{A}^{k}} (x \diamond_{\mathcal{A}^{k+l}} z).
    \end{align*}
    This proves the compatibility condition of the rack-pairing $(Q,  \diamond_{\mathcal{A}^{k}},  \diamond_{\mathcal{A}^{k+l}} ).$
\end{proof}

Let $(Q, \diamond)$ be a rack. A {\bf module} over $(Q, \diamond)$ is a set $X$ with a map $\Phi : Q \times X \rightarrow X,(x, u) \mapsto \Phi_x u$ such that $\Phi_x : X \rightarrow X$ is bijective for all $x \in Q$ and satisfies $\Phi_x \Phi_y = \Phi_{x \diamond y} \Phi_x$, for all $x, y \in Q$. The following remark can be regarded as a generalization of Proposition \ref{avg-rack-pair}.

%Given a module $(X, \Phi)$, a set-map $\pi : X \rightarrow Q$ is said to be {\bf equivariant} if it satisfies
%\begin{align*}
%    \pi (\Phi_x a) = x \diamond \pi (a), \text{ for } x \in Q, a \in X.
%\end{align*}

\begin{remark}
     Let $(Q, \diamond_Q)$ and $ (R, \diamond_R)$ be two racks. Suppose the rack $(Q, \diamond_Q)$ acts on $R$ by a map $\Phi: Q \times R \rightarrow R$ such that for each $x \in Q$, the map $\Phi_x: R \rightarrow R$ is a rack automorphism.  In this case, we say that the rack $(Q, \diamond_Q)$ acts on the rack $(R, \diamond_R)$ by rack automorphisms. Let $\mathcal{A}: R \rightarrow Q$ be a map that satisfies $\mathcal{A} (u) \diamond_Q \mathcal{A} (v) = \mathcal{A} ( \Phi_{\mathcal{A}(u) }v)$, for all $u, v \in R$. Then $\mathcal{A}$ is called a {\bf relative averaging operator} with respect to the action $\Phi$. Then the binary operation $\diamond_\mathcal{A}: R \times R \rightarrow R$ defined by $u \diamond_\mathcal{A} v:= \Phi_{\mathcal{A}(u) }v$ is a new rack structure on $R$. Moreover, the triple $(R, \diamond_R, \diamond_\mathcal{A})$ is a rack-pairing. 
     %This rack-pairing is $\lambda$-homomorphic if and only if $\Phi_{\mathcal{A} (x \diamond_R y)} = \Phi_{\mathcal{A} (x) \diamond_Q \mathcal{A} (y)}$, for all $x, y \in R$. 

     Conversely, any rack-pairing is obtained in this way. To see this, let $(Q, \diamond, \smallblackdiamond)$ be any rack-pairing. Then the rack $(Q, \smallblackdiamond)$ acts on the rack $(Q,\diamond)$ by rack automorphisms via the map $\Phi: Q \times Q \rightarrow Q,~ \Phi_x y = x \smallblackdiamond y$. 
 With this notation, the identity map $\mathrm{id}: Q \rightarrow Q$ is a relative averaging operator with respect to $\Phi$. Note that the induced rack-pairing structure is given by $(Q, \diamond, \diamond_\mathrm{Id} = \smallblackdiamond)$ which is the given one.

\end{remark}

The above remark doesn't say that any rack-pairing can be obtained from an averaging rack. However, the next result shows that under a suitable assumption, any rack-pairing comes from an averaging rack.

A (left) rack $(Q, \diamond)$ is said to be {\bf complete} if for any $x, y \in Q$, the equation $c \diamond x = y$ has a unique solution for $c$. This definition is motivated by the completeness of a group. Recall that a group $G$ is complete if $Z(G) = e$ and any automorphism of $G$ is obtained by conjugation. If $G$ is a complete group then the corresponding conjugation rack $(G, \diamond)$ is complete.

\begin{prop}
    Let $(Q, \diamond, \smallblackdiamond)$ be a rack-pairing such that $(Q, \diamond)$ is a complete rack. Then there exists an averaging operator $\mathcal{A} : Q \rightarrow Q$ on the rack $(Q, \diamond)$ such that $(Q, \diamond, \smallblackdiamond) = (Q, \diamond, \diamond_\mathcal{A})$.
\end{prop}

\begin{proof}
    We define a map $\mathcal{A} : Q \rightarrow Q$ by the relation
    \begin{align*}
        x \smallblackdiamond y = \mathcal{A}(x) \diamond y, \text{ for } x, y \in Q.
    \end{align*}
    The map $\mathcal{A}$ is well-defined as the rack $(Q, \diamond)$ is complete. Next, for any $x, y, z \in Q$, we observe that
    \begin{align}\label{compl1}
        \mathcal{A}(x) \diamond (\mathcal{A}(y) \diamond z) = \mathcal{A}(x) \diamond (y \smallblackdiamond z) = x \smallblackdiamond (y \smallblackdiamond z) 
        = (x \smallblackdiamond y) \smallblackdiamond (x \smallblackdiamond z)
        = \mathcal{A} (\mathcal{A}(x) \diamond y) \diamond (x \smallblackdiamond z).
    \end{align}
    On the other hand,
    \begin{align}\label{compl2}
        (\mathcal{A}(x) \diamond \mathcal{A}(y)) \diamond (\mathcal{A}(x) \diamond z) = (\mathcal{A}(x) \diamond \mathcal{A}(y)) \diamond (x \smallblackdiamond z).
    \end{align}
    Since $\diamond$ satisfies the left-distributive property, it follows from (\ref{compl1}) and (\ref{compl2}) that
    \begin{align*}
        \mathcal{A} (\mathcal{A}(x) \diamond y) \diamond (x \smallblackdiamond z) = (\mathcal{A}(x) \diamond \mathcal{A}(y)) \diamond (x \smallblackdiamond z).
    \end{align*}
    As the rack $(Q, \diamond)$ is complete, we have $\mathcal{A}(x) \diamond \mathcal{A}(y) = \mathcal{A} (\mathcal{A}(x) \diamond y)$, for all $x, y \in Q$. This shows that $\mathcal{A}: Q \rightarrow Q$ is an averaging operator on the rack $(Q, \diamond)$. Moreover, we have $x \diamond_\mathcal{A} y = \mathcal{A}(x) \diamond y = x \smallblackdiamond y$, for all $x, y \in Q$. This completes the last part.
\end{proof}

\medskip

We know that a rack can be obtained from a group using conjugation. In \cite{guar} Guarnieri and Vendramin have introduced the notion of a (left) skew brace which consists of two group operations on a nonempty set satisfying a compatibility. A skew brace is a generalization of the earlier notion of brace in which one group operation is abelian \cite{rump}. See \cite{bardakov,das-rathee} for various connections between Rota-Baxter groups and skew braces. Recently, Trappeniers \cite{trap} have extensively studied two-sided skew braces. In particular, he introduced the notion of a weakly trivial skew brace and showed that every two-sided skew brace is an extension of a weakly trivial skew brace by a two-sided brace. In the following, we will show that a two-sided skew brace gives rise to a rack-pairing by conjugation.

\begin{defn}
    A {\bf (left) skew brace} is a triple $(G, \cdot, \bullet)$ in which $(G, \cdot)$ and $(G, \bullet)$ are both groups satisfying the compatibility
    \begin{align}\label{skew-b-id}
        x \bullet (y z) = (x \bullet y) x^{-1} (x \bullet z), \text{ for } x, y, z \in G.
    \end{align}
    As before, we use the notation $xy$ for the product $x \cdot y$. Here $x^{-1}$ denotes the inverse of $x$ in the group $(G, \cdot)$.
\end{defn}

\begin{remark}
Let $(G, \cdot, \bullet)$ be a skew brace. Then it turns out from (\ref{skew-b-id}) that both the groups $(G, \cdot)$ and $(G, \bullet)$ share the same identity element, denoted simply by $e$.
\end{remark}

A skew brace $(G, \cdot, \bullet)$ is said to be a {\bf two-sided skew brace} if the following additional compatibility condition holds:
\begin{align}\label{two-br}
    (xy) \bullet z = (x \bullet z) z^{-1} (y \bullet z), \text{ for } x, y, z \in G.
\end{align}
Any group $(G, \cdot)$ can be realized as a two-sided skew brace  $(G, \cdot, \cdot)$. This is called the trivial two-sided skew brace.

%In the following, we will show that a two-sided skew brace gives rise to a group-rack, hence a skew-rack by conjugation. 

In a two-sided skew brace, we have the following observations.

\begin{lemma}\label{inv-lemma}
Let $(G, \cdot, \bullet)$ be a two-sided skew brace. Then for any $x, y \in G$, we have
\begin{itemize}
    \item[(i)] $(x y^{-1}) \bullet z = (x \bullet z) (y \bullet z)^{-1} z,$
    \item[(ii)] $(x^{-1} y) \bullet z = z (x \bullet z)^{-1} (y \bullet z),$
    \item[(iii)] $x^{-1} (\overline{x}^{-1} \bullet x) x^{-1} = e$, where $\overline{x}$ is the inverse of $x$ in the group $(G, \bullet)$.
\end{itemize}
\end{lemma}

\begin{proof}
    (i) In the identity (\ref{two-br}), take $w = xy$ (i.e. $x = w y^{-1}$). Then we get
    \begin{align*}
        w \bullet z = (w y^{-1} \bullet z) z^{-1} (y \bullet z) \quad \text{ equivalently, } \quad  (wy^{-1} \bullet z) = (w \bullet z) (y \bullet z)^{-1} z.
    \end{align*}

    (ii) Similarly, by taking $w = xy$ (i.e. $y= x^{-1} w$) in the identity (\ref{two-br}), we obtain
    \begin{align*}
        w \bullet z = (x \bullet z) z^{-1}  (x^{-1} w \bullet z)  \quad \text{ equivalently, } \quad  (x^{-1} w \bullet z) = z (x \bullet z)^{-1} (w \bullet z).
    \end{align*}

    (iii) It follows either from part (i) or (ii) that $x^{-1} \bullet z = z (x \bullet z)^{-1} z$, for any $x, z \in G$. Hence
    \begin{align*}
        x^{-1} (\overline{x}^{-1} \bullet x) x^{-1} = x^{-1} x (\overline{x} \bullet x) x x^{-1} = 1_G ~~~ \text{ as } x^{-1} x = x x^{-1} = \overline{x} \bullet x = e.
    \end{align*}
\end{proof}

\begin{prop}\label{prop-similar}
    Let $(G, \cdot, \bullet)$ be a two-sided skew brace. Then $(G, \cdot, \smallblackdiamond)$ is a group-rack, where
    \begin{align*}
        x \smallblackdiamond y := x \bullet y \bullet \overline{x}, \text{ for } x, y \in G.
    \end{align*}
    Hence $(G, \diamond, \smallblackdiamond)$ is a rack-pairing, where $x \diamond y := x y x^{-1}$, for $x, y \in G$. Here $\overline{x}$ is the inverse of $x$ in the group $(G, \smallblackdiamond)$.
\end{prop}

\begin{proof}
    Since $(G, \bullet)$ is a group, it follows that $(G, \smallblackdiamond)$ is a rack. Next, for any $x, y, z \in G$, we observe that
    \begin{align*}
        x \smallblackdiamond (yz) =~& x \bullet (yz) \bullet \overline{x} \\
        =~& \big(  (x \bullet y) x^{-1} (x \bullet z)  \big) \bullet \overline{x}  \quad (\text{by } (\ref{skew-b-id}))\\
        =~&  \big(  ((x \bullet y) x^{-1}) \bullet \overline{x} \big) \overline{x}^{-1} (x \bullet z \bullet \overline{x}) \quad (\text{by } (\ref{two-br}))\\
        =~& (x \bullet y \bullet \overline{x}) \underbrace{ \overline{x}^{-1} (x^{-1} \bullet \overline{x}) \overline{x}^{-1}}_{=~ e ~~ (\text{Lemma }\ref{inv-lemma} (iii))} (x \bullet z \bullet \overline{x}) \\
        =~& (x \smallblackdiamond y) (x \smallblackdiamond z).
    \end{align*}
    This proves that $(G, \cdot, \smallblackdiamond)$ is a group-rack. This is a pointed group-rack with $e$ as a distinguished element for the underlying rack. Hence the last part follows from Proposition \ref{gpr-skr}.
\end{proof}

In Example \ref{tr-rack-p}, we have seen that any rack can be regarded as trivial rack-pairing. On the other hand, any group is a trivial two-sided skew brace. Moreover, we have the obvious commutative diagram

\[
\xymatrix{
\mathrm{ group } \ar\ar@{^{(}->}[d] \ar[rr]^{\text{ rackification }} & & \mathrm{ rack } \ar\ar@{^{(}->}[d] \\
\text{two-sided skew brace} \ar[rr]_{\text{ rackification }} & & \text{rack-pairing.}
}
\]

\medskip

Like an averaging operator on a group gives rise to a rack structure, an averaging operator on a two-sided skew brace induces rack-pairing.
Let $(G, \cdot, \bullet)$ be a two-sided skew brace. 
%We say that the skew brace $(G, \cdot, \bullet)$ acts on a set $X$ if there exist maps  $\Phi_\cdot :  G \times X \rightarrow X$ and $\Phi_\bullet : G \times X \rightarrow X$ that satisfy the following

%- the group $(G, \cdot)$ acts on the set $X$ by the map $\Phi_\cdot$,

%- the group $(G, \bullet)$ acts on the set $X$ by the map $\Phi_\bullet$,

%- the maps $\Phi_\cdot$ and $\Phi_\bullet$ satisfy the compatibility: for $x, y \in G$,
%\begin{align}\label{action-comp}
%    (\Phi_\bullet)_x (\Phi_\cdot)_y = (\Phi_\cdot)_{(x \bullet y) x^{-1}} (\Phi_\bullet)_x.
%\end{align}

\begin{defn}
    Let $(G, \cdot, \bullet)$ be a two-sided skew brace.
    %which acts on a set $X$ by the maps $\Phi_\cdot$ and $\Phi_\bullet$. 
    An {\bf averaging operator} on $G$ is a map $\mathcal{A}: G \rightarrow G$ that satisfies
    \begin{align*}
        \mathcal{A}(x) \mathcal{A}(y) \mathcal{A}(x)^{-1} = \mathcal{A} ( \mathcal{A}(x) y \mathcal{A}(x)^{-1}) ~~~ \text{ and } ~~~ \mathcal{A}(x) \bullet \mathcal{A}(y) \bullet \overline{\mathcal{A}(x)} = \mathcal{A} ( \mathcal{A}(x) \bullet y \bullet \overline{\mathcal{A}(y)}), \text{ for } x, y \in X.
    \end{align*}
    It is called pointed if further $\mathcal{A} (e) = e$.
\end{defn}

\begin{prop}
    Let $(G, \cdot, \bullet)$ be a two-sided skew brace and $\mathcal{A}: G \rightarrow G$ be a pointed averaging operator. Then $(G, \diamond_\mathcal{A}, \smallblackdiamond_\mathcal{A})$ is a rack-pairing, where
    \begin{align*}
        x \diamond_\mathcal{A} y := \mathcal{A}(x) y \mathcal{A}(x)^{-1} ~~~ \text{ and } ~~~ x \smallblackdiamond_\mathcal{A} y := \mathcal{A}(x) \bullet y \bullet \overline{\mathcal{A}(x)}, \text{ for } x, y \in X.
    \end{align*}
\end{prop}

\begin{proof}
Since $\mathcal{A} : G \rightarrow G$ is an averaging operator for both the groups $(G, \cdot)$ and $(G, \bullet)$, it follows that $\diamond_\mathcal{A}$ and $\smallblackdiamond_\mathcal{A}$ are both rack structures on $G$. Moreover, by a calculation similar to the one given in Proposition \ref{prop-similar}, one yields that
\begin{align*}
    x \smallblackdiamond_\mathcal{A} (yz ) = \mathcal{A}(x) \bullet (yz) \bullet \overline{ \mathcal{A}(x)} = (x \smallblackdiamond_\mathcal{A} y) (x \smallblackdiamond_\mathcal{A} z),
\end{align*}
for any $x, y, z \in G$. Hence we have
\begin{align*}
    x \smallblackdiamond_\mathcal{A} (y \diamond_\mathcal{A} z) =~& x \smallblackdiamond_\mathcal{A} ( \mathcal{A}(y) z \mathcal{A}(y)^{-1} ) \\
    =~& (x \smallblackdiamond_\mathcal{A} \mathcal{A} (y) ) (x \smallblackdiamond_\mathcal{A} z) (x \smallblackdiamond_\mathcal{A} \mathcal{A} (y)^{-1} ) \\
    =~& (\mathcal{A} (x) \bullet \mathcal{A} (y) \bullet \overline{ \mathcal{A} (x)})  (x \smallblackdiamond_\mathcal{A} z) (\mathcal{A} (x) \bullet \mathcal{A} (y) \bullet \overline{ \mathcal{A} (x)})^{-1} \\
    =~& \mathcal{A} (x \smallblackdiamond_\mathcal{A} y) (x \smallblackdiamond_\mathcal{A} z)  \mathcal{A} (x \smallblackdiamond_\mathcal{A} y)^{-1} \\
    =~& (x \smallblackdiamond_\mathcal{A} y) \diamond_\mathcal{A} (x \smallblackdiamond_\mathcal{A} y)
\end{align*}
which proves the desired compatibility.
\end{proof}

%    From the compatibility condition (\ref{action-comp}), it follows that
%    \begin{align*}
 %       (\Phi_0)_{A(x)} (\Phi_\cdot)_{A(y)} z = (\Phi_\cdot)_{A(x) \circ A(y) \circ A(x)^{-1 (0)}} (\Phi_0)_{A(x)} z.
%    \end{align*}
%    This is equivalent to
%    \begin{align*}
%        x \diamond (y \smallblackdiamond z) = (x \diamond y) \smallblackdiamond (x \diamond z).
 %   \end{align*}
%    Hence the compatibility condition of skew-rack follows. This proves the result.
%\end{proof}

In the following, we give some characterizations of rack-pairing on a given rack. Let $G$ be a group and $\Phi : G \times X \rightarrow X$ be an action of $G$ on a set $X$. A  map $\pi : X \rightarrow G$ is said to be {\bf equivariant} if it satisfies
\begin{align*}
    \pi (\Phi_g x) = g \pi (x) g^{-1} \quad \text{ equivalently} \quad \pi (\Phi_g x) = \mathrm{Ad}_g \pi (x), \text{ for all } g \in G, x \in X.
\end{align*}
Let $(Q, \diamond)$ be a rack. 
%Then for any $x \in Q$, the map $L_x : Q \rightarrow Q$, $y \mapsto x \diamond y$ is a rack automorphism. 
Let $\mathrm{Inn}(Q, \diamond)$ be the subgroup of $\mathrm{Aut}(Q, \diamond)$ generated by the set $\{ L_x | x \in Q \}$. Then we have already seen that the left-distributivity of the rack is equivalent to the fact that
\begin{align}\label{l-red}
L_{x \diamond y} = L_x L_y L_x^{-1}, \text{ for all } x, y \in Q.
\end{align}
The group $\mathrm{Inn}(Q, \diamond)$ is called the group of inner automorphisms of $(Q, \diamond)$. Note that the group $\mathrm{Inn}(Q, \diamond)$ acts on the set $Q$ by the natural map $\Phi : \mathrm{Inn}(Q, \diamond) \times Q \rightarrow Q$ given by $\Phi_f y = f(y)$, for $f \in \mathrm{Inn}(Q, \diamond)$ and $y \in Q$. We have the following result \cite{solo}.

\begin{prop}\label{rack-prop}
    Let $Q$ be any nonempty set. Then the following are equivalent.

    (i) A rack structure on $Q$,

    (ii) A group $G$, a faithful $G$-action on $Q$ and an equivariant map $\pi : Q \rightarrow G$ such that $\pi (Q)$ generates the group $G$.
\end{prop}

\begin{proof}
Let $(Q, \diamond)$ be a rack structure on $Q$. Consider the group $G = \mathrm{Inn}(Q, \diamond)$ that acts on the set $Q$ as given above. This action is faithful. We define a map $\pi : Q \rightarrow \mathrm{Inn}(Q, \diamond)$ by $\pi (x) = L_x$, for $x \in Q$. Then $\pi (Q)$ generates the group $\mathrm{Inn}(Q, \diamond)$ and it follows from (\ref{l-red}) that $\pi$ is an equivariant map.

Conversely, let $G$ be a group, $\Phi : G \times Q \rightarrow Q$ be a faithful $G$-action on $Q$ and $\pi : Q \rightarrow G$ be an equivariant map such that $\pi (Q)$ generates $G$. Define a binary operation $\diamond : Q \times Q \rightarrow Q$ by $x \diamond y := \Phi_{\pi (x) } y$, for $x, y \in Q$. For any $x, y, z \in Q$, we have
\begin{align*}
    x \diamond (y \diamond z) = \Phi_{\pi (x) } \Phi_{\pi(y)} z =~&\Phi_{\pi(x) \pi(y)} z \\
    =~& \Phi_{\pi(x) \pi(y) \pi(x)^{-1} \pi (x)} z \\
=~& \Phi_{ \pi ( \Phi_{\pi(x)} y)} \Phi_{\pi(x)} z \\
=~& \Phi_{\pi (x \diamond y)} \Phi_{\pi(x)} z \\
=~& \Phi_{ \pi (x \diamond y)} (x \diamond z) \\
=~& (x \diamond y) \diamond (x \diamond z).
\end{align*}
Finally, since the map $\Phi_{\pi (x)}$ is a bijection on $Q$, the equation $x \diamond c = y$ has a unique solution for $c$. Hence the pair $(Q, \diamond)$ forms a rack.
\end{proof}

Using the above result, one can prove the following.

\begin{prop}
    Let $(Q, \diamond)$ be a rack. Then the following data are equivalent.

    (i) There exists a binary operation $\smallblackdiamond : Q \times Q \rightarrow Q$ that makes the triple $(Q, \diamond, \smallblackdiamond)$ into a rack-pairing,

    (ii) A group $G$, a faithful $G$-action on $(Q, \diamond)$ by rack automorphisms and an equivariant map $\pi : Q \rightarrow G$ such that $\pi (Q)$ generates the group $G$.
\end{prop}

\begin{proof}
    We have already seen in Proposition \ref{rack-prop} that a (new) rack structure $\smallblackdiamond : Q \times Q \rightarrow Q$ on the set $Q$ is equivalent to having a group $G$, a faithful $G$-action on the set $Q$ and an equivariant map $\pi : Q \rightarrow G$ such that $\pi (Q)$ generates $G$. Under this correspondence, the triple $(Q, \diamond, \smallblackdiamond)$ is a rack-pairing if and only if the $G$-action on the set $Q$ is an action by rack automorphisms of $(Q, \diamond)$.
\end{proof}

Let $G$ be a group. In \cite{guar}, Guarnieri and Vendramin gave a characterization of skew brace structures on $G$ in terms of regular subgroups of the holomorph. We generalize their result in the context of racks.

Let $(Q, \diamond)$ be a rack. Then the holomorph of $Q$ is the rack $\mathrm{Hol}(Q, \diamond) := \mathrm{Aut} (Q, \diamond) \times Q$ with the binary operation
\begin{align*}
    (f, x) \overline{\diamond} (g, y) = (f g f^{-1}, f(y)), \text{ for } (f, x) , (g, y) \in \mathrm{Hol}(Q, \diamond).
\end{align*}
Let $H \subset \mathrm{Hol}(Q, \diamond)$ be any subrack. Then $H$ acts on the set $Q$ by
\begin{align*}
    \Psi_{(f, x)} y = f(y), \text{ for } (f, x) \in H \text{ and } y \in Q.
\end{align*}
Let $\pi_2: \mathrm{Hol} (Q, \diamond) \rightarrow Q$ be the projection onto the second factor. A subrack $H \subset \mathrm{Hol}(Q, \diamond)$ is said to be {\bf regular} if the map $\pi_2|_H: H \rightarrow Q$ is bijective. 

With the above notations, we have the following result.

\begin{thm}
    Let $(Q, \diamond)$ be a rack. Then there is a one-to-one correspondence between rack-pairing structures on $(Q, \diamond)$ and regular subracks of $\mathrm{Hol}(Q, \diamond)$.
\end{thm}

\begin{proof}
    Let $(Q, \diamond, \smallblackdiamond)$ be a rack-pairing. For any $x \in Q$, we define $L_x^\smallblackdiamond = x \smallblackdiamond - \in \mathrm{Aut} (Q, \diamond)$. Then it is easy to see that
    \begin{align*}
        H= \{  (L_x^\smallblackdiamond , x) | x \in Q   \}
    \end{align*}
    is a subrack of $\mathrm{Hol}(Q, \diamond)$. Further, $H \subset \mathrm{Hol}(Q, \diamond)$ is obviously regular.

    Conversely, let $H$ be a regular subrack of $\mathrm{Hol}(Q, \diamond)$. Let $x \in Q$ be any element and let $(\pi_2|_H)^{-1} (x) = (f, x) \in H$. We define a binary operation $\smallblackdiamond: Q \times Q \rightarrow Q$ by $x \smallblackdiamond y:= f(y)$, for $x, y \in Q$. Let $x, y, z \in Q$ be arbitrary. Suppose $(\pi_2|_H)^{-1} (x) = (f, x) \in H$ and $(\pi_2|_H)^{-1} (y) = (g, y) \in H$. Then we have
    \begin{align*}
    x \smallblackdiamond ( y \smallblackdiamond z) = x \smallblackdiamond g(z) = fg (z).
    \end{align*}
    On the other hand,
    \begin{align*}
        (x \smallblackdiamond y ) \smallblackdiamond (x \smallblackdiamond z) = f(y) \smallblackdiamond f (z) =~& (fgf^{-1}) (f (z))  \quad (\because (f,x), (g, y) \in H \text{ implies } (fgf^{-1}, f(y)) \in H) \\
        =~& fg (z).
    \end{align*}
    Thus $\smallblackdiamond$ is left-distributive. Moreover, the equation $x \smallblackdiamond c = y$ has a unique solution for $c$ (namely $c = f^{-1} (y)$). Hence $(Q, \smallblackdiamond)$ is a rack. Finally,
    \begin{align*}
        x \smallblackdiamond (y \diamond z) = f (y \diamond z) = f(y) \diamond f(z) = (x \smallblackdiamond y) \diamond (x \smallblackdiamond z)
    \end{align*}
    which proves the compatibility. So $(Q, \diamond , \smallblackdiamond)$ is a rack-pairing.
\end{proof}

\medskip

\noindent {\bf Some particular classes of rack-pairings.} Here we consider $\lambda$-homomorphic and symmetric rack-pairings. Given a rack $(Q, \diamond)$ and an averaging operator $\mathcal{A} : Q \rightarrow Q$, we will discuss the condition under which the induced rack-pairing $(Q, \diamond, \diamond_\mathcal{A})$ is $\lambda$-homomorphic (resp. symmetric).

%
%These are respectively generalization of $\lambda$-homomorphic skew braces and symmetric skew braces considered in \cite{bar-manoj}. 

\begin{prop}
    Let $(Q, \diamond, \smallblackdiamond)$ be a rack-pairing. Then the map
    \begin{align*}
        \lambda : (Q, \smallblackdiamond) \rightarrow \mathrm{Aut}(Q, \diamond), \lambda (x) = L_{x}^{ \smallblackdiamond} 
    \end{align*}
    is a homomorphism of racks, where we equip $\mathrm{Aut}(Q, \diamond)$ with the conjugation rack structure induced from the group structure.
\end{prop}

\begin{proof}
    Since $(Q, \diamond, \smallblackdiamond)$ is a rack-pairing, it follows from the compatibility condition that $L_x^\smallblackdiamond \in \mathrm{Aut} (Q, \diamond)$. Moreover, we have $\lambda (x \smallblackdiamond y) = L^\smallblackdiamond_{x \smallblackdiamond y} = L_x^\smallblackdiamond L_y^\smallblackdiamond (L_x^\smallblackdiamond)^{-1}$, for any $x, y \in Q$. This proves the result.
\end{proof}

Note that the map $\lambda$ can also be viewed as a map $\lambda : (Q, \diamond) \rightarrow \mathrm{Aut} (Q, \diamond)$, which, in general, may not be a rack homomorphism.

\begin{defn}
A rack-pairing $(Q, \diamond, \smallblackdiamond)$ is said to be {\bf $\lambda$-homomorphic} (resp. {\bf $\lambda$-anti-homomorphic}) if the map $\lambda : (Q, \diamond) \rightarrow \mathrm{Aut} (Q, \diamond), \lambda (x) = L_{x}^{ \smallblackdiamond}$ is a rack homomorphism (resp. rack anti-homomorphism).
\end{defn}

Thus, it follows that a rack-pairing $(Q, \diamond, \smallblackdiamond)$ is $\lambda$-homomorphic if
\begin{align*}
    L^\smallblackdiamond_{x \diamond y} = L_x^\smallblackdiamond L_y^\smallblackdiamond (L_x^\smallblackdiamond)^{-1} \quad \big( \text{equivalently, } x \smallblackdiamond (y \smallblackdiamond z) = (x \diamond y) \smallblackdiamond (x \smallblackdiamond z) \big), \text{ for } x, y, z \in Q.
\end{align*}
On the other hand, $(Q, \diamond, \smallblackdiamond)$ is $\lambda$-anti-homomorphic if
\begin{align*}
    L^\smallblackdiamond_{y \diamond x} = L_x^\smallblackdiamond L_y^\smallblackdiamond (L_x^\smallblackdiamond)^{-1} \quad \big( \text{equivalently, } x \smallblackdiamond (y \smallblackdiamond z) = (y \diamond x) \smallblackdiamond (x \smallblackdiamond z) \big), \text{ for } x, y, z \in Q.
\end{align*}

\begin{defn}
A rack-pairing $(Q, \diamond, \smallblackdiamond)$ is called {\bf symmetric} if $(Q, \smallblackdiamond, \diamond)$ is also a rack-pairing, i.e.
\begin{align}\label{symmetric-rack}
    x \diamond (y \smallblackdiamond z) = (x \diamond y) \smallblackdiamond (x \diamond z) ~~ \text{ also holds, for any } x, y, z \in Q.
\end{align}
\end{defn}

 It follows from (\ref{symmetric-rack}) that a rack-pairing $(Q, \diamond, \smallblackdiamond)$ is symmetric if and only if $L^\smallblackdiamond_{x \diamond y} = L_x^\diamond L_y^\smallblackdiamond (L_x^\diamond)^{-1}$, for all $x, y \in Q$.

\begin{remark}
 The notion of multi-racks is a generalization of racks (see, for instance, \cite{turaev}). Recall that a multi-rack labelled by a nonempty set $S$ is a nonempty set $Q$ equipped with a collection $\{ \diamond_s : Q \times Q \rightarrow Q \}_{s \in S}$ of binary operations such that the left multiplication $L_x^{\diamond_s} = x \diamond_s -$ is bijective on $Q$ (for all $x \in Q$, $s \in S$) and satisfies
 \begin{align*}
     x \diamond_s (y \diamond_t z) = (x \diamond_s y) \diamond_t (x \diamond_s z), \text{ for } x, y, z \in Q \text{ and } s, t \in S.
 \end{align*}
It follows that a multi-rack labelled by a set $S$ consisting of two elements is nothing but a symmetric rack-pairing.
\end{remark}

%\begin{prop}
 %   Every $\lambda$-anti-homomorphic rack-pairing is symmetric.
%\end{prop}

%\begin{proof}   
%\end{proof}

%Let $(Q, \diamond, \smallblackdiamond)$ be a skew-rack. It is said to be $\lambda$-homomorphic if the map $\lambda : (Q, \diamond) \rightarrow \mathrm{Aut}(Q, \diamond)$, $\lambda (x) = x \smallblackdiamond -$ is a homomorphism of racks, i.e.,
%\begin{align*}
%    \lambda (x \diamond y) = \lambda (x) \lambda (y) \lambda(x)^{-1}  ~~~(\text{Equivalently, } (x \diamond y) \smallblackdiamond (x \smallblackdiamond z) = x \smallblackdiamond (y \smallblackdiamond z), \text{ for all } x, y, z \in Q).
%\end{align*}

\begin{prop}
    Let $(Q, \diamond)$ be a rack and $\mathcal{A}: Q \rightarrow Q$ be an averaging operator. Then the induced rack-pairing $(Q, \diamond, \diamond_\mathcal{A})$ is $\lambda$-homomorphic if and only if
    \begin{align*}
        L^\diamond_{\mathcal{A} (x \diamond y)} = L^\diamond_{  (\mathcal{A} (x) \diamond \mathcal{A} (y))}, \text{ for all } x, y \in Q.
    \end{align*}
    In particular, if $\mathcal{A}$ is also a rack homomorphism then $(Q, \diamond, \diamond_\mathcal{A})$ is $\lambda$-homomorphic.
\end{prop}

\begin{proof}
    Note that $(Q, \diamond, \diamond_\mathcal{A})$ is $\lambda$-homomorphic if and only if $ x \diamond_\mathcal{A} (y \diamond_\mathcal{A} z) = (x \diamond y) \diamond_\mathcal{A} (x \diamond_\mathcal{A} z),$ for all $x , y, z \in Q$. This is equivalent to the condition that
    \begin{align*}
       (\mathcal{A} (x) \diamond \mathcal{A} (y)) \diamond ( \mathcal{A} (x) \diamond z    ) =  \mathcal{A} (x \diamond y ) \diamond ( \mathcal{A} (x) \diamond z) , \text{ for } x, y, z \in Q.
    \end{align*}
    This holds if and only if $ L^\diamond_{  (\mathcal{A} (x) \diamond \mathcal{A} (y))} = L^\diamond_{\mathcal{A} (x \diamond y)} $, for all $x, y \in Q.$ The last part is straightforward.
\end{proof}

%Let $(Q, \diamond)$ be a rack. A {\bf module} over $(Q, \diamond)$ is a set $X$ with a map $\Phi : Q \times X \rightarrow X,(x, a) \mapsto \Phi_x a$ such that $\Phi_x : X \rightarrow X$ is bijective for all $x \in Q$ ans satisfies $\Phi_x \Phi_y = \Phi_{x \diamond y} \Phi_x$, for all $x, y \in Q$. In this case, we also say that the rack $(Q, \diamond)$ acts on the set $X$ by the maps $\Phi_x$'s. 
%Given a module $(X, \Phi)$, a set-map $\pi : X \rightarrow Q$ is said to be {\bf equivariant} if it satisfies
%\begin{align*}
%    \pi (\Phi_x a) = x \diamond \pi (a), \text{ for } x \in Q, a \in X.
%\end{align*}

%\begin{remark}
%     Let $(Q, \diamond_Q)$ and $ (R, \diamond_R)$ be two racks. Suppose the rack $(Q, \diamond_Q)$ acts on $R$ by a map $\Phi : Q \times R \rightarrow R$ such that for each $x \in Q$, the map $\Phi_x : R \rightarrow R$ is a rack automorphism. Let $\mathcal{A}: R \rightarrow Q$ be a map that satisfies $\mathcal{A} (x) \diamond_Q \mathcal{A} (y) = \mathcal{A} ( \Phi_{\mathcal{A}(x) }y)$, for all $x, y \in R$. Then the binary operation $\diamond_\mathcal{A} : R \times R \rightarrow R$ defined by $x \diamond_\mathcal{A} y := \Phi_{\mathcal{A}(x) }y$ is a rack structure on $R$. Moreover, the triple $(R, \diamond_R, \diamond_\mathcal{A})$ is a rack-pairing. This rack-pairing is $\lambda$-homomorphic if and only if $\Phi_{\mathcal{A} (x \diamond_R y)} = \Phi_{\mathcal{A} (x) \diamond_Q \mathcal{A} (y)}$, for all $x, y \in R$. 
%\end{remark}

\begin{prop}
    Let $(Q, \diamond)$ be a rack and $\mathcal{A}: Q \rightarrow Q$ be an averaging operator. Then the induced rack-pairing $(Q, \diamond, \diamond_\mathcal{A})$ is symmetric if and only if $L^\diamond_{\mathcal{A} (x \diamond y)} = L^\diamond_{x \diamond \mathcal{A} (y)}$, for all $x, y \in Q$. In particular, if $\mathcal{A} : Q \rightarrow Q$ is an equivariant map then $(Q, \diamond, \diamond_\mathcal{A})$ is symmetric.
\end{prop}

\begin{proof}
    Note that, for any $x, y, z \in Q$, we have
    \begin{align*}
        x \diamond (y \diamond_\mathcal{A} z) = x \diamond (\mathcal{A} (y) \diamond z) = (x \diamond \mathcal{A} (y)) \diamond (x \diamond z)
    \end{align*}
    and
    \begin{align*}
        (x \diamond y) \diamond_\mathcal{A} (x \diamond z) = \mathcal{A} (x \diamond y) \diamond (x \diamond z).
    \end{align*}
    Hence $(Q, \diamond_\mathcal{A}, \diamond)$ is a rack-pairing (i.e. $ x \diamond (y \diamond_\mathcal{A} z) =  (x \diamond y) \diamond_\mathcal{A} (x \diamond z)$, for all $x, y, z \in Q$) if and only if $L^\diamond_{\mathcal{A} (x \diamond y)} = L^\diamond_{x \diamond \mathcal{A} (y)}$, for all $x, y \in Q$. Hence the first part follows. The last part is a consequence of the first part.
\end{proof}

\medskip

\noindent {\bf Di-racks.} Here we will introduce a new algebraic structure, called a di-rack that is intimately related to rack-pairings.

\begin{defn}
    A {\bf di-rack} is a rack $(Q, \diamond)$ equipped with a binary operation $\triangleright : Q \times Q \rightarrow Q$ that satisfy the following conditions:
\begin{itemize}
    \item[-] for each $x \in Q$, the left multiplication map $L_{x}^\triangleright = x \triangleright -  : Q \rightarrow Q$ is a rack automorphism, i.e., $L_{x}^\triangleright$ is a bijective map on $Q$ and have
    \begin{align}\label{dir1}
        x \triangleright (y \diamond z) = (x \triangleright y) \diamond (x \triangleright z), \text{ for } x, y, z \in Q,
    \end{align}
    \item[-] the operation $\triangleright$ satisfies the ``weighted'' left-distributivity as follows:
    \begin{align}\label{dir2}
        x \diamond ( x \triangleright (y \triangleright z)) = (x \diamond (x \triangleright y)) \triangleright (x \diamond (x \triangleright z)), \text{ for } x, y, z \in Q.
    \end{align}
\end{itemize}
A di-rack as above is simply denoted by the triple $(Q, \diamond, \triangleright)$.
\end{defn}

\begin{exam}
    Let $(Q, \diamond)$ be a rack. Then $(Q, \diamond, \triangleright)$ is a di-rack, $x \triangleright y := y$, for $x, y \in Q$. In particular, if $G$ is a group, then $(G, \diamond, \triangleright)$ is a di-rack, where $x \diamond y = xyx^{-1}$ and $x \triangleright y = y$.
\end{exam}

\begin{exam}
    Let $G$ be a group. Then the triple $(G, \diamond, \triangleright')$ is also a di-rack, where $x \triangleright' y := x^{-2} y x^2$, for $x, y \in G$.
\end{exam}

Let $(Q, \diamond, \triangleright)$ and $(Q', \diamond', \triangleright')$ be two di-racks. A {\bf morphism} from $(Q, \diamond, \triangleright)$ to $(Q', \diamond', \triangleright')$ is given a set-map $\varphi : Q \rightarrow Q'$ that satisfies $\varphi (x \diamond y) = \varphi (x) \diamond' \varphi (y)$ and $\varphi (x \triangleright y) = \varphi (x) \triangleright' \varphi (y)$, for all $x, y \in Q$.

Let $(Q, \diamond, \triangleright)$ be a di-rack. If $\diamond$ is the trivial rack (i.e. $x \diamond y = y$, for all $x, y \in Q$) then it follows from (\ref{dir1}) and (\ref{dir2}) that $(Q, \triangleright)$ is a rack. In general, we have the following result.

\begin{prop}
     Let $(Q, \diamond, \triangleright)$ be a di-rack. Define a binary operation $\smallblackdiamond : Q \times Q \rightarrow Q$ by
     \begin{align*}
         x \smallblackdiamond y := x \diamond (x \triangleright y), \text{ for } x, y \in Q.
     \end{align*}
     \begin{itemize}
         \item[(i)] Then $(Q, \smallblackdiamond)$ is a rack, called the sub-adjacent rack of the di-rack $(Q, \diamond, \triangleright)$. Moreover, the triple $(Q, \diamond, \smallblackdiamond)$ is a rack-pairing.
         \item[(ii)] The map $\Phi : Q \times Q \rightarrow Q$, $(x, y) \mapsto \Phi_x y = {x \diamond (x \triangleright y)}$ defines an action of the sub-adjacent rack $(Q, \smallblackdiamond)$ on the rack $(Q, \diamond)$.
     \end{itemize}
    
     Let $(Q, \diamond, \triangleright)$ and $(Q', \diamond', \triangleright')$ be two di-racks and $\varphi : Q \rightarrow Q'$ be a morphism between them. Then $\varphi$ is a morphism between the corresponding sub-adjacent racks.
\end{prop}

\begin{proof}
  (i)  Let $x, y \in Q$. Since $(Q, \diamond)$ is a rack, the map $L_{x}^\diamond = x \diamond -$ is bijective on $Q$, and hence there exists a unique element $q \in Q$ such that $x \diamond q = y$. On the other hand, the map $L_{x}^\triangleright = x \triangleright -$ is bijective implying that there exists a unique $c \in Q$ such that $x \triangleright c = q$. Hence we obtain a unique $c$ for which
    \begin{align*}
        x \smallblackdiamond c = x \diamond (x \triangleright c) = x \diamond q = y.
    \end{align*}
    This proves that the map $L_{x}^\smallblackdiamond$ is bijective on $Q$. Finally, for any $x, y, z \in Q$, we have
    \begin{align*}
        x \smallblackdiamond ( y \smallblackdiamond z) 
        =~& x \diamond \big(  x \triangleright ( y  \diamond (y \triangleright z))  \big) \\
        =~& x \diamond \big(  (x \triangleright y) \diamond ( x \triangleright (y \triangleright z)  )  \big) \\
        =~& \big(   x \diamond (x \triangleright y)  \big)   \diamond \big(  x \diamond ( x \triangleright (y \triangleright z)) \big) \\
        =~& \big(   x \diamond (x \triangleright y)  \big)   \diamond \big(  \big(   x \diamond (x \triangleright y)  \big) \triangleright \big(   x \diamond (x \triangleright z)  \big)     \big) \\
        =~& \big(   x \diamond (x \triangleright y)  \big)  \smallblackdiamond \big(   x \diamond (x \triangleright z)  \big)  \\
        =~& (x \smallblackdiamond y) \smallblackdiamond (x \smallblackdiamond z).
    \end{align*}
    This proves the left-distributivity of the operation $\smallblackdiamond$. Hence $(Q, \smallblackdiamond)$ is a rack. Moreover, we have
    \begin{align*}
        x \smallblackdiamond (y \diamond z)
        =~& x \diamond \big( x \triangleright (y \diamond z)  \big) \\
        =~& x \diamond  \big(  (x \triangleright y) \diamond (x \triangleright z)    \big) \\
        =~& \big(  x \diamond (x \triangleright y)  \big) \diamond \big( x \diamond (x \triangleright z)   \big) \\
        =~& (x \smallblackdiamond y) \diamond (x \smallblackdiamond z).
    \end{align*}
    This proves the compatibility condition of rack-pairing.

    (ii) Note that for any $x \in Q$, the map $\Phi_x = L_x^\smallblackdiamond$ is an automorphism of the rack $(Q, \diamond)$. Moreover, for any $x, y \in Q$, we have $\Phi_{x \smallblackdiamond y} = \Phi_x \Phi_y \Phi_x^{-1}$ as the operation $\smallblackdiamond$ defines a rack structure on $Q$. Hence the result follows.

    Finally, if $\varphi : Q \rightarrow Q'$ is a morphism of di-racks from $(Q, \diamond, \triangleright)$ to $(Q', \diamond', \triangleright')$ then
    \begin{align*}
        \varphi (x \smallblackdiamond y) = \varphi (x \diamond (x \triangleright y)) = \varphi (x) \diamond' \varphi (x \triangleright y) = \varphi (x) \diamond'  ( \varphi (x) \triangleright' \varphi (y)) = \varphi (x) \smallblackdiamond' \varphi (y),
    \end{align*}
    for all $x, y \in Q$. This proves the last part.
\end{proof}

\begin{prop}\label{rack-pair-dirack}
    Let $(Q, \diamond, \smallblackdiamond)$ be a rack-pairing. Define a binary operation $\triangleright : Q \times Q \rightarrow Q$ by
    \begin{align*}
        x \triangleright y := (L_x^\diamond)^{-1} (x \smallblackdiamond y), \text{ for } x, y \in Q. 
    \end{align*}
    Then $(Q, \diamond, \triangleright)$ is a di-rack. Moreover, if $\varphi : Q \rightarrow Q'$  is a morphism of rack-pairings from $(Q, \diamond, \smallblackdiamond)$ to $(Q', \diamond', \smallblackdiamond')$, then $\varphi$ is also a morphism of di-racks from $(Q, \diamond, \triangleright)$ to $(Q', \diamond', \triangleright')$.
\end{prop}

\begin{proof}
    Let $x, y \in Q$ be arbitrary. Since $\smallblackdiamond$ defines a rack operation on $Q$, there exists a unique $c \in Q$ such that $x \smallblackdiamond c = x \diamond y$. In other words,
    \begin{align*}
        x \triangleright c = (L_x^\diamond)^{-1} (x \smallblackdiamond c) = (L_x^\diamond)^{-1} (x \diamond y) = y.
    \end{align*}
    This shows that the map $L_x^\triangleright: Q \rightarrow Q$ is bijective on $Q$. For any $x, y, z \in Q$, we have
    \begin{align*}
        x \triangleright (y \diamond z) =~& (L_x^\diamond)^{-1} \big(  x \smallblackdiamond (y \diamond z)   \big) \\
        =~& (L_x^\diamond)^{-1} \big(  (x \smallblackdiamond y) \diamond (x \smallblackdiamond z) \big) \\
        =~& \big(  (L_x^\diamond)^{-1} (x \smallblackdiamond y)  \big) \diamond \big(  (L_x^\diamond)^{-1} (x \smallblackdiamond z)  \big) = (x \triangleright y) \diamond (x \triangleright z).
    \end{align*}
    Hence the left multiplication map $L_x^\triangleright: Q \rightarrow Q$ is a rack automorphism for $(Q, \diamond)$. Next, we observe that
    \begin{align}\label{compare1}
        x \smallblackdiamond (y \smallblackdiamond z) =~& x \diamond \big(  x \triangleright ( y \diamond (y \triangleright z)) \big) \nonumber \\
        =~& x \diamond \big(    (x \triangleright y) \diamond ( x \triangleright (y \triangleright z)) \big) \quad (\because ~ L_x^\triangleright \text{ is a rack automorphism})  \nonumber \\
        =~& \big(  x \diamond (x \triangleright y)  \big) \diamond \big(  x \diamond (x \triangleright (y \triangleright z))   \big).
    \end{align}
    On the other hand,
    \begin{align}\label{compare2}
        (x \smallblackdiamond y) \smallblackdiamond (x \smallblackdiamond z) =~& (x \diamond (x \triangleright y))  \smallblackdiamond (x \diamond (x \triangleright z)) \nonumber \\
        =~& (x \diamond (x \triangleright y)) \diamond \big(  ( x \diamond (x \triangleright y) )\triangleright (x \diamond (x \triangleright z))   \big).
    \end{align}
    Since $\smallblackdiamond$ is left-distributive, it follows from (\ref{compare1}) and (\ref{compare2}) that 
    \begin{align*}
        \big(  x \diamond (x \triangleright y)  \big) \diamond \big(  x \diamond (x \triangleright (y \triangleright z))   \big) = (x \diamond (x \triangleright y)) \diamond \big(  ( x \diamond (x \triangleright y) )\triangleright (x \diamond (x \triangleright z))   \big).
    \end{align*}
    By applying the operation $\big( L_{ x \diamond (x \triangleright y)}   \big)^{-1}$ to both sides of the above identity, we obtain the weighted left-distributivity. This proves that $(Q, \diamond, \triangleright)$ is a di-rack.

    Finally, if $\varphi : Q \rightarrow Q'$ is a morphism of rack-pairings from $(Q, \diamond, \smallblackdiamond)$ to $(Q', \diamond', \smallblackdiamond')$ then we have
    \begin{align*}
        \varphi (x \triangleright y) = \varphi \big(  (L_x^\diamond)^{-1} (x \smallblackdiamond y)   \big) = (  L_{\varphi (x)}^{\diamond'})^{-1} ( \varphi (x) \smallblackdiamond' \varphi (y)) = \varphi (x) \triangleright' \varphi(y),
    \end{align*}
    for all $x, y \in Q$. Hence $\varphi$ is a morphism of di-racks from $(Q, \diamond, \triangleright)$ to $(Q', \diamond', \triangleright')$.
\end{proof}

Using the above construction, one may find many examples of di-racks some of which are listed below.

\begin{prop}
    Let $(G, \cdot, \smallblackdiamond)$ be a pointed group-rack. Then $(G, \diamond, \triangleright)$ is a di-rack, where
    \begin{align*}
        x \diamond y = x yx^{-1} ~~~ \text{ and } ~~~ x \triangleright y = x^{-1} (x \smallblackdiamond y) x, \text{ for } x, y \in G.
    \end{align*}
\end{prop}

\begin{proof}
    Since $(G, \cdot, \smallblackdiamond)$ is a pointed group-rack, it follows from Proposition \ref{gpr-skr} that $(G, \diamond, \smallblackdiamond)$ is a rack-pairing, where $x \diamond y = x y x^{-1}$. Hence by Proposition \ref{rack-pair-dirack} that $(G, \diamond, \triangleright)$ is a di-rack, where $x \triangleright y = (L_x^\diamond)^{-1} (x \smallblackdiamond y) = x^{-1} (x \smallblackdiamond y) x$, for $x, y \in G$.
\end{proof}

\begin{prop}
    Let $(Q, \diamond)$ be a rack and $\mathcal{A} : Q \rightarrow Q$ be an averaging operator on $Q$. Then $(Q, \diamond, \triangleright_\mathcal{A})$ is a di-rack, where the operation $\triangleright_\mathcal{A} : Q \times Q \rightarrow Q$ is implicitly defined by the identity
    \begin{align*}
        x \diamond (x \triangleright_\mathcal{A} y) := \mathcal{A} (x) \diamond y, \text{ for } x, y \in Q.
    \end{align*}
\end{prop}

\begin{proof}
    Since $\mathcal{A} : Q \rightarrow Q$ is an averaging operator on the rack $(Q, \diamond)$, it follows from Proposition \ref{avg-rack-pair} that the triple $(Q, \diamond, \diamond_\mathcal{A})$ is a rack-pairing, where $x \diamond_\mathcal{A} y := \mathcal{A}(x) \diamond y$, for $x, y \in Q$. Hence it follows from Proposition \ref{rack-pair-dirack} that $(Q, \diamond, \triangleright_\mathcal{A})$ is a di-rack, where $x \triangleright_\mathcal{A} y := (L_x^\diamond)^{-1} (x \diamond_\mathcal{A} y) = (L_x^\diamond)^{-1} (\mathcal{A} (x) \diamond y)$, for $x, y \in Q$. This proves the result.
\end{proof}

As a consequence, we get the following.

\begin{prop}
    Let $G$ be a group and $\mathcal{A}: G \rightarrow G$ be an averaging operator on $G$. Then $(G, \diamond, \triangleright_\mathcal{A})$ is a di-rack, where
    \begin{align*}
        x \diamond y := x y x^{-1} ~~~ \text{ and } ~~~ x \triangleright_\mathcal{A} y := x^{-1} \mathcal{A} (x) y \mathcal{A}(x)^{-1} x, \text{ for } x, y \in G.
    \end{align*}
\end{prop}

\section{Averaging operators on cocommutative Hopf algebras and braided vector spaces}\label{sec5}

Till now, we have considered averaging operators on (Lie) groups and racks. Since groups are closely related to cocommutative Hopf algebras, we may also define averaging operators on a cocommutative Hopf algebra $H$ (see the next definition). Then we show that there is a one-to-one correspondence between averaging operators on a group $G$ and averaging operators on $H= {\bf k}[G]$.

Let $(H, \mu, \eta, \Delta, \epsilon, S)$ be a cocommutative Hopf algebra. 

\begin{defn}
    An {\bf averaging operator} on $H$ is a coalgebra map $B: H \rightarrow H$ that satisfies
    \begin{align}\label{avg-coc}
        B(x_{(1)}) B(y) S (    B (x_{(2)})) = B \big(  B(x_{(1)}) y S (    B (x_{(2)}))   \big), \text{ for all } x, y \in H.
    \end{align}
\end{defn}

\begin{exam}
    The identity map $\mathrm{Id} : H \rightarrow H$ is an averaging operator on $H$.
\end{exam}

\begin{exam}
    If $H$ is involutive (i.e. the antipode $S$ satisfies $S^2 = \mathrm{Id}$) and $S$ is also a coalgebra map then $S$ is an averaging operator on $H$. 
    %Note that a finite-dimensional semisimple Hopf algebra that is also cocommutative is involutive. Hence the antipode is an averaging operator.
\end{exam}

\begin{exam}
    Let $G$ be a finite abelian group and let $H = {\bf k}^G$ be the Hopf algebra of functions from $G$ to ${\bf k}$. Then any coalgebra map $B : H \rightarrow H$ is an averaging operator on $H$. 
\end{exam}

\begin{thm}
Let $G$ be a group and $\mathcal{A} : G \rightarrow G$ be an averaging operator on $G$. Then the map $\mathcal{A}$ can be uniquely extended to an averaging operator $B: {\bf k}[G] \rightarrow {\bf k}[G]$ on the group algebra ${\bf k}[G]$ viewed as a cocommutative Hopf algebra. Conversely, for any group $G$, if $B: {\bf k}[G] \rightarrow {\bf k}[G]$ is an averaging operator on the group algebra ${\bf k}[G]$, then $ \mathcal{A} = B|_G$ is an averaging operator on the group $G$.
\end{thm}

\begin{proof}
    We define the map $B : {\bf k}[G] \rightarrow {\bf k}[G]$ by
    \begin{align*}
        B (\sum \alpha_i g_i) = \sum \alpha_i \mathcal{A} (g_i), \text{ for } \sum \alpha_i g_i \in {\bf k}[G].
    \end{align*}
    The map $B: {\bf k}[G] \rightarrow {\bf k}[G]$ is a coalgebra map. Finally, for any $x = \sum \alpha_i g_i \in {\bf k}[G]$ and $y \in G$, we have
    \begin{align*}
        B (x_{(1)}) B(y) S (B (x_{(2)})) =~& \alpha_i B (g_i) B(y) S (B (g_i)) \\
        =~& \alpha_i \mathcal{A} (g_i) \mathcal{A}(y) \mathcal{A} (g_i)^{-1} \\
        =~& \alpha_i \mathcal{A} \big(  \mathcal{A} (g_i) y \mathcal{A} (g_i)^{-1} \big) \\
        =~& \alpha_i B \big(   B(g_i) y S (B(g_i)) \big) = B \big(  B(x_{(1)}) y S (B (x_{(2)})) \big).
    \end{align*}
Since the elements of $G$ form a linear basis for the vector space ${\bf k}[G]$, it follows that the identity (\ref{avg-coc}) holds for all $x, y \in {\bf k}[G]$.

Conversely, let $G$ be a group and $B: {\bf k}[G] \rightarrow {\bf k}[G]$ be an averaging operator on the group algebra ${\bf k}[G]$. Let $g \in G$. Since $B : {\bf k}[G] \rightarrow {\bf k}[G]$ is a coalgebra map, we have
\begin{align}\label{coalg-map}
    \Delta (B(g)) = (B \otimes B) \Delta (g) = B (g) \otimes B(g).
\end{align}
On the other hand, since $\epsilon (B(g)) = \epsilon (g) = 1$, we have $B(g) \neq 0$. Hence it follows from (\ref{coalg-map}) that $B(g)$ is a group-like element.
In other words, $B(g) \in G$. Finally, for any $g, h \in G$, we have
\begin{align*}
    \mathcal{A} (g) \mathcal{A} (h) \mathcal{A} (g)^{-1} = B(g) B(h) S (B(g)) = B \big(  B(g) h S (B(g)) \big) = \mathcal{A} \big( \mathcal{A} (g) h \mathcal{A}(g)^{-1} \big).
\end{align*}
This shows that $\mathcal{A} = B|_G$ is an averaging operator on the group $G$. 
\end{proof}

On the other hand, groups and racks are also related to set-theoretical solutions of the Yang-Baxter equation. Recall that a set-theoretical solution of the Yang-Baxter equation gives rise to a braided vector space \cite{etingof}. Here we will define averaging operators on braided vector spaces and braided sets and find their relations with averaging operators on Lie algebras, groups and racks.

A {\bf braided vector space} is a pair $(V, S)$ consisting of a vector space $V$ equipped with a linear map $S: V \otimes V \rightarrow V \otimes V$ satisfying the following Yang-Baxter equation (or braided equation):
\begin{align}\label{yb-eqn}
    (\mathrm{Id} \otimes S) (S \otimes \mathrm{Id}) (\mathrm{Id} \otimes S)  =   (S \otimes \mathrm{Id})  (\mathrm{Id} \otimes S) (S \otimes \mathrm{Id}).
\end{align}
Note that if $X$ is a basis of the vector space $V$ and $r : X \times X \rightarrow X \times X$ is a map such that
%A {\bf set-theoretical solution of the Yang-Baxter equation} on $X$ is a set-map $S: X \times X \rightarrow X \times X$ satisfying
\begin{align}
    (\mathrm{Id} \times r) (r \times \mathrm{Id}) (\mathrm{Id} \times r)  =   (r \times \mathrm{Id})  (\mathrm{Id} \times r) (r \times \mathrm{Id})
\end{align}
holds, then $r$ induces a solution of the Yang-Baxter equation. In this case, $r$ is called a {\bf set-theoretical solution of the Yang-Baxter equation} on $X$ and the pair $(X, r)$ is called a {\bf braided set}.

%Finding new solutions of the Yang-Baxter equation and the set-theoretical solutions of the Yang-Baxter equation is a central open problem in mathematics and mathematical physics.

In the following, we introduce averaging operators on braided vector spaces and braided sets. We show that they are closely related to averaging operators on Lie algebras and racks.

\begin{defn}
(i) Let $(V, S)$ be a braided vector space. A {\bf (left) averaging operator} on $(V, S)$ is a linear map $\mathcal{A} : V \rightarrow V$ satisfying
\begin{align}\label{braided-avg-1}
    S (\mathcal{A}(x) \otimes \mathcal{A} (y) ) = (\mathcal{A} \otimes \mathrm{Id}) S ( \mathcal{A} (x) \otimes  y), \text{ for } x, y \in V.
\end{align}

(ii) Let $(X, r)$ be a braided set. A {\bf (left) averaging operator} on $(X,r)$ is a set-map $\mathcal{A}: X \rightarrow X$ that satisfies
\begin{align}\label{braided-avg}
    r (\mathcal{A}(x) , \mathcal{A} (y) ) = (\mathcal{A} \times \mathrm{Id}) r ( \mathcal{A} (x) ,  y), \text{ for } x, y \in X.
\end{align}
\end{defn}

\begin{remark}
    It follows from (\ref{braided-avg}) that if $\mathcal{A}: X \rightarrow X$ is an averaging operator on the braided set $(X, r)$ then for any $k \geq 0$, the map $\mathcal{A}^k: X \rightarrow X$ is also an averaging operator on $(X, r)$.
\end{remark}

In \cite{baez-crans} the authors have shown that Lie algebras are closely related to braided vector spaces. Let $\mathfrak{g}$ be a vector space equipped with a skew-symmetric bilinear bracket $[~, ~] : \mathfrak{g} \times \mathfrak{g} \rightarrow \mathfrak{g}$. Consider the vector space $V = {\bf k} \oplus \mathfrak{g}$ and define a linear map $S: V \otimes V \rightarrow V \otimes V$ given by
\begin{align}\label{lie-yb}
    S \big(  (a, x) \otimes (b, y) \big) = (b, y) \otimes (a, x) + (0, [x, y]) \otimes (1, 0),
\end{align}
for $(a, x), (b, y) \in {\bf k} \oplus \mathfrak{g} = V$. Then $S$ satisfies the Yang-Baxter equation (\ref{yb-eqn}) if and only if $(\mathfrak{g}, [~,~])$ is a Lie algebra.

\begin{prop}
    Let $(\mathfrak{g}, [~,~])$ be a Lie algebra and consider the corresponding braided vector space $({\bf k} \oplus \mathfrak{g}, S)$, where $S$ is given by (\ref{lie-yb}). A linear map $P: \mathfrak{g} \rightarrow \mathfrak{g}$ is an averaging operator on the Lie algebra $\mathfrak{g}$ if and only if the linear map ${\mathcal{A}}: {\bf k} \oplus \mathfrak{g} \rightarrow {\bf k} \oplus \mathfrak{g}$ defined by ${\mathcal{A}} (a, x) = (a, P(x))$ is an averaging operator on the braided vector space $({\bf k} \oplus \mathfrak{g}, S)$.
\end{prop}

\begin{proof}
    We have
    \begin{align*}
        S \big( {\mathcal{A}} (a, x) \otimes {\mathcal{A}} (b, y)   \big) =~& S \big( (a, P(x)) \otimes (b, P (y))   \big) \\
        =~&  (b, P (y)) \otimes  (a, P (x)) + (0, [ P(x), P(y)]) \otimes (1, 0).
    \end{align*}
    On the other hand,
    \begin{align*}
        ({\mathcal{A}} \otimes \mathrm{Id}) S \big( {\mathcal{A}} (a, x) \otimes (b, y)  \big) =~& ({\mathcal{A}} \otimes \mathrm{Id}) S \big(  (a, P(x)) \otimes (b, y)   \big) \\
        =~& (b, P (y)) \otimes  (a, P (x)) + (0, P [ P(x), y]) \otimes (1, 0).
    \end{align*}
    Hence the result follows by comparing the above two expressions.
\end{proof}

Let $G$ be a group. Then the map $r : G \times G \rightarrow G \times G$ defined by $r (x, y) = (xyx^{-1}, x)$ is a set-theoretical solution of the Yang-Baxter equation on $G$. Then a map $\mathcal{A} : G \rightarrow G$ is an averaging operator on the group $G$ if and only if $\mathcal{A}$ is an averaging operator on the braided set $(G, r)$. This result can be generalized to the context of racks as follows.

Let $Q$ be a nonempty set equipped with a binary operation $\diamond : Q \times Q \rightarrow Q$. Suppose for each $x \in Q$, the left multiplication map $L_x^\diamond = x \diamond - : Q \rightarrow Q$ is invertible. We define a map $r: Q \times Q \rightarrow Q \times Q$ by $r (x, y) = (x \diamond y, x)$, for $x, y \in Q$. Then $r$ is a set-theoretical solution of the Yang-Baxter equation on $Q$ (in other words, $(Q, r)$ is a braided set) if and only if $(Q, \diamond)$ is a rack, i.e. $\diamond$ satisfies the left-distributivity.

\begin{prop}
    Let $(Q, \diamond)$ be a rack and consider the corresponding braided set $(Q, r)$, where $r(x, y) = (x \diamond y, x)$, for $x, y \in Q$. A map $\mathcal{A}: Q \rightarrow Q$ is an averaging operator on the rack $Q$ if and only if $\mathcal{A}$ is an averaging operator on the braided set $(Q, r)$.
\end{prop}

\begin{proof}
    For any $x, y \in Q$, we observe that
    \begin{align*}
        r (\mathcal{A} (x), \mathcal{A} (y)) = \big( \mathcal{A} (x) \diamond \mathcal{A}(y), \mathcal{A}(x) \big) \quad \text{ and } \quad (\mathcal{A} \times \mathrm{Id}) r (\mathcal{A}(x), y) = \big(    \mathcal{A} ( \mathcal{A} (x) \diamond y) , \mathcal{A} (x) \big).
    \end{align*}
    It follows that $\mathcal{A}$ is an averaging operator on the rack $Q$ if and only if it is an averaging operator on the braided set $(Q, r)$.
\end{proof}

In \cite{guar} Guarnieri and Vendramin have shown that a skew brace (in particular, a two-sided skew brace) gives rise to a set-theoretical solution of the Yang-Baxter equation. In their approach, they have shown that a skew brace yields a braided group and hence provides a set-theoretical solution using a result of \cite{etingof}. Recall that a braided group is nothing but a matched pair of groups satisfying an extra assumption. In a future project, we first aim to define the notion of a matched pair of racks and the concept of braided racks. Since rack-pairings are related to two-sided skew braces via conjugation, we expect that a rack-pairing may yield a braided rack which could give a set-theoretical solution of the Yang-Baxter equation.

\section{Di-Leibniz algebras and averaging operators on Leibniz algebras}\label{sec6}
In this section, we first introduce the notion of a di-Leibniz algebra and provide some examples. Next, we introduce averaging operators and relative averaging operators on Leibniz algebras. We show that an averaging operator or a relative averaging operator induces a di-Leibniz algebra. Finally, we show that any di-Leibniz algebra can be embedded into an averaging Leibniz algebra.

\begin{defn}
    A {\bf Leibniz algebra} is a pair $(\mathfrak{h}, \{ ~, ~ \})$ consisting of a vector space $\mathfrak{h}$ equipped with a bilinear bracket (called the Leibniz bracket) $\{ ~, ~ \}: \mathfrak{h} \times \mathfrak{h} \rightarrow \mathfrak{h}$ satisfying the following Leibniz identity:
    \begin{align}\label{leib-id}
        \{ x, \{ y, z \} \} = \{ \{ x, y \}, z \} + \{ y, \{x, z \}\}, \text{ for } x, y , z \in \mathfrak{h}.
    \end{align}
    We often denote a Leibniz algebra as above simply by $\mathfrak{h}$ when the Leibniz bracket is clear from the context. 
\end{defn}

A Leibniz algebra as above is called a left Leibniz algebra as the identity (\ref{leib-id}) is equivalent to the fact that the left multiplications $\{ x , - \} : \mathfrak{h} \rightarrow \mathfrak{h}$ are derivations for the bracket. Similarly, one can define right Leibniz algebra as a pair $(\mathfrak{h}, \{ ~, ~ \})$ of a vector space $\mathfrak{h}$ with a bilinear bracket such that the right multiplications $\{ - , x \}$ are derivations for the bracket. There is a one-to-one correspondence between left and right Leibniz algebras. More precisely, a pair $(\mathfrak{h}, \{ ~, ~ \})$ is a left Leibniz algebra if and only if $(\mathfrak{h}, \{ ~, ~ \}^\mathrm{op})$ is a right Leibniz algebra, where $\{ x, y \}^\mathrm{op} = \{ y, x \}$, for $x, y \in \mathfrak{h}$. In general, any results about left Leibniz algebras can be easily generalized to right Leibniz algebras without much work.

\begin{defn}
    Let $(\mathfrak{h}, \{ ~, ~ \})$ be a Leibniz algebra. A {\bf representation} of this Leibniz algebra consists of a triple $(V, \rho^L, \rho^R)$ in which $V$ is a vector space and $\rho^L : \mathfrak{h} \times V \rightarrow V$, $\rho^R : V \times \mathfrak{h} \rightarrow V$ are bilinear maps (called the left and right $\mathfrak{h}$-actions, respectively) satisfying for $x, y \in \mathfrak{h}$ and $v \in V$,
    \begin{align*}
        \rho^L (x, \rho^L (y, v)) =~& \rho^L ( \{ x, y \}, v) + \rho^L (y, \rho^L (x, v)),\\
        \rho^L (x, \rho^R (v, y)) =~& \rho^R (\rho^L (x, v), y) + \rho^R (v, \{ x, y \}), \\
        \rho^R (v, \{ x, y \}) =~& \rho^R (\rho^R (v, x), y) + \rho^L (x, \rho^R (v,y)).
    \end{align*}
\end{defn}

Let $(\mathfrak{h}, \{ ~, ~ \})$ be a Leibniz algebra. Then $(\mathfrak{h}, \rho^L, \rho^R)$ is a representation, where both the left and right $\mathfrak{h}$-actions are given by the Leibniz bracket, i.e. $\rho^L (x, y) = \rho^R (x, y) = \{ x, y \}$, for $x, y \in \mathfrak{h}$. This is called the {\bf adjoint representation}.

In \cite{loday-di} Loday introduced the notion of a diassociative algebra (also called dialgebra) in the study of Leibniz algebras. He has shown that dissociative algebras are closely related to the combinatorics of planar binary trees. A diassociative algebra is a vector space equipped with two bilinear operations satisfying five associative-style identities. Inspired by his work, we now introduce the notion of a di-Leibniz algebra.

%Similar to a diassociative algebra, a di-Leibniz algebra is given by two binary operations satisfying five Leibniz-style identities.

\begin{defn}
    A {\bf (left) di-Leibniz algebra} is a triple $(\mathfrak{d}, \{ ~, ~ \}_\dashv , \{ ~, ~ \}_{\vdash})$ consisting of a vector space $\mathfrak{d}$ with two bilinear maps $\{ ~, ~ \}_\dashv , \{ ~, ~ \}_{\vdash} : \mathfrak{d} \times \mathfrak{d} \rightarrow \mathfrak{d}$ satisfying the following set of identities
    \begin{align}
        \{ x , \{ y, z \}_\dashv \}_\dashv =~& \{ \{ x, y \}_\dashv , z \}_\dashv + \{ y, \{ x, z \}_\dashv \}_\vdash, \label{dl1}\\
        \{ x , \{ y, z \}_\vdash \}_\dashv =~& \{ \{ x, y \}_\dashv , z \}_\dashv + \{ y, \{ x, z \}_\dashv \}_\vdash, \label{dl2}\\
        \{ x , \{ y, z \}_\dashv \}_\vdash =~& \{ \{ x, y \}_\vdash , z \}_\dashv + \{ y, \{ x, z \}_\dashv \}_\dashv, \label{dl3}\\
        \{ x , \{ y, z \}_\vdash \}_\vdash =~& \{ \{ x, y \}_\dashv , z \}_\vdash + \{ y, \{ x, z \}_\vdash \}_\vdash, \label{dl4}\\
        \{ x , \{ y, z \}_\vdash \}_\vdash =~& \{ \{ x, y \}_\vdash , z \}_\vdash + \{ y, \{ x, z \}_\vdash \}_\vdash, \label{dl5}
    \end{align}
    for $x, y, z \in \mathfrak{d}$. A di-Leibniz algebra as above may be simply denoted by $\mathfrak{d}$ when the structure maps are clear from the context.
\end{defn}

Let  $(\mathfrak{d}, \{ ~, ~ \}_\dashv , \{ ~, ~ \}_{\vdash})$ be a di-Leibniz algebra. In presence of the identity (\ref{dl1}), the identity (\ref{dl2}) is equivalent to $  \{ x , \{ y, z \}_\dashv \}_\dashv =  \{ x , \{ y, z \}_\vdash \}_\dashv$. Using this, the identity (\ref{dl3}) can be equivalently written as
\begin{align*}
    \{ x, \{ y, z \}_\dashv \}_\vdash = \{ \{ x, y \}_\vdash , z \}_\dashv + \{ y, \{ x, z \}_\vdash \}_\dashv.
\end{align*}
On the other hand, it follows from (\ref{dl4}) and (\ref{dl5}) that $\{ \{ x, y \}_\dashv , z \}_\vdash = \{ \{ x, y \}_\vdash, z \}_\vdash$.

It is also important to note that the identity (\ref{dl5}) says that the bracket $\{ ~, ~ \}_\vdash$ is a Leibniz bracket on the vector space $\mathfrak{d}$. However, in general, $\{ ~, ~ \}_\dashv$ need not be a Leibniz bracket.

%\begin{remark}
%    Let $(\mathfrak{d}, \{ ~, ~ \}_\dashv, \{ ~, ~ \}_\vdash)$ be a di-Leibniz algebra. Then it follows from (dL1) and (dL2) that
%    \begin{align*}
%        \{ x, \{ y, z \}_\dashv \}_\dashv = \{ x, \{ y, z \}_\vdash \}_\dashv.
%    \end{align*}
%\end{remark}

\begin{exam}
    Any Leibniz algebra $(\mathfrak{h}, \{ ~, ~ \})$ can be regarded as a di-Leibniz algebra in which $\{ ~, ~ \}_\dashv = \{ ~, ~ \}_\vdash = \{ ~, ~ \}$. Conversely, a di-Leibniz algebra in which both the operations are the same is nothing but a Leibniz algebra.
\end{exam}

\begin{exam}
    A differential Leibniz algebra is a Leibniz algebra $(\mathfrak{h}, \{ ~, ~\})$ equipped with a linear map $d : \mathfrak{h} \rightarrow \mathfrak{h}$ satisfying
    \begin{align*}
        d^2 = 0  \quad \text{ and } \quad d\{ x, y \} = \{ dx , y \} + \{ x, dy \}, \text{ for } x, y \in \mathfrak{h}.
    \end{align*}
    Then it is easy to verify that $(\mathfrak{h}, \{~, ~ \}_\dashv , \{ ~, ~ \}_\vdash)$ is a di-Leibniz algebra, where $\{ x, y \}_\dashv = \{ x, dy \}$ and $\{ x, y \}_\vdash = \{ dx, y \}$, for $x, y \in \mathfrak{h}$.
\end{exam}

\begin{exam}\label{bimodule-map}
Let $(\mathfrak{h}, \{ ~, ~\})$ be a Leibniz algebra and $(V, \rho^L, \rho^R)$ be a representation of it. Suppose $f : V \rightarrow \mathfrak{h}$ is a morphism between $\mathfrak{h}$-representations (from $V$ to the adjoint representation $\mathfrak{h}$), i.e. $f (\rho^L (x, v) ) = \{ x, f(v) \}$ and $f (\rho^R (v, x)) = \{ f(v), x \}$, for all $x \in \mathfrak{g}$ and $v \in V$. Then there is a di-Leibniz algebra structure on $V$ with the operations $\{ u, v \}_\dashv = \rho^R (u, f(v))$ and $\{ u, v \}_\vdash = \rho^L (f (u), v)$, for $u, v \in V$.    
\end{exam}

\begin{exam}
   Let $(\mathfrak{h}, \{ ~, ~ \})$ be a Leibniz algebra. Then the direct sum $\underbrace{\mathfrak{h} \oplus \cdots \oplus \mathfrak{h}}_{n \text{ copies}}$ carries a di-Leibniz algebra structure with the operations
    \begin{align*}
        \{   (x_1, \ldots, x_n ) , (y_1, \ldots, y_n ) \}_\dashv =~& \big(  \{ x_1, y_1 + \cdots + y_n \}, \ldots, \{ x_i, y_1 + \cdots + y_n \}  , \ldots, \{ x_n, y_1 + \cdots + y_n \} \big),\\
         \{   (x_1, \ldots, x_n ) , (y_1, \ldots, y_n ) \}_\vdash =~& \big( \{ x_1+ \cdots + x_n, y_1 \}, \ldots,   \{ x_1+ \cdots + x_n, y_i \}, \ldots,  \{ x_1+ \cdots + x_n, y_n \}  \big).
    \end{align*}
\end{exam}

Another class of examples arise from representations of a Leibniz algebra as given in the next result.

\begin{prop}
    Let $(\mathfrak{h}, \{ ~, ~ \})$ be a Leibniz algebra and $(V, \rho^L, \rho^R)$ be a representation of it. Then the triple $(\mathfrak{h} \oplus V, \{ ~, ~\}_\dashv , \{ ~, ~\}_\vdash)$ is a di-Leibniz algebra, where for $(x, u), (y, v) \in \mathfrak{h} \oplus V$,
    \begin{align*}
        \{ (x, u), (y, v) \}_\dashv := (\{ x,  y \}, \rho^R (u, y)) ~~~~ \text{ and } ~~~~ \{ (x, u), (y, v) \}_\vdash := (\{x, y \}, \rho^L (x, v)).
    \end{align*}
\end{prop}

\begin{proof}
    For any $(x, u), (y, v), (z, w) \in \mathfrak{h} \oplus V$, we have
    \begin{align}\label{dl11}
         \{ (x, u) , \{ (y, v), (z, w) \}_\dashv \}_\dashv =~& \big(  \{ x, \{ y, z \} \}, \rho^R (u, \{ y, z \})  \big) \\
         =~& \big(  \{ \{ x, y \}, z \} + \{ y, \{ x, z \} \}, \rho^R (\rho^R (u, y), z) + \rho^L (y, \rho^R (u, z))  \big) \nonumber \\
         =~& \big(  \{ \{ x, y \} , z \} , \rho^R (\rho^R (u, y), z)  \big) +  \big( \{ y, \{ x, z \} \} + \rho^L (y, \rho^R (u, z ))  \big) \nonumber \\
         =~& \{ \{ (x, u), (y, v) \}_\dashv , (z, w) \}_\dashv + \{ (y, v), \{ (x, u), (z, w) \}_\dashv \}_\vdash. \nonumber
    \end{align}
    Thus, the identity (\ref{dl1}) follows. We also have
    \begin{align*}
        \{ (x, u), \{ (y, v), (z, w) \}_\vdash \}_\dashv =~& \big( \{ x, \{ y, z \} \}, \rho^R (u, \{ y, z \} )  \big) \\
        =~& \{ \{ (x, u), (y, v) \}_\dashv , (z, w) \}_\dashv + \{ (y, v) , \{ (x, u), (z, w) \}_\dashv \}_\vdash \quad  (\text{by } (\ref{dl11}))
    \end{align*}
    which verifies the identity (\ref{dl2}). Similarly, 
    \begin{align*}
        \{ (x, u), \{ (y, v), (z, w) \}_\dashv \}_\vdash 
        =~& \big(  \{ x, \{ y, z \} \}, \rho^L (x, \rho^R (v, z) )  \big) \\
        =~& \big(  \{ \{ x, y \}, z \} + \{ y, \{ x, z \} \}, \rho^R (\rho^L (x, v), z) + \rho^R (v, \{ x, z \}) \big) \\
        =~& \big(  \{ \{ x, y \}, z \} , \rho^R (\rho^L (x, v), z) \big) + \big(  \{ y, \{ x, z \} \}, \rho^R (v, \{ x, z \} ) \big) \\
        =~& \{ \{ (x, u), (y, v) \}_\vdash , (z, w) \}_\dashv + \{ (y, v), \{ (x, u), (z, w) \}_\dashv \}_\dashv.
    \end{align*}
    Hence the identity (\ref{dl3}) also follows. Finally, 
    \begin{align*}
        &\{ (x, u), \{ (y, v), (z, w) \}_\vdash \}_\vdash \\
        &= \big(  \{ x, \{ y, z \} \}, \rho^L (x, \rho^L (y, w)) \big) \\
        &= \big( \{ \{ x, y \}, z \} + \{ y, \{ x, z \} \}, \rho^L (\{ x, y \}, w) + \rho^L (y, \rho^L (x, w))   \big) \\
        &= \big(  \{ \{ x, y \}, z \}, \rho^L (\{ x, y \}, w) \big) + \big( \{ y, \{x, z \} \}, \rho^L (y, \rho^L (x, w) )    \big) \\
        &=\begin{cases}
            = \{ \{ (x, u), (y, v) \}_\dashv , (z, w) \}_\vdash + \{ (y, v), \{ (x, u), (z, w) \}_\vdash \}_\vdash ~~~ (\text{which verifies } (\ref{dl4})), \\
            = \{ \{ (x, u), (y, v) \}_\vdash , (z, w) \}_\vdash + \{ (y, v), \{ (x, u), (z, w) \}_\vdash \}_\vdash ~~~ (\text{which verifies } (\ref{dl5})).
        \end{cases}
    \end{align*}
    This completes the proof.
\end{proof}

The di-Leibniz algebra $(\mathfrak{h} \oplus V, \{ ~, ~\}_\dashv, \{ ~, ~\}_\vdash)$ constructed in the above proposition is called the hemi-semidirect product and it is denoted by $\mathfrak{h} \oplus_\mathrm{Hemi} V$.

\begin{defn}
  Let $(\mathfrak{h}, \{ ~, ~ \})$ be a Leibniz algebra. A linear map $P: \mathfrak{h} \rightarrow \mathfrak{h}$ is said to be 
  \begin{itemize}
      \item[(i)]  a {\bf left averaging operator} on $(\mathfrak{h}, \{ ~, ~ \})$ if it satisfies
  \begin{align*}
     \{ P (x) , P (y) \} = P \{ P(x), y \}, \text{ for } x, y \in \mathfrak{h}.
  \end{align*}
  \item[(ii)] an {\bf averaging operator} on $(\mathfrak{h}, \{ ~, ~ \})$ if it satisfies
  \begin{align*}
     \{ P (x) , P (y) \} = P \{ P(x), y \} = P \{ x, P(y)\}, \text{ for } x, y \in \mathfrak{h}.
  \end{align*}
  \end{itemize}

  An {\bf averaging Leibniz algebra} is a pair $(\mathfrak{h}, P)$ consisting of a Leibniz algebra $\mathfrak{h}$ endowed with an averaging operator $P$ on it.
  \end{defn}

   Let $(\mathfrak{h}, \{ ~, ~ \})$ be a Leibniz algebra and $(V, \rho^L, \rho^R)$ be a representation of it. A {\bf relative averaging operator} on $\mathfrak{h}$ with respect to the representation $V$ is a linear map $P : V \rightarrow \mathfrak{h}$ that satisfies
  \begin{align*}
      \{ P(u), P(v) \} = P (\rho^L (P (u), v)) = P (\rho^R (u, P(v))), \text{ for } u, v \in V.
  \end{align*}
  It follows that an averaging operator on the Leibniz algebra $\mathfrak{h}$ is nothing but a relative averaging operator with respect to the adjoint representation.

\begin{exam}
    Let $(\mathfrak{h}, \{ ~, ~ \})$ be a Leibniz algebra. Then the identity map $\mathrm{Id}: \mathfrak{h} \rightarrow \mathfrak{h}$ is an averaging operator on $\mathfrak{h}$.

    More generally, let $(\mathfrak{h}, \{ ~, ~ \})$ be a Leibniz algebra and $(V, \rho^L, \rho^R)$ be a representation of it. Suppose $f: V \rightarrow \mathfrak{h}$ is a morphism of $\mathfrak{h}$-representations from $V$ to the adjoint representation (see Example \ref{bimodule-map}). Then $f$ is a relative averaging operator.
\end{exam}

\begin{exam}
    Let $(\mathfrak{h}, d)$ be a differential Leibniz algebra. Then it is easy to see that $d: \mathfrak{h} \rightarrow \mathfrak{h}$ is an averaging operator on the Leibniz algebra $\mathfrak{h}.$
\end{exam}

\begin{exam}
    Let $(\mathfrak{h}, \{ ~, ~ \})$ be a Leibniz algebra. Then the space $\underbrace{\mathfrak{h} \oplus \cdots \oplus \mathfrak{h}}_{n \text{ copies}}$ can be given a representation of the Leibniz algebra $\mathfrak{h}$ with the left and right $\mathfrak{h}$-actions given by
    \begin{align*}
        \rho^L (x, (x_1, \ldots, x_n )) := \big( \{ x, x_1 \}, \ldots, \{ x, x_n \}   \big) ~~ \text{ and } ~~ \rho^R ((x_1, \ldots, x_n ), x) := \big( \{ x_1, x \}, \ldots, \{ x_n , x \} \big),
    \end{align*}
    for $x \in \mathfrak{h}$ and $(x_1, \ldots, x_n) \in \mathfrak{h} \oplus \cdots \oplus \mathfrak{h}$. With this notation, the map 
    \begin{align*}
        P : \mathfrak{h} \oplus \cdots \oplus \mathfrak{h} \rightarrow \mathfrak{h}, ~ (x_1, \ldots, x_n) \mapsto x_1 + \cdots + x_n
    \end{align*}
    is a relative averaging operator. Moreover, for each $1 \leq i \leq n$, the $i$-th projection map $\mathrm{pr}_i : \mathfrak{h} \oplus \cdots \oplus \mathfrak{h} \rightarrow \mathfrak{h},$ $ (x_1, \ldots, x_n) \mapsto x_i$ is also a relative averaging operator.
\end{exam}

In the next, we give a characterization of a (relative) averaging operator in terms of its graph. More precisely, we have the following result.

\begin{prop}
    Let $(\mathfrak{h}, \{ ~, ~ \})$ be a Leibniz algebra and $(V, \rho^L, \rho^R)$ be a representation. A linear map $P: V \rightarrow \mathfrak{h}$ is a relative averaging operator if and only if the graph $\mathrm{Gr} (P) = \{ (P(v), v) |~ v \in V \}$ is a sub di-Leibniz algebra of the hemi-semidirect product $\mathfrak{h} \oplus_\mathrm{Hemi} V$.
\end{prop}

\begin{proof}
    Let $(P(u), u)$ and $(P(v), v)$ be two arbitrary elements in $\mathrm{Gr} (P)$. Then we have
    \begin{align*}
        \{ (P(u), u), (P(v), v) \}_\dashv =~& \big(\{ P(u), P(v) \}, \rho^R (u, P(v)) \big), \\  \{ (P(u), u), (P(v), v) \}_\vdash =~& \big(\{ P(u), P(v) \}, \rho^L (P(u), v) \big).
    \end{align*}
    The above two elements are in $\mathrm{Gr} (P)$ if and only if $\{ P(u), P(v) \} = P \big( \rho^R (u, P(v)) \big)$ and $\{ P(u), P(v) \}= P \big(  \rho^L (P(u), v) \big)$. This proves the desired result.
\end{proof}

Since $\mathrm{Gr}(P)$ is linearly isomorphic to $V$, the above proposition yields the following result.

\begin{prop}
   Let $(\mathfrak{h}, \{ ~, ~ \})$ be a Leibniz algebra, $(V, \rho^L, \rho^R)$ be a representation and $P : V \rightarrow \mathfrak{h}$ be a relative averaging operator. Then $V$ inherits a di-Leibniz algebra structure with the operations 
   \begin{align*}
   \{ u, v \}_\dashv^P := \rho^R (u, P(v)) ~~~ \text{ and } ~~~ \{ u, v \}_\vdash^P := \rho^L (P(u), v), \text{ for } u, v \in V.
   \end{align*}
\end{prop}

The above result says that a relative averaging operator on a Leibniz algebra with respect to a representation induces a di-Leibniz algebra structure on the underlying representation space. (In particular, a left averaging operator $P : \mathfrak{h} \rightarrow \mathfrak{h}$ on a Leibniz algebra $(\mathfrak{h}, \{ ~, ~\})$ induces a new Leibniz algebra $(\mathfrak{h}, \{ ~, ~ \}_P)$, where $\{ x, y \}_P = \{ P(x), y \}$ for any $x, y \in \mathfrak{h}$.) The next result gives the converse.

\begin{thm}
Every di-Leibniz algebra is induced by a relative averaging operator on a Leibniz algebra with respect to a representation.
\end{thm}

\begin{proof}
    Let  $(\mathfrak{d}, \{ ~, ~ \}_\dashv , \{ ~, ~ \}_{\vdash})$ be a di-Leibniz algebra. Consider the space $\mathfrak{d}_\mathrm{Leib}$ which is the quotient of $\mathfrak{d}$ by the subspace generated by elements of the form $\{ x, y \}_\dashv - \{ x, y \}_\vdash$, for $x, y \in \mathfrak{d}$. Then $\mathfrak{d}_\mathrm{Leib}$ carries a Leibniz algebra structure with the bracket
    \begin{align*}
        \{ \langle x \rangle, \langle y \rangle \} := \langle \{ x, y \}_\dashv \rangle = \langle \{ x, y \}_\vdash \rangle, \text{ for } \langle x \rangle, \langle y \rangle \in \mathfrak{d}_\mathrm{Leib}.
    \end{align*}
    Here $\langle x \rangle$ denotes the class of an element $x \in \mathfrak{d}$.
    We define two bilinear maps $\rho^L : \mathfrak{d}_\mathrm{Leib} \times \mathfrak{d} \rightarrow \mathfrak{d}$ and $\rho^R : \mathfrak{d} \times \mathfrak{d}_\mathrm{Leib} \rightarrow \mathfrak{d}$ by
    \begin{align*}
        \rho^L ( \langle x \rangle , y ) := \{ x, y \}_\vdash ~~~ \text{ and } ~~~ \rho^R (y, \langle x \rangle) = \{ y, x \}_\dashv, \text{ for } \langle x \rangle \in \mathfrak{d}_\mathrm{Leib}, y \in \mathfrak{d}.
    \end{align*}
    It is easy to verify that the maps $\rho^L, \rho^R$ define a representation of the Leibniz algebra $\mathfrak{d}_\mathrm{Leib}$ on the vector space $\mathfrak{d}$. Moreover, the quotient map $q : \mathfrak{d} \rightarrow \mathfrak{d}_\mathrm{Leib},~ x \mapsto \langle x \rangle$ is a relative averaging operator as 
    \begin{align*}
        \{ q (x) , q(y) \} = \{ \langle x \rangle, \langle y \rangle \} = \begin{cases}
            = \langle \{ x, y \}_\vdash  \rangle = \langle  \rho^L (\langle x \rangle, y) \rangle = q \big( \rho^L (q(x), y) \big), \\
            = \langle \{ x, y \}_\dashv \rangle = \langle  \rho^R (x, \langle y \rangle)  \rangle = q \big( \rho^R (x, q (y)) \big),
        \end{cases}
    \end{align*}
    for all $x, y \in \mathfrak{d}$. Let $(\mathfrak{d}, \{ ~, ~ \}_\dashv^q, \{ ~, ~ \}_\vdash^q)$ be the di-Leibniz algebra structure on $\mathfrak{d}$ induced by the relative averaging operator $q$. Then we have
    \begin{align*}
        \{ x, y \}_\dashv^q = \rho^R (x, \langle y \rangle) = \{ x, y \}_\dashv \quad \text{ and } \quad \{x, y \}_\vdash^q = \rho^L ( \langle x \rangle, y) = \{ x, y \}_\vdash, \text{ for } x, y \in \mathfrak{d}.
    \end{align*}
    This shows that the induced di-Leibniz algebra $(\mathfrak{d}, \{ ~, ~ \}_\dashv^q, \{ ~, ~ \}_\vdash^q)$ coincides with the given one.
\end{proof}

\begin{remark}
It is important to remark that an arbitrary di-Leibniz algebra may not induced from an averaging Leibniz algebra. However, any di-Leibniz algebra can be embedded into an averaging Leibniz algebra. More precisely, let $(\mathfrak{d}, \{ ~, ~ \}_\dashv , \{ ~, ~ \}_{\vdash})$ be a di-Leibniz algebra. Then the direct sum $\mathfrak{d}_\mathrm{Leib} \oplus \mathfrak{d}$ inherits a Leibniz algebra structure with the bracket
\begin{align*}
    \{ (\langle x \rangle, y) ,  (\langle x' \rangle, y') \}_\ltimes := \big( \{ \langle x \rangle, \langle x' \rangle \}, \rho^L (\langle x \rangle, y') + \rho^R (y, \langle x' \rangle)  \big),
\end{align*}
for $(\langle x \rangle, y), (\langle x' \rangle, y') \in \mathfrak{d}_\mathrm{Leib} \oplus \mathfrak{d}$. Moreover, the map $P : \mathfrak{d}_\mathrm{Leib} \oplus \mathfrak{d} \rightarrow \mathfrak{d}_\mathrm{Leib} \oplus \mathfrak{d}$, $P (\langle x \rangle, y) = (\langle y \rangle, 0)$ is an averaging operator. Hence $( \mathfrak{d}_\mathrm{Leib} \oplus \mathfrak{d}, P )$ is an averaging Leibniz algebra. Then the inclusion map $i : \mathfrak{d} \rightarrow \mathfrak{d}_\mathrm{Leib} \oplus \mathfrak{d}$, $i (x) = (0, x)$ is an embedding of the di-Leibniz algebra $(\mathfrak{d}, \{ ~, ~ \}_\dashv , \{ ~, ~ \}_{\vdash})$ into the averaging Leibniz algebra $( \mathfrak{d}_\mathrm{Leib} \oplus \mathfrak{d}, P )$.
\end{remark}

\section{Some Lie functors}\label{sec7}

In Section \ref{sec2}, we have shown that the differentiation of a smooth pointed averaging operator on a Lie group gives rise to an averaging operator on the corresponding Lie algebra. In this section, we continue similar studies on Lie groups and Lie racks. Among others, we show that the differentiation of a smooth pointed averaging operator on a pointed Lie rack gives rise to a left averaging operator on the corresponding Leibniz algebra.

We begin with Ad-invariant maps on groups. Let $G$ be an abstract group. A map $\mathcal{C} : G \rightarrow G$ is said to be {\bf Ad-invariant map} on $G$ if it satisfies
\begin{align}\label{ad-cond}
    \mathcal{C} (ghg^{-1}) = g \mathcal{C}(h) g^{-1}, \text{ for all } g, h \in G.
\end{align}
The condition (\ref{ad-cond}) can be equivalently expressed as $\mathcal{C} \circ \mathrm{Ad}_g = \mathrm{Ad}_g \circ \mathcal{C}$, for all $g \in G$, where $\mathrm{Ad}_g \in \mathrm{Aut}(G)$ is defined by $\mathrm{Ad}_g (h) = ghg^{-1}$. The map $\mathcal{C}$ is said to be pointed if $\mathcal{C}(e) = e$.

Let $(\mathfrak{g}, [~, ~ ])$ be a Lie algebra. A linear map $C: \mathfrak{g} \rightarrow \mathfrak{g}$ is said to be {\bf ad-invariant} if 
\begin{align}\label{ad-inv-cond}
C[x, y] = [x, C(y)], \text{ for all } x, y \in \mathfrak{g}.
\end{align}
In other words, $C \circ \mathrm{ad}_x = \mathrm{ad}_x \circ C$, for all $x \in \mathfrak{g}$, where $\mathrm{ad}_x : \mathfrak{g} \rightarrow \mathfrak{g}$ is the map given by $\mathrm{ad}_x (y) = [x, y].$
Note that the ad-invariant condition (\ref{ad-inv-cond}) is equivalent to $C[x, y] = [C(x) , y]$, for all $x, y \in \mathfrak{g}$.

It is important to remark that Ad-invariant maps or ad-invariant maps are inverses (when invertible) to averaging operators on groups and Lie algebras, respectively. Hence one may expect that Ad-invariant maps on Lie groups and ad-invariant maps on corresponding Lie algebras are related by differentiation. Explicitly, we have the following result.

\begin{thm}
    Let $G$ be a Lie group and $\mathcal{C}: G \rightarrow G$ be a smooth pointed $\mathrm{Ad}$-invariant map on $G$. If $\mathfrak{g}$ is the Lie algebra of the Lie group $G$, then $C = \mathcal{C}_{* e} : \mathfrak{g} \rightarrow \mathfrak{g}$ is an $\mathrm{ad}$-invariant map on $\mathfrak{g}$.
\end{thm}

\begin{proof}
    For any $x, y \in \mathfrak{g}$, we have
    \begin{align*}
        C[x,y] =~& \frac{d^2}{dt ds} \bigg|_{t,s = 0} \mathcal{C} \big( \mathrm{exp}^{tx} \mathrm{exp}^{sy} \mathrm{exp}^{- tx}  \big) \\
        =~& \frac{d^2}{dt ds} \bigg|_{t,s = 0} \mathrm{exp}^{tx} \mathcal{C} (\mathrm{exp}^{sy}) \mathrm{exp}^{- tx}   \\
        =~&  \frac{d^2}{dt ds} \bigg|_{t,s = 0} \mathrm{Ad}_{ \mathrm{exp}^{tx}} \mathcal{C} (\mathrm{exp}^{sy}) = [x, C(y)].
    \end{align*}
    This completes the proof.
\end{proof}

Next, we recall pointed Lie racks (see, for instance, \cite{kinyon}) and the construction of their Leibniz algebras. This construction is a generalization of the construction of the Lie algebra of a Lie group.

\begin{defn}
    A {\bf Lie rack} is a rack $(Q, \diamond)$ in which $Q$ is a smooth manifold and the rack operation $\diamond: Q \times Q \rightarrow Q$ is a smooth map. A Lie rack $(Q, \diamond)$ is said to be {\bf pointed} if the underlying rack is pointed.
\end{defn}

    Like racks are a generalization of the conjugation operation of groups, Lie racks are a generalization of the conjugation of Lie groups. More precisely, let $G$ be a Lie group with the identity element $e$. Then the corresponding conjugation rack $(G, \diamond)$ is a pointed Lie rack with distinguished element $e$. More precisely, we have the following result which generalizes Proposition \ref{prop-des} in the smooth context.

    \begin{prop}
    Let $G$ be a Lie group (with the identity element $e$) and $\mathcal{A} : G \rightarrow G$ be a smooth pointed averaging operator, then the descendent rack $(G, \diamond_\mathcal{A})$ is a pointed Lie rack with distinguished element $e$.
\end{prop}

Given a pointed Lie rack $(Q, \diamond, e)$, it has been observed in \cite{kinyon} that the distinguished tangent space $T_e Q$ inherits a Leibniz algebra structure. For our convenience, we recall the construction. Let $(Q, \diamond, e)$ be a pointed Lie rack. For any $x \in Q$, consider the smooth (left multiplication) map $L_x^\diamond: Q \rightarrow Q$. Since $L_x^\diamond (e) = e$, by taking the differentiation of $L_x^\diamond$ at the point $e$, one obtains a linear map $T_e (L_x^\diamond) : T_e Q \rightarrow T_e Q$. Moreover, the left multiplication map $L_x^\diamond$ is invertible implies that $T_e (L_x^\diamond) \in GL (T_e Q)$, for all $x \in Q$. As a result, we get a smooth map $\Theta : Q \rightarrow GL (T_e Q)$, $x \mapsto T_e (L_x^\diamond)$. Further, we have $\Theta (e) = T_e (L_e^\diamond) = \mathrm{Id}_{T_e Q}$. By differentiating the map $\Theta$ at the point $e$, one gets the map $\mathrm{ad}: T_e Q \rightarrow gl (T_e Q)$, where we have identified the tangent space at the identity element of $GL (V)$ for a vector space $V$ with the general linear Lie algebra $gl (V)$. For any $X, Y \in T_e Q$, we define a bilinear operation
\begin{align*}
    \{ X, Y \} := \mathrm{ad}(X) Y.
\end{align*}
Note that, in terms of the left multiplications, the left-distributivity property of the rack can be written as
\begin{align}\label{left-lier}
    L_x^\diamond L_y^\diamond (z) = L^\diamond_{L_x^\diamond (y)} L_x^\diamond (z), \text{ for } x, y, z \in Q.
\end{align}
By taking the differentiation of (\ref{left-lier}) at $e$, first with respect to $z$ and then with respect to $y$, we obtain $\Theta (x) \{ Y, Z \} = \{ \Theta (x) Y, \Theta (x) Z \}$, for all $x \in Q$ and $Y, Z \in T_e Q$. Finally, by differentiating the above identity with respect to $x$ at the point $e$, we get the Leibniz identity 
\begin{align*}
    \{ X, \{ Y, Z \} \} = \{ \{ X, Y \}, Z \} + \{ Y, \{ X, Z \} \}, \text{ for } X, Y, Z \in T_e Q.
\end{align*}
Hence $(T_e Q, \{ ~, ~ \})$ becomes a Leibniz algebra, called the tangent Leibniz algebra of the pointed Lie arack $(Q, \diamond, e)$.

\medskip

Let $G$ be a Lie group and $\mathcal{A}: G \rightarrow G$ be a smooth pointed averaging operator on $G$. In Theorem \ref{firr-thm}, we have seen that the map $\mathcal{A}_{* e}: \mathfrak{g} \rightarrow \mathfrak{g}$ (the differentiation of $\mathcal{A}$ at the identity element $e$) is an averaging operator on the corresponding Lie algebra $\mathfrak{g} = T_e G$. Hence by Proposition \ref{avg-ind-leib}, the space $\mathfrak{g}$ inherits a Leibniz algebra structure with the bracket
\begin{align}\label{first-l}
    \{ X,Y \}_{\mathcal{A}_{* e}} :=~& [\mathcal{A}_{* e} (X), Y] \nonumber\\
    =~& \frac{d^2}{dt ds}|_{t, s = 0} \mathrm{exp}^{\mathcal{A}_{*e} (tX)} \mathrm{exp}^{sY} \mathrm{exp}^{- \mathcal{A}_{*e} (tX)} \nonumber \\
    =~& \frac{d^2}{dt ds}|_{t, s = 0} \mathcal{A} (\mathrm{exp}^{tX}) \mathrm{exp}^{sY} \mathcal{A} (\mathrm{exp}^{tX})^{-1},
\end{align}
for $X,Y \in \mathfrak{g}.$
On the other hand, the smooth pointed averaging operator $\mathcal{A} : G \rightarrow G$ gives rise to the pointed Lie rack $(G, \diamond_\mathcal{A}, e)$. Hence one obtains the tangent Leibniz algebra on the vector space $T_e G$. The Leibniz bracket is precisely given by
\begin{align}\label{new-l}
 \{ X, Y \}  := \mathrm{ad}(X) Y = T_e (\Theta) (X) Y = \frac{d^2}{dt dt} |_{t, s =0} \mathrm{Ad}_{  \mathcal{A} (\mathrm{exp}^{tX}) }  \mathrm{exp}^{sY} ~~ (\text{since } L_x^{\diamond_\mathcal{A}} = L^\diamond_{\mathcal{A} (x)}).
\end{align}
It follows from the expressions (\ref{first-l}) and (\ref{new-l}) that one obtains the commutative diagram 
\begin{align}\label{first-diag}
\xymatrix{
\text{Lie group} \ar[d]_{\text{Diff}} \ar[rr]^{\text{pointed avg operator}} & & \text{pointed Lie rack} \ar[d]^{\text{Diff}} \\
\text{Lie algebra} \ar[rr]_{\text{avg operator}} & & \text{Leibniz algebra.}
}
\end{align}

%\[
%\xymatrix{
%\text{Lie group} \ar[d]_{\text{Diff}} \ar[rr]^{\text{pointed avg operator}} & & \text{pointed Lie rack} \ar[d]^{\text{Diff}} && \text{Lie group} \ar[d]_{\text{Diff}} \ar[rr]^{\text{pointed avg operator}} & & \text{Lie group-rack} \ar[d]^{\text{Diff}} \\
%\text{Lie algebra} \ar[rr]_{\text{avg operator}} & & \text{Leibniz algebra} && \text{Lie algebra} \ar[rr]_{\text{avg operator}} & & \text{Lie-Leibniz algebra.}
%}
%\]

%If $\mathcal{A}$ is a smooth pointed averaging operator on a Lie group $G$ then the descendent rack $(G, \diamond_\mathcal{A})$ is also pointed.

\medskip

Next, we consider pointed Lie group-racks and show that the differentiation at the identity element gives rise to the structure of a Lie-Leibniz algebra. We begin with the algebraic notion first.

\begin{defn}
    A {\bf Lie-Leibniz algebra} is a triple $(\mathfrak{g} , [~,~], \{ ~, ~ \})$ in which $(\mathfrak{g} , [~,~])$ is a Lie algebra and $(\mathfrak{g}, \{ ~, ~ \})$ is a Leibniz algebra satisfying the additional compatibility
    \begin{align}\label{ll-iden}
        \{ x, [y, z ] \} = [ \{ x, y \}, z] + [y, \{ x, z \}], \text{ for } x, y, z \in \mathfrak{g}.
    \end{align}
\end{defn}

Note that the compatibility condition (\ref{ll-iden}) says that all the left translations $\{ x, - \}$ by the Leibniz bracket is a derivation for the Lie bracket on $\mathfrak{g}$.

Any Lie algebra $(\mathfrak{g} , [~,~])$ can be regarded as a Lie-Leibniz algebra by taking $\{ ~, ~ \} = [~,~]$. In the following, we give two important classes of examples, one arising from representations and another arising from averaging operators on Lie algebras.

\begin{exam}
    Let $(\mathfrak{g}, [~, ~])$ be a Lie algebra and $(V, \rho)$ be a representation (i.e. $\rho : \mathfrak{g} \times V \rightarrow V$ is a bilinear map satisfying $\rho ([x, y], v) = \rho (x , \rho (y, v)) - \rho (y, \rho (x, v))$, for $x, y \in \mathfrak{g}$ and $v \in V$). Then the direct sum $\mathfrak{g} \oplus V$ inherits the semidirect product Lie algebra and the hemi-semidirect product Leibniz algebra structure with the brackets
    \begin{align*}
        &\text{semidirect product Lie} \qquad \qquad \qquad [(x, u), (y, v)]_\ltimes := ([x, y], \rho (x, v) - \rho (y, u)),\\
        &\text{hemi-semidirect product Leibniz} \qquad \{ (x, u), (y, v)\} := ([x, y], \rho (x, v) ),
    \end{align*}
    for $(x, u), (y, v) \in \mathfrak{g} \oplus V$. Then it is easy to see that $(\mathfrak{g} \oplus V, [~,~]_\ltimes , \{ ~, ~ \})$ is a Lie-Leibniz algebra.
\end{exam}

\begin{prop}\label{avg-ind-leib}
    Let $(\mathfrak{g}, [~,~])$ be a Lie algebra and $\mathcal{A} : \mathfrak{g} \rightarrow \mathfrak{g}$ be an averaging operator on $\mathfrak{g}$. Then $(\mathfrak{g}, \{ ~,~ \}_\mathcal{A})$ is a Leibniz algebra (called the descendent Leibniz algebra), where $\{ x, y \}_\mathcal{A}:= [ \mathcal{A} (x), y]$, for $x, y \in \mathfrak{g}$. Moreover, the triple $(\mathfrak{g} , [~,~], \{ ~,~ \}_\mathcal{A})$ is a Lie-Leibniz algebra.
\end{prop}

\begin{proof}
    For any $x, y, z \in \mathfrak{g}$, we have
    \begin{align*}
        \{ x, \{ y, z \}_\mathcal{A} \}_\mathcal{A} = [\mathcal{A}(x), [\mathcal{A} (y), z]] =~& [[ \mathcal{A}(x), \mathcal{A}(y)], z ] + [ \mathcal{A}(y), [\mathcal{A}(x), z]] \\
        =~& \{ \{ x, y \}_\mathcal{A} , z \}_\mathcal{A} + \{ y, \{ x, z \}_\mathcal{A} \}_\mathcal{A}.
    \end{align*}
    Hence $(\mathfrak{g}, \{ ~,~ \}_\mathcal{A})$ is a Leibniz algebra. Moreover, we have
    \begin{align*}
        \{ x, [y, z ] \}_\mathcal{A} = [\mathcal{A}(x) , [y, z] ] =~& [[\mathcal{A} (x), y] , z ] + [y, [\mathcal{A}(x), z ]] \\
        =~& [\{ x, y \}_\mathcal{A}, z] + [y, \{ x, z \}_\mathcal{A}].
    \end{align*}
    This shows that $(\mathfrak{g} , [~,~], \{ ~,~ \}_\mathcal{A})$ is a Lie-Leibniz algebra.
\end{proof}

The next result gives the partial converse of the above proposition.

\begin{thm}
    Every Lie-Leibniz algebra can be embedded into an averaging Lie algebra.
\end{thm}

\begin{proof}
    Let $(\mathfrak{g}, [~,~], \{ ~, ~ \})$ be a Lie-Leibniz algebra. Consider the Lie algebra $\mathfrak{g}^{ \{ ~, ~ \} }_\mathrm{Lie}$ associated to the Leibniz algebra $(\mathfrak{g}, \{ ~, ~ \})$. That is, $\mathfrak{g}^{ \{ ~, ~ \} }_\mathrm{Lie} = \mathfrak{g} / \{ \mathfrak{g}, \mathfrak{g} \} $ equipped with the Lie bracket
    \begin{align*}
         \llbracket \langle x \rangle, \langle y \rangle \rrbracket := \langle \{ x, y \} \rangle, \text{ for } \langle x \rangle, \langle y \rangle \in \mathfrak{g}^{ \{ ~, ~ \} }_\mathrm{Lie}.
    \end{align*}
     The Lie algebra $( \mathfrak{g}^{ \{ ~, ~ \} }_\mathrm{Lie}, \llbracket ~, ~ \rrbracket )$ has a representation on the vector space $\mathfrak{g}$ with the action map $\rho : \mathfrak{g}^{ \{ ~, ~ \} }_\mathrm{Lie} \times \mathfrak{g} \rightarrow \mathfrak{g}$ given by $\rho (\langle x \rangle, y) := \{ x, y \}$, for $\langle x \rangle \in \mathfrak{g}^{ \{ ~, ~ \} }_\mathrm{Lie}$ and $y \in \mathfrak{g}$. It follows from (\ref{ll-iden}) that the map $\rho$ additionally satisfies
    \begin{align*}
        \rho (\langle x \rangle,  [y, z]) = [\rho (\langle x \rangle, y) , z] + [y , \rho (\langle x \rangle, z)], \text{ for } y, z \in \mathfrak{g}.
    \end{align*}
    In other words, the Lie algebra $\mathfrak{g}^{ \{ ~, ~ \} }_\mathrm{Lie}$ acts on the Lie algebra $(\mathfrak{g}, [~,~ ])$ by derivations. As a consequence, the space $\mathfrak{g}^{ \{ ~, ~ \} }_\mathrm{Lie} \oplus \mathfrak{g}$ inherits a Lie algebra structure with the Lie bracket
    \begin{align*}
       \llbracket (\langle x \rangle, y), (\langle x' \rangle, y') \rrbracket_\ltimes := \big( \llbracket \langle x \rangle, \langle x' \rangle \rrbracket, \rho (\langle x \rangle, y') - \rho (\langle x' \rangle,  y) + [y, y'] \big),
    \end{align*}
    for $(\langle x \rangle, y), (\langle x' \rangle, y') \in \mathfrak{g}^{ \{ ~, ~ \} }_\mathrm{Lie} \oplus \mathfrak{g}.$ Moreover, the map $P : \mathfrak{g}^{ \{ ~, ~ \} }_\mathrm{Lie} \oplus \mathfrak{g} \rightarrow \mathfrak{g}^{ \{ ~, ~ \} }_\mathrm{Lie} \oplus \mathfrak{g}$ defined by $P (\langle x \rangle, y) := (\langle y \rangle, 0)$ is an averaging operator on the Lie algebra $\big( \mathfrak{g}^{ \{ ~, ~ \} }_\mathrm{Lie} \oplus \mathfrak{g}, \llbracket ~, ~ \rrbracket_\ltimes \big)$. That is, $\big( \mathfrak{g}^{ \{ ~, ~ \} }_\mathrm{Lie} \oplus \mathfrak{g}, P  \big)$ is an averaging Lie algebra. Finally, it is easy to verify that the inclusion map $i : \mathfrak{g} \hookrightarrow  \mathfrak{g}^{ \{ ~, ~ \} }_\mathrm{Lie} \oplus \mathfrak{g}, ~x \mapsto (0, x)$ is a morphism of Lie-Leibniz algebras, where the Lie-Leibniz algebra structure on $ \mathfrak{g}^{ \{ ~, ~ \} }_\mathrm{Lie} \oplus \mathfrak{g}$ is induced by the averaging operator $P$. Hence the result follows.
\end{proof}

\begin{defn}
    A {\bf pointed Lie group-rack} is a triple $(G, \cdot, \smallblackdiamond)$ in which $(G, \cdot)$ is a Lie group with the identity element $e$ and $(G, \smallblackdiamond, e)$ is a pointed Lie rack satisfying $x \smallblackdiamond (yz) = (x \smallblackdiamond y) (x \smallblackdiamond z)$, for $x, y, z \in G$.
\end{defn}

It follows that a pointed Lie group-rack is a pointed group-rack in which the underlying set is a smooth manifold such that the group operations (multiplication and inverse) and the rack operation are all smooth.

\begin{exam}
    Let $G, H$ be two Lie groups and $\Phi : G \rightarrow \mathrm{Aut}(H)$ be a Lie group homomorphism. Then the cartesian product $G \times H$ carries the semidirect product Lie group structure with the multiplication 
    \begin{align*}
        (x,h) \cdot (y, k) = (xy, h \Phi_x k), \text{ for } (x, h), (y, k) \in G \times H.
    \end{align*}
    On the cartesian product $G \times H$, there is also a Lie rack structure with the operation
    \begin{align*}
        (x,h) \diamond (y,k) = (xyx^{-1}, \Phi_x k), \text{ for } (x, h), (y, k) \in G \times H.
    \end{align*}
    Then it is easy to verify that $(G \times H, \cdot , \diamond)$ is a pointed Lie group-rack.
\end{exam}

\begin{exam}
    Let $(G, \cdot)$ be a Lie group and $\mathcal{A} : G \rightarrow G$ be a smooth pointed averaging operator. Then the triple $(G, \cdot , \diamond_\mathcal{A})$ is a pointed Lie group-rack, where $x \diamond_\mathcal{A} y := \mathcal{A} (x) y \mathcal{A} (x)^{-1}$ is the descendent rack structure.
\end{exam}

Let $(G, \cdot, \smallblackdiamond)$ be a pointed Lie group-rack. Since $(G, \cdot)$ is a Lie group, it yields a Lie algebra structure on the distinguished tangent space $ T_e G$. On the other hand, $(G, \smallblackdiamond, e)$ is a pointed Lie rack implies that $T_e G$ inherits a Leibniz algebra structure with the bracket $\{ X, Y \} := \mathrm{ad}(X) Y$, for $X, Y \in T_e G$. Note that the compatibility condition $x \smallblackdiamond (yz) = (x \smallblackdiamond y) (x \smallblackdiamond z)$  simply means that all left multiplications $L_x^\smallblackdiamond$'s are group automorphisms. Hence for each $x \in G$, the map $T_e (L_x^\smallblackdiamond) : T_eG \rightarrow T_eG$ is a Lie algebra automorphism. Thus, in this case, we obtain a smooth map $\Theta : G \rightarrow \mathrm{Aut} (T_eG), x \mapsto T_e (L_x^\smallblackdiamond)$. By differentiating at the identity element $e$, we precisely get the map $\mathrm{ad} : T_e G \rightarrow \mathrm{Der} (T_e G)$. Hence for any $X, Y, Z \in T_e G$, we have
\begin{align*}
    \{ X, [Y, Z] \} = \mathrm{ad} (X)[Y, Z] 
    =~& [\mathrm{ad} (X) Y, Z] + [Y, \mathrm{ad} (X) Z] \\
   =~& [\{ X, Y \}, Z] + [Y, \{ X, Z \}].
\end{align*}
Therefore, we obtain a Lie-Leibniz algebra $(T_e G, [~,~], \{ ~, ~ \})$, called the tangent Lie-Leibniz algebra of the pointed Lie group-rack $(G, \cdot, \smallblackdiamond)$.

Let $(G, \cdot)$ be a Lie group and $\mathcal{A} : G \rightarrow G$ be a smooth pointed averaging operator. Then the Lie group $(G, \cdot)$ and the descendent pointed Lie rack $(G, \diamond_\mathcal{A}, e)$ together form a pointed Lie group-rack. On the infinitesimal level, the map $\mathcal{A}$ induces an averaging operator $\mathcal{A}_{* e} : \mathfrak{g} \rightarrow \mathfrak{g}$ on the corresponding Lie algebra $(\mathfrak{g} = T_e G, [~,~])$. The Lie algebra $(\mathfrak{g}, [~,~])$ and the descendent Leibniz algebra $(\mathfrak{g}, \{ ~, ~ \}_{\mathcal{A}_{* e}})$ forms a Lie-Leibniz algebra. Since the differentiation of the pointed Lie rack $(G, \diamond_\mathcal{A}, e)$ is the Leibniz algebra $(\mathfrak{g}, \{ ~, ~ \}_{\mathcal{A}_{* e}})$ and we have the commutative diagram (\ref{first-diag}), we get another commutative diagram 
\[
\xymatrix{
 \text{Lie group} \ar[d]_{\text{Diff}} \ar[rr]^{\text{pointed avg operator}} & & \text{Lie group-rack} \ar[d]^{\text{Diff}} \\
 \text{Lie algebra} \ar[rr]_{\text{avg operator}} & & \text{Lie-Leibniz algebra.}
}
\]

\medskip

Finally, we consider smooth pointed averaging operators on a pointed Lie rack and show that their differentiation gives rise to averaging operators on the corresponding Leibniz algebra.

\begin{thm}
    Let $(Q, \diamond, e)$ be a pointed Lie rack with the corresponding Leibniz algebra $(T_e Q, \{ ~, ~\})$. Let $\mathcal{A} : Q \rightarrow Q$ be a smooth pointed averaging operator on $(Q, \diamond, e)$.

    (i) Then the differentiation $\mathcal{A}_{*e} : T_e Q \rightarrow T_e Q$ is a left averaging operator on the Leibniz algebra $(T_e Q, \{ ~, ~\})$.

    (ii) The Leibniz algebra corresponding to the descendent pointed Lie rack $(Q, \diamond_\mathcal{A}, e)$ is precisely given by $(T_e Q, \{ ~, ~ \}_{\mathcal{A}_{*e}}).$ 
\end{thm}

\begin{proof}
    (i) Since $\mathcal{A} : Q \rightarrow Q$ is a smooth map with $\mathcal{A} (e) =e$, the differentiation $\mathcal{A}_{* e}$ is a linear map from $T_e Q$ to itself. On the other hand, $\mathcal{A}$ is an averaging operaor implies that $L_{\mathcal{A} (x)}^\diamond \circ \mathcal{A} = \mathcal{A} \circ L_{\mathcal{A} (x)}^\diamond$, for all $x \in Q$. By taking the differentiation at the distinguished element $e$, one obtains 
    \begin{align}\label{leib-r-a}
        T_e (L^\diamond_{\mathcal{A} (x)}) \circ \mathcal{A}_{* e} = \mathcal{A}_{* e} \circ T_e (L^\diamond_{\mathcal{A} (x)}), \text{ for all } x \in Q.
    \end{align}
Next, consider the smooth map $\Theta : Q \rightarrow GL (T_e Q)$ given by $\Theta (x) = T_e (L_x^\diamond)$, for $x \in Q$. Then it follows from (\ref{leib-r-a}) that $\Theta (\mathcal{A} (x) ) \circ \mathcal{A}_{* e} = \mathcal{A}_{* e} \circ \Theta (\mathcal{A} (x) ) $, for all $x \in Q$. Hence for any $Y \in T_e Q$, we have
\begin{align*}
    \Theta (\mathcal{A}(x)) (\mathcal{A}_{*e} (Y)) = \mathcal{A}_{* e} \big(  \Theta ( \mathcal{A} (x)) (Y)  \big).
\end{align*}
Finally, by differentiating the above identity with respect to $x$ at the distinguished point $e$, one get $\{ \mathcal{A}_{* e} (X) ,  \mathcal{A}_{* e} (Y) \} =  \mathcal{A}_{* e} \{  \mathcal{A}_{* e} (X), Y \}$, for any $X, Y \in T_e Q$. This proves that $\mathcal{A}_{* e} : T_e Q \rightarrow T_e Q$ is a left averaging operator on the Leibniz algebra $(T_e Q, \{ ~, ~ \}).$

(ii) Let $\Theta_\mathcal{A} : Q \rightarrow GL (T_e Q)$ be the smooth map given by $\Theta_\mathcal{A} (x) := T_e (L_x^{\diamond_\mathcal{A}})$, for $x \in Q$. Since $L_x^{\diamond_\mathcal{A}} = L_{\mathcal{A} (x)}^\diamond$, we have $\Theta_\mathcal{A} (x) = T_e ( L^\diamond_{\mathcal{A} (x)}) = \Theta (\mathcal{A}(x))$, for all $x \in Q$. Hence we get $\Theta_\mathcal{A} = \Theta \circ \mathcal{A}$. If $(T_e Q, \{ ~, ~ \}_{\diamond_\mathcal{A}})$ is the Leibniz algebra corresponding to the pointed Lie rack $(Q, \diamond_\mathcal{A}, e)$ then for any $X, Y \in T_e Q$, we have
\begin{align*}
    \{ X, Y \}_{\diamond_\mathcal{A}} := T_e (\Theta_\mathcal{A}) (X) Y  
    =~& T_e (\Theta) (\mathcal{A}_{* e} (X)) Y  \quad (\text{as } \Theta_\mathcal{A} = \Theta \circ \mathcal{A}) \\
    =~& \mathrm{ad} ( \mathcal{A}_{* e} (X) ) Y \quad (\because \mathrm{ad} = T_e (\Theta) ) \\
    =~& \{ \mathcal{A}_{* e} (X), Y \} = \{ X, Y \}_{\mathcal{A}_{* e }}.
\end{align*}
This proves the desired result.  
\end{proof}

The last part of the above theorem yields the first commutative diagram of (\ref{two-comm}). By taking the higher powers of the averaging operator $\mathcal{A}$, one obtain the second commutative diagram of (\ref{two-comm}),

\begin{align}\label{two-comm}
\xymatrix{
 \substack{\text{pointed Lie rack} \\ (Q, \diamond, e)} \ar[d]_{\text{Diff}} \ar[rr]^{ \substack{\text{pointed avg operator} \\ \mathcal{A}}} & & \substack{\text{pointed Lie rack} \\ (Q, \diamond_\mathcal{A}, e)}   \ar[d]^{\text{Diff}} &  \substack{\text{pointed Lie rack} \\ (Q, \diamond, e)} \ar[d]_{\text{Diff}} \ar[rr]^{ \substack{\text{pointed avg operator} \\ \mathcal{A}^k}} & & \substack{\text{pointed Lie rack} \\ (Q, \diamond_{\mathcal{A}^k}, e)}   \ar[d]^{\text{Diff}} \\
 \substack{ \text{Leibniz algebra} \\ (T_e Q, \{ ~, ~ \})} \ar[rr]_{ \substack{\text{left avg operator} \\ \mathcal{A}_{*e} } } & & \substack{ \text{Leibniz algebra} \\ (T_e Q, \{ ~, ~ \}_{\mathcal{A}_{*e}})} &   \substack{ \text{Leibniz algebra} \\ (T_e Q, \{ ~, ~ \})} \ar[rr]_{ \substack{\text{left avg operator} \\ {\mathcal{A}^k_{*e} }}  } & & \substack{ \text{Leibniz algebra} \\ (T_e Q, \{ ~, ~ \}_{\mathcal{A}^k_{*e}}).} 
}
\end{align}

One may also define pointed Lie rack-pairings and pointed Lie di-racks, and find their possible connections to di-Leibniz algebras. We will consider these topics in some other work.

\section{Some generalizations and future works}\label{sec8}
In this section, we will give a quick sketch of some generalizations of the topics covered in this paper and mention some future works.

Our main objective of the present paper was to study averaging operators on various algebraic structures and figure out the corresponding induced structures. Many of the structures related to racks are intimately connected to the corresponding structures related to groups. For instance, rack-pairings are related to two-sided skew braces by conjugation. Recently, \cite{bardakov} Bardakov and Gubarev introduced the notion of a skew $k$-brace in the study of Rota-Baxter operators on groups.

\begin{defn}\label{skewk} A {\bf (left) skew $k$-brace} is a tuple $(G, \cdot_0, \cdot_1, \ldots, \cdot_k)$ consisting of a nonempty set $G$ with binary operations $\cdot_0, \cdot_1, \ldots, \cdot_k : G \times G \rightarrow G$ such that

- $(G, \cdot_i)$ is a group, for each $0 \leq i \leq k$, 

- for each $1 \leq i \leq k$, the triple $(G, \cdot_{i-1}, \cdot_i)$ is a skew brace, that is,
\begin{align*}
    x \cdot_i (y \cdot_{i-1} z) = (x \cdot_i y) \cdot_{i-1} x^{-1 (\cdot_{i-1})} \cdot_{i-1} (x \cdot_i z), \text{ for } x, y, z \in G.
\end{align*}
Here $x^{-1 (\cdot_{i-1})}$ is the inverse of the element $x$ in the group $(G, \cdot_{i-1})$.
\end{defn}

\begin{remark}
One may also define the notion of a two-sided skew $k$-brace as a skew $k$-brace $(G, \cdot_0, \cdot_1, \ldots, \cdot_k)$ such that $(G, \cdot_{i-1}, \cdot_i)$ is a two-sided skew brace, for any $1 \leq i \leq k$.
\end{remark}

Generalizing Definition \ref{skewk} in the context of racks, one gets the following notion.

\begin{defn}
Let $k$ be a natural number. A {\bf rack $k$-pairing} is a tuple $(Q, \diamond_0, \diamond_1, \ldots, \diamond_k)$ consisting of a nonempty set $Q$ equipped with $(k+1)$ binary operations $\diamond_0, \diamond_1, \ldots, \diamond_k : Q \times Q \rightarrow Q$ such that the following conditions are hold:

- $(Q, \diamond_i)$ is a rack, for each $0 \leq i \leq k$,

- for all $x, y, z \in Q$ and $1 \leq i \leq k$,
\begin{align*}
    x \diamond_i (y \diamond_{i-1} z) = (x \diamond_i y) \diamond_{i-1} (x \diamond_i z).
\end{align*}
\end{defn}
 It follows from the above definition that a rack $1$-pairing is simply a rack-pairing. More generally, a rack $k$-pairing is a tuple $(Q, \diamond_0, \diamond_1, \ldots, \diamond_k)$ such that $(Q, \diamond_{i-1} , \diamond_i)$ is a rack-pairing, for each $1 \leq i \leq k$. The following result is a generalization of Proposition \ref{prop-similar}.

 \begin{prop}
    Let $(G, \cdot_0, \cdot_1, \ldots, \cdot_k)$ be a two-sided skew $k$-brace. Then the tuple $(G, \diamond_0, \diamond_1, \ldots, \diamond_k)$ is a rack $k$-pairing, where
    \begin{align*}
        x \diamond_i y := x \cdot_i y \cdot_i x^{-1 (\cdot_i)}, \text{ for any } x, y \in G \text{ and } 0 \leq i \leq k.
    \end{align*}
\end{prop}

In Proposition \ref{avg-rack-pair}, we have shown that any averaging operator on a rack yields the structure of a rack-pairing. In the following result, we show that an averaging operator also induces a rack $k$-pairing.

 \begin{prop}
     Let $(Q, \diamond)$ be a rack and $\mathcal{A}: Q \rightarrow Q$ be an averaging operator on $Q$. Define binary operations $\diamond_1, \ldots, \diamond_k : Q \times Q \rightarrow Q$ by
     \begin{align*}
          x \diamond_i y = \mathcal{A}^i (x) \diamond y, \text{ for all } x, y \in Q.
     \end{align*}
     Then $(Q, \diamond_0 = \diamond, \diamond_1, \ldots, \diamond_k)$ is a rack $k$-pairing.
 \end{prop}

 It is not hard to define pointed rack $k$-pairings and related structures. Is a pointed rack $k$-pairing always embedded into an averaging rack? On the other hand, one can define pointed Lie rack $k$-pairings by considering smooth structures. It is interesting to find their infinitesimal version. Another interesting question could be the use of rack $k$-pairings in knot theory.

 \medskip

 The notion of $n$-valued groups (also called multivalued groups) are generalization of usual abstract groups (see, for instance, \cite{buch}). The theory of $n$-valued groups generalizes many results from classical group theory.
 %Recently, the authors in \cite{bar} introduced the notion of $n$-valued racks and $n$-valued quandles. 
 In a subsequent paper, we aim to define Rota-Baxter operators and averaging operators on $n$-valued groups and the concept of $n$-valued skew braces. We are also looking forward to defining $n$-valued Lie algebras and finding possible connections to Lie $n$-valued groups.

 \medskip

 Another higher generalization of racks was introduced by Biyogmam \cite{biyog} who considered $n$-racks. Note that $n$-valued racks and $n$-racks are two different notions. In the former one, the defining operation takes values in the $n$ cartesian product of the underlying set, while in the latter one, the defining operation is defined on the $n$ cartesian product of the underlying set. The author in \cite{biyog} also considered Lie $n$-racks and showed that the differentiation of a pointed Lie $n$-rack yields an $n$-Leibniz algebra. In a forthcoming article, we will define the notion of an averaging operator on $n$-racks and show that such an operator induces an averaging operator on the corresponding $n$-Leibniz algebra. We also aim to formulate and study the $n$-ary generalization of the set-theoretical solutions of the Yang-Baxter equation and show that an $n$-rack naturally gives rise to a solution of such an equation.

%\[
%\xymatrix{
%\text{pointed Lie rack} \ar[d]_{\text{Diff}} \ar[rr]^{\text{pointed avg operator}} & & \text{pointed Lie di-rack} \ar[d]^{\text{Diff}} \\
%\text{Leibniz algebra} \ar[rr]_{\text{avg operator}} & & \text{di-Leibniz algebra.}
%}
%\]

\medskip
    
\noindent  {\bf Acknowledgements.} The author would like to thank the Department of Mathematics, IIT Kharagpur for providing the beautiful academic atmosphere where the research has been conducted.
%\vspace*{1cm}

\medskip

\noindent {\bf Data Availability Statement.} Data sharing does not apply to this article as no new data were created or analyzed in this study.

\end{document}